%% file: VEM_weaklyBC_rev.tex
\documentclass[11pt]{amsart}

\usepackage[colorinlistoftodos]{todonotes}

\usepackage{mathabx}
\usepackage{graphicx}
\usepackage{subfig}
\usepackage{tabularx}

\usepackage{mathrsfs}     % selects Times Roman as basic font
\usepackage{helvet}         % selects Helvetica as sans-serif font
\usepackage{courier}        % selects Courier as typewriter font
\usepackage{type1cm}      % activate if the above 3 fonts are
\usepackage{xcolor}
\usepackage{lipsum} 
\usepackage{url}

\usepackage{geometry,calc,color}
\usepackage{amsmath,amssymb}

\usepackage{enumerate}
\usepackage{arydshln}
\usepackage{siunitx}
\usepackage{booktabs}
\usepackage{multirow}
\usepackage{tikz}

\usepackage{cleveref}
\crefrangelabelformat{equation}{(#3#1#4--#5\crefstripprefix{#1}{#2}#6)}
\crefrangelabelformat{subequation}{(#3#1#4--#5\crefstripprefix{#1}{#2}#6)}

\usepackage{pgfplots}
\pgfplotsset{compat=newest}

\usepackage{textcomp}

\allowdisplaybreaks[4]

\renewcommand{\epsilon}{\varepsilon}
\renewcommand{\phi}{\varphi}
\renewcommand{\theta}{\vartheta}
\newcommand{\p}{k}

\newcommand{\face}{\textsl{f }}
\newcommand{\el}{K}

\newcommand{\Omh}{\Omega_h}

\newcommand{\n}{{\nu}}
\newcommand{\K}{K}
\newcommand{\Poly}[1]{\mathbb{P}_{#1}}

\newcommand{\Th}{\mathcal{T}_h}

\newcommand{\Pinabla}{\Pi^{\nabla}}

\newcommand{\Pinablaf}{\Pi^{\nabla,k}_\face}
\newcommand{\Bone}{{\mathbb{B}_\p}}
\newcommand{\Vfp}{V^{\face}_\p}
\newcommand{\VKp}{V^{K}_\p}
\newcommand{\tVfp}{\widetilde{V}_k^\face}

\newcommand{\PolyDom}{{\Omega}}

\newcommand{\pwPoly}{\Poly{k}(\Tess) }
\newcommand{\Svem}{S_a^\el}

\newcommand{\hel}{h_\el}

\newcommand{\hKf}{\widetilde h_\face} 

\newcommand{\lh}{\lambda_h}

\newcommand{\uh}{u_h}
\newcommand{\vh}{v_h}
\newcommand{\Vh}{V_h}

\newcommand{\norm}[1]{\|#1\|}
\newcommand{\snorm}[1]{| #1 |}

\newcommand{\Fb}{{\mathcal{F}^\partial }}

\newcommand{\dn}{\partial_\nu}
\newcommand{\Oh}{\Omega_h}
\newcommand{\Pn}{\Pi^\nabla\!}

\newcommand{\ahBH}{\mathfrak{a}_h^{{\textsc{BH}}}}

\newcommand{\Vhglob}{{\mathbb{V}_h}}

\newcommand{\FK}{\mathcal{F}_K}
\newcommand{\BDT}{\mathscr{C}}
\newcommand{\Gh}{\Gamma_h}

\newcommand{\uI}{u_I}
\newcommand{\lI}{\lambda_I}
\newcommand{\Lh}{\Lambda_h}
\newcommand{\mh}{\mu_h}

\newcommand{\wCinv}{\widetilde C_{\mathrm{inv}}}

\newcommand{\CBDT}{C_{\natural}\,}
\newcommand{\monomials}[1]{\mathcal{M}_{#1}}
\newcommand{\PinablaK}{\Pi^ {\nabla\!,k}_K}
\newcommand{\Prhs}{\widehat\Pi^0}

\newcommand{\Tess}{\Th}
\newcommand{\nK}{{\n_K}}

\newcommand{\cstar}{c_*}
\newcommand{\Cstar}{C^*}
\newtheorem{theorem}{Theorem}[section]
\newtheorem{corollary}[theorem]{Corollary}
\newtheorem{lemma}[theorem]{Lemma}

\theoremstyle{definition}

\newtheorem{assumption}[theorem]{Assumption}
\theoremstyle{remark}
\newtheorem{remark}[theorem]{Remark}

\numberwithin{equation}{section}

\newcommand{\roundPrecision}{2}

\newcommand{\multeqref}[2]{\crefrange{#1}{#2}}

\definecolor{darkpastelgreen}{rgb}{0.01, 0.75, 0.24}
\definecolor{cadmiumgreen}{rgb}{0.0, 0.42, 0.24}
\newcommand{\Pib}{\Pi^0_\partial}
\newcommand{\Cinv}{C_\mathrm{inv}}
\newtheorem{comment}[theorem]{Comment}

\long\def\MCOMMENT#1{{\color{red}\bf }} %%
\title[Weakly imposed Dirichlet boundary conditions for 2D and 3D VEM]{Weakly imposed Dirichlet boundary conditions\\ for  2D and 3D Virtual Elements} %High order VEM on three dimensional curved domains}
\author[S. Bertoluzza]{Silvia Bertoluzza}
\address{ IMATI ``E. Magenes'', CNR, Pavia (Italy)}%
\email{silvia.bertoluzza@imati.cnr.it}%

\author[M. Pennacchio]{Micol Pennacchio}
\address{IMATI ``E. Magenes'', CNR, Pavia (Italy)}%
\email{micol.pennacchio@imati.cnr.it}%

\author[D. Prada]{Daniele Prada}
\address{IMATI ``E. Magenes'', CNR, Pavia (Italy)}%
\email{daniele.prada@imati.cnr.it}%

\date{\today}
\thanks{{This paper has been realized in the framework of ERC Project CHANGE, which has received funding from the European Research Council (ERC) under the European Union Horizon 2020 research and innovation programme (grant agreement No 694515), was co-funded by the MIUR Progetti di Ricerca di Rilevante Interesse Nazionale (PRIN) Bando 2017 (grant 201744KLJL), and it was partially supported by INdAM--GNCS}}

\begin{document}

\begin{abstract}
In the framework of 
virtual element discretizazions, we address the problem of imposing non homogeneous Dirichlet boundary conditions in a weak form, both on polygonal/polyhedral domains and on two/three dimensional domains with curved boundaries. 
%Different ways  can be considered and here we 
We consider a Nitsche's type method \cite{Nitsche,Stenberg_Nitsche_2009}, and the stabilized formulation of the Lagrange multiplier method  proposed by  Barbosa and Hughes in 
\cite{HughesBarbosa}. 
We prove that also for the virtual element method (VEM),  provided the stabilization parameter is suitably chosen (large enough for Nitsche's method and small enough for the Barbosa-Hughes Lagrange multiplier method), the resulting discrete problem is well posed, and yields convergence with optimal order on polygonal/polyhedral domains. 
On smooth two/three dimensional domains, we combine both methods with a projection approach similar to the one of \cite{BDT}.
We prove that, given a polygonal/polyhedral approximation $\Omega_h$ of the domain $\Omega$, an optimal convergence rate can be achieved by using a suitable correction depending on high order derivatives of the discrete solution along outward directions (not necessarily orthogonal) at the boundary facets of $\Omega_h$. 
Numerical experiments validate the theory.
% and show  the robusteness of the VEM framework combined with our strategy for imposing boundary conditions when applied to a challenging problem of ..... 3D reconstruction of the boundary. 

%
%The virtual element method (VEM) for solving the Poisson equation on a three-dimensional domain $\Omega$ with curved boundary in studied. 
%%Given a polygonal approximation $\Omega_h$ of the domain $\Omega$, the standard order $m$ VEM~\cite{hitchVEM}, for $m$ increasing, leads to a suboptimal convergence rate. 
%The approach of \cite{BDT}, already adapted to VEM in 2D \cite{VEM_curvo}, is now considered for three dimensional domains.  
%We prove that, given a polygonal approximation $\Omega_h$ of the domain $\Omega$, an optimal convergence rate can be achieved by using a suitable correction depending on high order normal derivatives of the discrete solution at the boundary edges of $\Omega_h$. 
%%, which, to retain computability, is evaluated after applying the projector $\Pi^\nabla$ onto the space of polynomials.   
%Numerical experiments validate %confirm
% the theory.
\end{abstract}

\maketitle

\input{intro}

\input{notation}

\input{VEM3D}

\input{tessellation}

\input{weakBC}

\input{Nitsche_BH}

\input{BH_BDT_VEM}
\section{Numerical Tests}\label{sec:expes}
In this section we present numerical experiments to test and validate the proposed methods on domains with curved boundaries. More specifically,   we aim at verifying the error bounds of Theorems \ref{thm:BH} and \ref{theo:error_BH}. and/or Corollaries~\ref{cor:nitsche} and \ref{cor:nitsche:curved}, as well as studying their dependence on the parameters  $\gamma$, and, for the curved domain case, on the order $k^*$ of the Taylor expansion, and the distance $\delta(x)$. As Nitsche's method and the Barbosa-Hughes method are equivalent, case by case we only report the result for one of the two. 
The calculation of the function $\delta(x)$ for a general curved domain can be cast as a rootfinding problem: letting the surface of the domain be represented by a nonlinear scalar equation $F = 0$, given $x\in\Gamma_h$,  and an outward direction $\sigma$, we let $\delta(x)$ be the smallest positive root of the equation \(F(x+ \delta\sigma) = 0.\)

\

Letting $u_h$ denote the discrete solution obtained by the order $k$ VEM methods proposed in the previous sections, we measure the following relative errors in broken $H^{1}(\Oh)$ and in the $L^2(\Oh)$ norm
\begin{align}\label{eS}
e^u_1 &:= \frac
{\big(\sum_{K\in \Tess}|| \nabla u - \Pi^{0,k-1}_K(\nabla u_h) ||^2_{0,K}\big)^{1/2}}
{|| \nabla u ||_{0,\Oh}}, & e^u_0 &:= \frac
{\big(\sum_{K\in \Tess}|| u - \Pi^{0,k}_K(u_h) ||^2_{0,K}\big)^{1/2}}
{|| u ||_{0,\Oh}}.
\end{align}

%
%Different choices for the stabilization bilinear form $\Svem$  can be found in the literature . 
	In the numerical tests, we made the standard choice of defining the stabilization bilinear form  $\Svem$ as the suitably weighted (properly scaled) euclidean scalar product of the vectors of degrees of freedom, according to the so called ``D-recipe"  \cite{dassi_mascotto_3DVEM}. In defining the degrees of freedom,  for each $K$ of $\mathcal T_h$ we use, as a basis for $\mathbb P_k(K)$, the set of polynomials obtained by orthogonalizing, with respect to the $L^2(K)$ inner product, the scaled monomials of degree less then or equal to $k$
	\begin{equation}\label{eq:monomials-3d}
	m_{\alpha}(x,y,z) = \left(\frac{x-x_K}{h_K}\right)^{\alpha_1}\left(\frac{y-y_K}{h_K}\right)^{\alpha_2}\left(\frac{z-z_K}{h_K}\right)^{\alpha_3}
	\quad\forall\,{\alpha} = (\alpha_1,\alpha_2,\alpha_3)\in\mathbb N^3,\quad\alpha_1+\alpha_2+\alpha_3 \leq k,
	\end{equation}
	where $(x_K, y_K, z_K)$ are the coordinates of the centroid of $K$. Moreover, as a basis for $\mathbb P_k(\face)$ for each face $\face$ of $\mathcal T_h$, we use the polynomials obtained by orthogonalizing, with respect to the $L^2(\face)$ inner product, the following monomials of degree less then or equal to $k$
	\begin{equation}\label{eq:monomials-2d}
	p_{\beta}(x,y) = \left(\frac{x-x_K}{h_K}\right)^{\beta_1}\left(\frac{y-y_K}{h_K}\right)^{\beta_2}
	\quad\forall\,{\beta} = (\beta_1,\beta_2)\in\mathbb N^2,\quad\beta_1+\beta_2 \leq k.
	\end{equation}
	% \cite{Mascotto}.

Some geometrical data relative to the meshes employed will be shown below. We use the following notation: $N_P$, number of polyhedra; $N_F$, number of faces; $N_E$, number of edges; $N_V$, number of vertices; $h = \max_{K\in\mathcal T_h}h_k$; $\overline h = \frac{1}{N_P}\sum_{K\in\mathcal T_h}h_K$; $h^{\text{min}} = \min_{K\in\mathcal T_h}h_K^{\text{min}}$, with $h_K^{\text{min}}$ being the minimum distance between any pair of vertices of $K$; $\gamma_0$, regularity parameter introduced in Assumption~\ref{shape_regular}. 
In the figures, the errors are plotted against the average mesh size $\overline h = \frac{1}{N_P}\sum_{K\in\mathcal T_h}h_K$, where $N_P$ is the number of polyhedra of a mesh.

\input{numerical_convergence_nitsche}

\input{numerical.tex}

\begin{remark}
	In the framework of the meshes considered in Test 3. (and more generally, in the framework of polytopal meshes), independently of the diameter of the elements, it is always possible, by allowing very small edges and faces, to make the distance $\delta$ between the approximated boundary $\partial\Oh$ and the continuous boundary $\partial\Omega$ as small as wanted. We could then make it as small as needed for the virtual element method to yield an optimal convergence without any correction. It is known (see \cite{strang1973change}) that imposing Dirichlet boundary conditions on $\partial \Oh$ rather than on $\partial \Omega$ yields an error of order $\delta^{3/2}$. To obtain an optimal convergence we would then need to chose the mesh so that $\delta = h^{2k/3}$. In particular, for the meshes considered in Test 3, this would imply that the diameter of the faces is $h_\face \sim h_K^{3k/2}$. For $k \geq 3$, this choice would imply an excessive increase in the computational cost of the method, and  a degradation of the constants appearing, in particular, in condition \eqref{defSK}, when using the most common stabilization strategies resulting from a degradation in the constant $N_0$ appearing in Assumption \ref{shape_regular}. For the formulation we are proposing, thanks to the inclusion of the correction term, we obtain optimal convergence rate for meshes with $\delta \simeq h_f \simeq \tau h_K$
	\end{remark}

	\begin{remark}\label{rem:deltasmooth}
		For all three test cases, the smoothness of the domain plays a key role
		in the implementation of the method, which requires choosing the direction $\sigma$ on each face of the tessellation (in our case we use the gradient of the distance $\delta$ to the boundary, which justifies our assumption that $\Omega$ is of class $C^1$), and in actually evaluating the correction term, which requires numerically evaluating $\delta$ within a quadrature rule. Use of high order quadrature rules requires high smoothness of the function $\delta$ (and therefore of the domain). If the domain is only piecewise smooth, this will have to be taken into account in designing the quadrature rule, and it might require giving up the assumption that the direction $\sigma$ is constant on each face. 
	\end{remark}

\section{Conclusion}
	We presented and analyzed two (equivalent) methods, namely the Barbosa-Hughes stabilized Lagrange multiplier method and Nitsche's method, for weakly imposing non homogeneous Dirichlet boundary conditions in the framework of the two- and three-dimensional virtual element method. For both methods we proved stability and optimal error estimates, under the customary conditions on the stabilization parameters. We also considered the combination of both methods with a modified version of the projection method by Bramble, Dupont and Thom\'ee,  for the virtual element solution of problems on smooth domains with curved boundaries, approximated by polygonal/polyhedral domains. We remark that, differently from the method proposed by Bramble and coauthors, in our version of the method, the direction used for computing the Taylor expansion involved in the correction is not necessarily the one orthogonal to the boundary of the approximating domain. The resulting method, for which  we can prove stability and optimal error estimates, is more flexible, allowing more freedom in the choice of the approximating domain, which, in three dimensions, is a highly desirable feature.
	
	The results of the numerical tests confirm the theoretical bounds, and suggest that the virtual element method with weakly imposed boundary conditions and boundary correction might provide a valid alternative to the finite element method, particularly in those situations where very large cubic (and, more in general, hexahedral and tetrahedral) meshes are used on problems set on smooth domains (as it happens, for instance, in the {\em microFEM} method). Working on polyhedral meshes with larger interior elements and with polyhedral boundary elements obtained by agglomeration, and  using the method presented here might allow to obtain more precise results with a much lower number of degrees of freedom.

%%%%%%
\bibliographystyle{amsplain}

\bibliography{biblio}

\end{document}

%% file: intro.tex
%!TEX root = ./VEM_weaklyBC_rev.tex

\section{Introduction}
Initially introduced in \cite{basicVEM} and \cite{hitchVEM}, the virtual element method (VEM) is a recent  PDE discretization framework aimed at generalizing the finite element method (FEM) to  
 meshes consisting of very general polytopes \cite{basicVEM}. 
The method  looks for the solution in a conforming discretization space using a non conforming Galerkin approach based on an approximate bilinear form, split as the sum of two components:
the first one is strongly consistent on polynomials, and guarantees the accuracy; the second one is a non consistent stabilization term, vanishing for polynomials and ensuring the well posedness of the discrete problem. 

\noindent The elements of the discretization space are not known in closed form, but are themselves solution of local PDEs, and everything is computed directly in terms of a  set of degrees of freedom, by resorting to suitable ``computable'' (in terms of the degrees of freedom) elemental projectors onto the space of polynomials (see \cite{hitchVEM}), thus avoiding  the explicit construction of the basis functions, whence the name {\em virtual}. 

Thanks to its flexibility and robustness with respect to mesh design, VEM enjoyed, in recent years, wide success: the theoretical analysis of the method has been extended in different directions \cite{
VEM_mixed,beirao_hp,beirao_hp_exponential,VEM-curved,curved_Trefftz,berrone_borio_17,de2016nonconforming,BMPP_VEM_nonconforming_stab,brenner2018virtual}.
Problems related to the efficiency of the method have been addressed \cite{FETI_VEM_2D, FETI_VEM_3D, antonietti_p_VEM,CALVO20191163}, %, because of its ﬂexibility and robustness with respect to mesh design and handling; 
and different model problems 
have been tackled, see e.g.  (\cite{beirao_parab,Antonietti_VEM_Stokes,Antonietti_VEM_Cahn,beirao_stokes,beirao_Navier_Stokes,perugia_Helmholtz,beirao_elastic,beirao_linear_elasticity,VEM_3D_elasticity,VEM_discrete_fracture,VEM_Laplace_Beltrami, ABPV:minsurf,wriggers2017efficient,wriggers2016virtual,brenner2021ac}).

\

In this paper, we specifically focus on the extension to the VEM of two widely used methods for imposing non homogeneous Dirichlet boundary conditions in a weak form: the Lagrange multiplier method, in its stabilized formulation as proposed by Barbosa and Hughes (\cite{HughesBarbosa, HughesBarbosa2}), and  Nitsche's method \cite{Nitsche}. While apparently more cumbersome, weak imposition of boundary conditions has some advantages over the more straightforward inclusion of the boundary value in the definition of the discretization space. Both the Lagrange multiplier method and, after suitable post-processing, Nitsche's method provide, as a byproduct, a stable approximation of the outer normal flux of the solution on the boundary, a quantity that might be relevant for the end user, and they can be used for weakly coupling different discretizations \cite{Hansbo2005nitsche}. Moreover, such methods turn out to be advantageous for several class of equations, such as advection diffusion problems or Navier-Stokes equation, allowing for higher accuracy on relatively coarse meshes \cite{BazilevHughes2007,BAZILEVS20074853}.

\

%The Lagrange multiplier method, 
Originally proposed by Babuska  \cite{babuska} for FEM, the Lagrange multiplier method requires the difference between the trace of the discrete solution on the boundary and the Dirichlet boundary data to be orthogonal to a suitable multiplier space. 
%  to converge optimally,  
%the finite element spaces i,.e. 
In order for the method to converge optimally, such a space must satisfy,  with the approximation space for the solution, an ``inf-sup" type compatibility condition, which can be too restrictive already in the finite element framework, and the more so, in the virtual element method.
To relax such a requirement and gain more freedom in the choice of the multiplier space, we rather consider here the stabilized formulation proposed by  Barbosa and Hughes in 
\cite{HughesBarbosa}, that 
allows the approximation spaces for the two unknowns (solution and multiplier)
to be chosen arbitrarily, independently of each other.  In the VEM framework, on polygonal domains, we can then choose the multiplier to be a space of discontinuous piecewise polynomials. In the VEM spirit, we ensure the computability of the method  by replacing, in the Barbosa-Hughes formulation, the normal derivative of test and trial functions, with the normal derivative of the respective projections on the space of discontinuous polynomials, and we prove that, also in such framework, provided the stabilization parameter is small enough, the resulting discrete problem is well posed, and yields convergence with optimal order. As it happens in the finite element context \cite{Stenberg.1995}, also for the VEM, eliminating the multiplier in the Barbosa Hughes formulation results in Nitsche's method. The analysis of the latter can then be directly derived from the analysis of the former.

\

Both the Barbosa Hughes Lagrange multiplier formulation and Nitsche's method can also be used to handle problems on domains with smooth curved boundaries, suitably approximated by polygonal/polyhedral domains. Different approaches for the accurate treatment of curved domains in the finite element framework can be found in literature, see e.g. \cite{FEMsurvey_curved}.  We recall that  the plain approximation of a curved domain by straight facets introduces an error that, for higher order methods, can dominate the analysis. To compensate for such an error, we rely here on the strategy proposed by Bramble, Dupont and Thom\'ee for finite elements in the framework of Nitsche's method \cite{BDT}. Such a strategy has already been adapted to the Lagrange multiplier method, once again in the context of finite elements, in \cite{burman2020dirichlet}. The idea consists in imposing a corrected boundary condition which, by a suitable Taylor expansion, takes into account that the Dirichlet data is given not on the boundary of the approximating domain but, rather, on the boundary of the exact curved domain. It  has already been applied to the VEM in \cite{VEM_curvo}, in the context of Nitsche's method in two dimensions. Here, we combine such a strategy with the Barbosa Hughes Lagrange multiplier method, resulting in a method for which we prove stability and optimal error estimate, under suitable conditions on the approximating polyhedral domain (see \eqref{deftau}). Eliminating the multiplier results once again in a Nitsche's type method, which, in two dimension, coincides with the one proposed in \cite{VEM_curvo}.
It is worth pointing out that, as proposed in \cite{VEM_curved_beirao,VEM-curved,curved_Trefftz}, the  discretization of two dimensional problems with non homogeneous boundary conditions on domains with curved boundaries can also be carried out by resorting to virtual elements or Trefftz elements with curved edges, which can also be exploited to solve problems with curved interior interfaces. 

\

Unlike \cite{BDT, burman2020dirichlet, VEM_curvo}, which consider a Taylor expansion in the direction $\nu_h$ normal to the facets of the approximating domain, in order reduce the error deriving from approximating the domain, thus gaining more freedom in the choice of the approximation, here, in defining the corrected boundary condition, we consider an arbitrary direction $\sigma$ (see Figure \ref{fig:sigma}), which can be chosen in such a way that the distance to the exact boundary along $\sigma$  is as small as possible. In three dimensions, where the problem of constructing polyhedral domains approximating a curved domain can be quite challenging, this freedom in the choice of the direction for the Taylor expansion will turn out to be particularly advantageous,  as it allows to use a much larger class of approximating domains, such as, for instance, polycubic domains, which naturally arise in many applications (for instance those where the domain of definition of the problem is given as a result of an imaging process such as micro-CT \cite{van1995new}), thus potentially providing, in such situations, a valid alternative to finite element  methods.

%The results of the numerical tests confirm the theoretical bounds, and suggest that the virtual element method with weakly imposed boundary conditions and boundary correction might provide a valid alternative to the finite element method, particularly in those situations where very large cubic (and, more in general, hexahedral and tetrahedral) meshes are used on problems set on smooth domains (as it happens, for instance, in the {\em microFEM} method). Working on polyhedral meshes with larger interior elements and with polyhedral boundary elements obtained by agglomeration, and  using the method presented here might allow to obtain more precise results with a much lower number of degrees of freedom.

%
%Then we considered %the problem of extending the method to problems in 
%3D domains with smooth curved boundaries.  

\

The paper  is organized as follows. In section \ref{sec:VEM} we recall the virtual element method whereas in Section \ref{sec:BH-Nitsche} we derive the Barbosa-Hughes method in the virtual element framework and we provide a theoretical analysis of the resulting discretization, proving stability and an error estimate.
%We also verify that also in the VEM context, as observed by Stenberg in \cite{Stenberg.1995}, the Nitsche and the Barbosa-Hughes methods are strictly related to one another.
 In the same section we show that  eliminating the multiplier 
by static condensation,  the well known Nitsche's method for imposing Dirichlet boundary condition can be retrieved. 
 Then  in Section \ref{sec:curvo} we introduce and analyze the discretization for the problem on domains with curved boundaries, proving also in this case stability and optimal error estimate, as well as, once again, the equivalence with a Nitsche's type method. Finally, in Section \ref{sec:expes}, we test the method on several three dimensional test cases, with discretizations of different order. 

%%%%%%

%% file: notation.tex
%!TEX root = VEM_weaklyBC_rev.tex

\newcommand{\Tf}{\mathcal{T}_\face}
\newcommand{\Tstar}{\mathcal{T}^*}
\newcommand{\hf}{h_\face}
\newcommand{\df}{\delta_\face}
\newcommand{\BB}{\mathcal{B}}
\newcommand{\Pif}{\Pi^{0,k'}_\face}

\section{The Virtual Element Discretization} \label{sec:VEM} 

\subsection{Notation}\label{sec:notation}
We use standard notation for Sobolev spaces, norms and semi norms. More precisely, given any $d$--dimensional domain $D $, $d=1,2,3$ we let the order $k$ Sobolev space $H^k(D)$ be  endowed with the standard seminorm and norm
\[
 | u |^2_{k,D} = \sum_{|\beta|=k} \int_D | D^\beta u 
|^2, \qquad \| u \|^2_{k,D} = \sum_{|\beta|\leq k} \int_D | D^\beta u 
|^2,
\]
with, for $\beta \in \mathbb{N}^d$, $| \beta | = \beta_1 + \cdots +\beta_d$,  $D^\beta = \partial^{\beta_1}_{x_1} \cdots \partial^{\beta_d}_{x_d}$.
We also denote by $| \cdot |_{\kappa,\infty,D}$ the standard semi norm for the Sobolev space $W^{k,\infty}(D)$
\[
| u |_{k,\infty,D} = \sum_{|\beta| = k} \mathrm{ess}\ \sup_x | D^\beta u(x) |.
\]
Moreover, we denote by $H^{1/2}(D)$ the fractionary Sobolev space of order $1/2$, endowed with the norm and seminorm
\[
| u |_{1/2,D}^2 = \int_D \int_D \frac{|u(x) - u(y)|^2}{| x-y|^{1+d}}\,dx \,dy , \qquad \| u \|^2_{1/2,D} = \| u \|^2_{0,D} +  | u |^2_{1/2,D}.
\]

\

We let $\mathbb{P}_k$ denote the space of polynomals of degrees less than or equal to $k$, and by $\mathbb{P}_k(D)$ its restriction to the domain $D$, with the convention that $\Poly{-1}= \{ 0\}$.
We let $\Pi^{\nabla,k}_D: H^1(D) \to \Poly{k}(D)$ denote the $H^1$ type projection defined by
\begin{equation}\label{defPinabla}
 \int_D\nabla \Pi^{\nabla,k}_D (v) \cdot \nabla w = \int_D\nabla v \cdot \nabla w, \quad \forall w \in \Poly{k}\pagebreak(D), \qquad  \int_{\partial D} \Pi^{\nabla,k}_D (v) = \int_{\partial D} v.
\end{equation}
We also let $\Pi^{0,k}_D: L^2(D) \to \Poly{k}(D)$ denote the orthogonal projection with respect to the $L^2(D)$ scalar product:
\[
\int_D \Pi^{0,k}_D(v) \, w = \int_D w\, v, \quad \forall v \in \Poly{k}(D).
\]

\subsection{The tessellation}

Let $\Omega \subset \mathbb{R}^d$, $d=2,3$, be a bounded polygonal/polyhedral domain. We let $\Gamma = \partial\Omega$ denote its boundary, and $\n$ denote the outer unit normal on $\Gamma$.
We consider in the following a family $\{ \Th \}_h$ of tessellations of $\Omega$ into a finite number of  polygonal/polyhedral elements $K$.  For the sake of simplicity, from now on we shall use a three-dimensional notation, and speak therefore of polyhedrons and faces. 
The change of terminology in the polygonal case is obvious and left to the reader. 

Given a tessellation $\Th$ in the family, we let 
$\Fb$ and $\FK$ denote, respectively, the set of faces of $\Th$ lying on, respectively, $\Gamma$ and  $\partial K$. Remark that $\Fb$ is the tessellation on $\Gamma$ obtained as the trace of $\Tess$.
For each element $K \in \Tess$ we let $\nK$ denote the outer unit normal to $K$. For a polyhedron $K$ we let $h_K$ denote its diameter, and we let $h = \max_{K\in \Th} h_K$. For a boundary face $\face\in \Fb$, we denote by $h_\face$ its diameter and by $\hKf$ the diameter of the element $K$ such that $\face \in \FK$:
\begin{equation}\label{defhtildef}
\hKf := h_K\quad \text{ for } \face \in \Fb\cap\FK.
\end{equation}

\

We make following assumptions on the family of tessellations $\{\Tess\}_h$ 
	\begin{assumption}\label{shape_regular}   There exist positive constants $\gamma_0$, $N_0$ such that
		the following properties hold for all elements of all the tessellations in $\{\Tess\}_h$ :
		\begin{enumerate}
			\item\label{sr1} $K$ is star shaped with respect to all points of a $d$ dimensional ball of radius $\rho_K \geq \gamma_0 h_K$;
		\item\label{sr3} The number $N_K$ of $(d-1)$ faces of $K$ satisfies $N_K \leq N_0$;
		\item\label{sr4} If $d = 3$, then all face $\face$ of $K$ is star shaped with respect to a disc of radius $\rho_\face \geq \gamma_0 h_\face$, and the number $N_\face$ of edges of $\face$ satisfies $N_\face \leq N_0$.
		\end{enumerate}
	\end{assumption}

%
%
%\begin{figure}
%	\includegraphics[width=4cm]{figures/polygon1.png}\qquad
%		\includegraphics[width=4cm]{figures/polygon2.png}\quad
%				\includegraphics[width=4cm]{figures/polygon3.png}
%		\caption{
%For three polygons we draw non overlapping triangles $\kappa_T$, $T \in \Tstar = \FK$ (in yellow, the triangle with the lowest height). Given a constant $\gamma_1$,  if the height of such triangle  is $\geq 2 \gamma_1 h_K$, then  Assumption \ref{shape_regular}(\ref{sr2}) holds.}
%\end{figure}

\

Unless the specific value of the constant $C$ is explicitely  needed, in the following we will write $A \lesssim B$ (resp. $A \gtrsim B$) to indicate that the quantity $A$ is bounded from above (resp. from below) by a constant $C$ times the quantity $B$, with $C$ possibly depending on $\Omega$ as well as on $\gamma_0$ and $N_0$ but otherwise independent of the shape and size of the elements of the tessellations. The notation $A \simeq B$ will stand for $A \lesssim B \lesssim A$.

\

We let $\Poly{k}(\Tess)$ and $\Poly{k}(\Fb)$ denote the spaces of discontinuous piecewise polynomials of order up to $k$  defined, respectively, on the tessellation $\Tess$ and on its trace $\Fb$ on $\Gamma$:
\begin{gather*}
\Poly{k}(\Tess) = \{
v \in L^2(\Omega): \ v|_K \in \Poly{k}(K)\ \text{for all }K \in \Tess
\},\\[2mm]
\Poly{k}(\Fb) = \{
v \in L^2(\Gamma): \ v|_\face \in \Poly{k}(\face)\ \text{for all }\face \in \Fb
\}.
\end{gather*}

\

Under Assumption \ref{shape_regular}, several bounds hold with constants depending on the constants $\gamma_0$ and $N_0$, but otherwise independent of the shape and size of the polyhedrons (see \cite{Cangiani_book,brenner2018virtual}). 

\subsubsection*{Inverse inequalities for polynomials}  Let $K \in \Tess$.
For all $p \in \Poly{k}(K)$ and for all 
%$j,k$ with  $0 \leq j \leq k$
$m,j$ with  $0 \leq m \leq j$ it holds that
\begin{gather}\label{inversebase}
%	\| p \|_{k,\el} \lesssim h_\el^{j-k} \| p \|_{j,\el}.
\| p \|_{j,\el} \lesssim h_\el^{m-j} \| p \|_{m,\el},\\
\label{tracePoly}
\norm{p}_{0,\partial\el} \leq \Cinv\,  \hel^{-1/2} \norm{p}_{0,\el}.
%\\
%\label{inversenormal}
%h_K \int_{\partial K}\left|  \dn p  \right|^2 \leq \Cinv | p |_{1,K}^2.
\end{gather}
	Applying \eqref{tracePoly} to $\nabla p$, $p \in \Poly{k}(K)$ we immediately obtain 
		\begin{equation}\label{inversenormal}
		\| \nabla p \|_{0,\partial K} \leq  \Cinv\, \hel^{-1/2}
		| p |_{1,\el}.	\end{equation}

%\subsubsection* {Trace inequalities} Let $K \in \Tess$. For all $v \in H^1(K)$, for all $T \in \Tstar$ it holds that
%\begin{equation}\label{trace0s}
%	\| v \|^2_{0,T} \lesssim h_T^{-1}   \| v \|^2_{0,K_T} + \| v \|^2_{0,K_T} \| \nabla v \|_{0,K_T} \lesssim h_T^{-1}   \| v \|^2_{0,K_T} +  h_T \| \nabla v \|^2_{0,K_T},
%	\end{equation}
%	($\Tstar$ and $\kappa_T$ as in Assumption \ref{shape_regular}(\ref{sr2})).

\subsubsection*{Polynomial approximation} Let $K \in \Tess$ and $\face \in \FK$, and let 
 $w \in H^{s}(K)$ and $\phi \in H^{s}(\face)$, $0 \leq s \leq k+1$. Then, for $0 \leq t \leq s$ we have the following approximation bounds:
			\begin{equation}\label{approxj}
			\| u - \Pi^{0,k}_K (u )\|_{t,K} \lesssim h_K^{s-t} | u |_{s,K},
			\end{equation}	
				\begin{equation}\label{approxjface}
				\| \phi - \Pi^{0,k}_\face (\phi)\|_{t,\face} \lesssim h_\face^{s-t} | \phi |_{s,\face}.
				\end{equation}	
				Moreover, provided $s > 3/2$, we have that $\nabla u|_{\partial\Omega} \in (H^{s-3/2} (\partial K))^2 \subset (L^2(\partial K))^2$ so that $\nabla  u  \cdot \n _K \in L^2(\partial\K)$ and, if $s \leq k+1$, we have
				\begin{equation}\label{eq:approxdn}
				\|\nabla u \cdot \n _K - \nabla \PinablaK(u)\cdot \n_K \|^2_{0,\partial K} \lesssim h_K^{s-1} | u |_{s,K}.
				\end{equation}

%\COMMENT{Capire se servono anche queste.
%	
%	Questo non serve, secondo me.
%	\begin{lemma}
%		Let $T\subset \mathbb{R}^d$, $d = 2,3$, be a simplex and let $\face$ be one of its faces. Then, for all $v \in H^1(T)$ it holds that
%		\[
%		\| v \|_{0,\face} \lesssim \frac{| \face | }{| T |} \left(
%		| v \|^2_{0,T} +
%		h_T	 \| v \|_{0,T}	 \| \nabla v \|_{0,T}
%		\right),
%		\]
%		the constant in the inequality depending on $d$ but independent of the shape regularity of $T$.
%	\end{lemma}
%	
%\begin{itemize}
%	\item for $K_T$ shape regular (\ref{trace0}) implies
%
%	\item 
%	%
%	for all $v \in H^1(\el)$ we get that:
%	\begin{equation}\label{trace1}
%	\| v \|^2_{0,\face} \lesssim h_\face^{-1}\| v \|^2_{0,K_\face} + h_\face|v |^2_{1,K_\face}.
%	\end{equation}
%\end{itemize}
%
%	Non so se questa serve
%	\begin{equation}
%	\label{inverseBase}
%	\| \nabla p \|_{0,K} \lesssim h_K \| p \|_{0,K}, \qquad \forall p \in \Poly{k}(K)
%	\end{equation}}

	%\todo{questo spostare dopo con il curvo}
	%
	%Introduce $\delta$, $\sigma$, $\tau$......
	%
	%\begin{lemma}\label{inverse_ineq}
	%	Let $\delta\leq \hf$ then we have:
	%	\begin{equation}
	%		\hel^{-1} \sum_{\face\in \Fb} \df^{2j} \norm{\partial_\sigma^j p}^2_{0,\face} \lesssim \tau^{2j-1} \snorm{p}^2_{1,\el}
	%	\end{equation}
	%\end{lemma}
	%\begin{proof}
	%	
	%		\begin{gather}
	%		\hel^{-1} \sum_{\face\in \Fb} \df^{2j} \norm{\partial_\sigma^j p}^2_{0,\face} \lesssim 	\hel^{-1} \sum_{\face\in \Fb}  \df^{2j} \norm{D^j_p(\sigma,0)}
	%		\\
	%		\lesssim  \tau^{2j-1} \snorm{p}^2_{1,\el}
	%	\end{gather}
	%	
	%	\end{proof}
	%
	%\bigskip

%% file: VEM3D.tex
%!TEX root = VEM_weaklyBC_rev.tex

\subsection{The virtual element space}

For the sake of completeness, and to introduce the relevant notation, let us  recall the definition and some of the properties of the Virtual Element discretization space (see \cite{3DVEM}). We focus on the more complex three dimensional case, and we refer to \cite{basicVEM} for the two dimensional case.
 %,highorderVEM_3D}.  %hitchVEM,beirao_stab} 
%\subsubsection*{The local face space}  
The order $\p$  VE discretization space on $\Th$ is
defined, element by element, starting from the edges of the tessellation, where the discrete functions are defined as continuous degree $\p$ piecewise polynomials. The virtual element functions are  then subsequently defined on the faces, and  in the interior of the polyhedra. 
More precisely, on the boundary of each face  $\face$ we define the space: 
\begin{gather*}
	\Bone(\partial \face) = 
	\left\{g\in C^0(\partial \face) :  g_{|{e}}\in \mathbb{P}_\p(e) \text{ for all edge } e \subseteq \partial \face 
	\right\}.
\end{gather*}

%The value of a function $g$ in $\Bone(\partial \face)$ is uniquely determined by
%\begin{enumerate}[A]
%	\item the value of $g$ at the vertices of $\face$;
%	\item \color{blue}{the value of $g$ at the $\p - 1$ interior nodes of the $\p + 1$ Gauss--Lobatto quadrature formula on $\edge$ for all edge $\edge$ of $\face$. }
%\end{enumerate}
%In order to define the discrete face space $V^\face$, $\face$ being a face of one of the polyhedra of the tessellation, w
We let  $\tVfp$ be defined   as 
\begin{gather}\label{defspace1}
	\tVfp =  \{ v \in C^0(\face):\, v_{|
		\partial \face} \in\Bone(\partial \face),\, 
	\Delta v \in \mathbb{P}_\p(\face) \}.
\end{gather}
It is known (see \cite{basicVEM}) that an element in $\tVfp$ is uniquely identified once its trace on $\partial \face$ and its moments up to order $\p$ are known. 
We let $\monomials{k-2} \subset\monomials{k}$ denote two bases for, respectively, $\Poly{k-2}(\face)$ and $\Poly{k}(\face)$.
We can then introduce the  {\em local virtual face space} $\Vfp$:
\begin{gather}\label{defspace2loc}
	\Vfp =  \{ v \in \tVfp:\  \int_\face  q\, v = \int_\face q\, \Pinablaf( v),\, \forall q \in \monomials{k} \setminus\monomials{k-2} \} ,
\end{gather}
where $\Pinablaf$ is defined according to \eqref{defPinabla}.

\

The {\em local boundary space} is built on the boundary of each polyhedron  $K$  by assembling the local face spaces:
$$ \Bone(\partial K) : = 
\left\{g\in C^0(\partial K) : g_{|\face}\in \Vfp \text{ for all face } \face \subseteq \partial K 
\right\}.
$$
%defined on the boundary of $K$, 
Finally, the {\em local element space} on $K$ is defined as
\begin{equation}\label{defspace2} \VKp: = 
	\left\{v\in H^1(K) : v_{|\partial K}\in \Bone(\partial K)  \text{ and } \Delta v\in \mathbb{P}_{\p-2}(K)
	\right\} .
\end{equation}
% defined on $K$. 
%Finally  

The global discrete VE space   $V_h$ is finally assembled by continuity:
\begin{equation}\label{VEMm}
V_h : = 
\left\{V \in H^1(\Omega) : v_{|K}\in \VKp\text{ for all } K  \in \Th
\right\} .
\end{equation}
It can be easily checked that for all faces $\face$ and all polyhedrons $K$, $\Poly{k}(\face) \subseteq \Vfp$ and $ \mathbb{P}_{\p}(K)\subseteq \VKp$.

\

A function in $V_h$ is uniquely determined by the following degrees of freedom:
\begin{itemize} 
	\item the values at the vertices of the tesselation and the  values at the $\p-1$ internal nodes of the $\p+1$-points Gauss-Lobatto quadrature rule;
	on each edge $e$;
	\item the moments up to order $k-2$ on each face;
	\item the moments up to order $k-2$ in each element.
	\end{itemize}

In particular, for any given function $w \in H^2(\Omega)$ we can define the unique interpolant function $w_I \in V_h$ whose degrees of freedom coincide with the ones of $w$.
	The following local approximation bound \cite{basicVEM,3DVEM} holds: 
	if $w \in H^s(K)$, with $2 \leq s \leq \p+1$, then
	\begin{equation}\label{VEMapproxI}
	\| w - w_I \|_{0,K} + h_K | w - w_I |_{1,K} \lesssim h_K^s | w |_{s,K}.
	\end{equation}

%\begin{enumerate}
%	\item[\labeldofv] the values at the vertices of the tesselation
%\end{enumerate}
%\hspace{.7cm} and for $\p\ge 2$ 
%\begin{enumerate}
%	\item[\labeldofe]  the  values at the $\p-1$ internal points %of the $\p+1$-points Gauss-Lobatto quadrature rule 
%	on each edge $e$
%	\item[\labeldoff] the face moments $\displaystyle \int_\face v p_{\p-2} \quad \forall p_{\p-2}\in \mathbb{P}_{\p-2}(\face)$, $\forall f$ face of $\partial K$
%	\item[\labeldofK]  the internal moments $\displaystyle\int_K v p_{\p-2} \quad \forall p_{\p-2}\in \mathbb{P}_{\p-2}(K).$
%	%\end{enumerate}
%\end{enumerate} 
%A  function in $V_h$ is uniquely determined by the following degrees of freedom 
%\begin{itemize}
%	\item its  values at the vertices of the tessellation;
%	\item (only for $m \geq 2$) for each edge $e$, its values at the $m-1$ internal points of the $m+1$-points Gauss-Lobatto quadrature rule on $e$;
%	\item (only for $m \geq 2$)  for each face $f$, its moments in $f$ up to order $m-2$
%	\item (only for $m \geq 2$)  for each element $K$, its moments in $K$ up to order $m-2$. 
%\end{itemize} 
%
%\begin{enumerate}
%	\item[\labeldofv] the values of $v$ at the vertices of $\face$ (\emph{vertex degrees of freedom});
%	\item[\labeldofe] the values of $v$ at the $\p - 1$ interior nodes of the  $\p + 1$ points Gauss-Lobatto integration rule (\emph{edge degrees of freedom});
%	\item[\labeldoff] the  value of the scalar products  (\emph{face degrees of freedom}) 
%	\[
%	\int_\face v \pjf ,\quad \text{for all $\pjf  \in \BasePoly_{\p - 2}$} .
%	\]
%	%	(face degrees of freedom).
%\end{enumerate}

\subsection{Computable bilinear forms and operators}

A key concept in the  definition of  virtual element methods is the one of {\em computability}. Essentially,  operators or bilinear forms, acting on virtual element functions, are said to be computable if the knowledge of the degrees of freedom of the argument functions is sufficient for the direct evaluation of the operator/bilinear form, without the need of solving the PDEs implicitly involved in the definition of $\VKp$.
 We recall that both the elliptic projector $\Pinablaf$  and the $L^2$ projector $\Pi^{0,k}_\face$  onto $\Poly{k}(\face)$  (see Section \ref{sec:notation}) are computable, as are the elliptic projector $\PinablaK$ onto $\Poly{k}(K)$ and the $L^2$ projector $\Pi^{0,k-2}_K$ onto $\Poly{k-2}(K)$. On the other hand, neither the projector $\Pi^{0,k}_K$, nor the bilinear form $a: H^1(\Omega) \times H^1(\Omega) \to \mathbb{R}$ and its local counterpart $a^K: H^1(K) \times H^1(K) \to \mathbb{R}$
\[ a(\phi,\psi) =  \int_\PolyDom \nabla \phi \cdot \nabla \psi, \qquad a^K(\phi,\psi) = \int_K\nabla \phi \cdot \nabla \psi,\]
 are computable. The bilinear form $a$ is replaced, in the definition of the virtual element discretization, by an approximate bilinear form $a_h: V_h \times V_h \to \mathbb{R}$
\[
a_h(u_h,v_h) = \sum_K a_h^K(u_h,v_h),
\]
 where the elemental approximate bilinear form $a_h^ K: \VKp \times\VKp \to \mathbb{R}$ is defined as
\[
a^{K}_h(\phi,\psi) = a^K(\PinablaK  (\phi), \PinablaK  (\psi) ) + \Svem(\phi - \PinablaK  (\phi), \psi - \PinablaK  (\psi)),
\]
the stabilizing bilinear form  $ \Svem$ being   any computable  symmetric  bilinear form satisfying
%We recall that different choices are possible for the bilinear form $\Svem$ (see \cite{beirao_stab}),
%the essential requirement being that it satisfies
%
\begin{equation}
\label{defSK}
\cstar a^K(\phi,\phi) \leq  \Svem(\phi,\phi) \leq \Cstar a^K(\phi,\phi),\quad \forall \phi \in \VKp \ \text{ with } \Pinabla  \phi=0,
\end{equation}
with $\cstar$ and $\Cstar$ two positive constants independent of $K$. 
Different choices for the bilinear form $\Svem$ are available in the literature (see \cite{beirao_stab}), some of which allow to obtain bounds of the form \eqref{defSK} under the shape regularity assumption \eqref{shape_regular}, while others might require the tessellation to satisfy some more restrictive assumption.  It is out of the scope of this paper to detail all the possible constructions for such a term.
We briefly recall that the simplest, and widely used, choice is directly expressed in terms of the degrees of freedom as the suitably scaled euclidean scalar product \cite{3DVEM2}.
%In the numerical tests performed we made the standard choice of defining $\Svem$ as the properly scaled weighted  euclidean scalar product of the vectors of degrees of freedom, according to the so called ``D-recipe" \cite{dassi_mascotto_3DVEM}.
%
An alternative which gives sharper bounds  as the polynomial order increases and that we chose to use in the numerical tests  performed in Section \ref{sec:expes},  is the so called ``D-recipe" \cite{dassi_mascotto_3DVEM}, where $\Svem$ is defined as a suitably weighted (and, once again, properly scaled) euclidean scalar product of the vectors of degrees of freedom. 
Other recipes  include suitably scaled versions of the $L^2(\partial K)$ and of the $H^1(\partial K)/\mathbb{R}$ scalar products \cite{beirao_stab}, and  the
bilinear form corresponding to the Laplace-Beltrami operator on the boundary  $\partial K$  (a recipe proposed for the nonconforming version of the method but that applies also to the conforming version here considered, see \cite{BMPP_VEM_nonconforming_stab}).

\

The local discrete bilinear forms satisfy by construction the
following two properties, which play a key role in the analysis of virtual element methods:
\begin{itemize}
	\item {\em Stability and continuity}: 
	\begin{equation}\label{contcoercK}
| \phi |_{1,K}^2 \lesssim	a_h^K(\phi,\phi)\quad \forall \phi \in \VKp, \qquad\text{and}\qquad a^K_h(\phi,\psi) \lesssim | \phi |_{1,K} | \psi |_{1,K}\quad \forall \phi,\psi \in \VKp;
	\end{equation}
	\item {\em $\p$-consistency}: 
	\begin{equation}\label{mconsistency}
		a_h^K(\phi,p) = a^K(\phi,p) \qquad  \forall \phi\in \Vh \text{ and } p\in \Poly{\p}(K). 
	\end{equation}
\end{itemize}

\

The local stability and continuity property \eqref{contcoercK} implies, of course, the validity of global stability and continuity bounds:
	\begin{equation}\label{contcoerc}
	| \phi |_{1,\Omega}^2 \lesssim	a_h(\phi,\phi)\quad \forall \phi \in V_h, \qquad\text{and}\qquad a_h(\phi,\psi) \lesssim | \phi |_{1,\Omega} | \psi |_{1,K}\quad \forall \phi,\psi \in V_h.
	\end{equation}

%\subsubsection{Discrete load term}
%We denote by $f_h$ the piecewise polynomial approximation of $f$ on $\Omega$  given by
%\[f =\Pzero f\]

\begin{remark} Assumption \ref{shape_regular} implies the validity of bounds \eqref{inversebase} through  \eqref{eq:approxdn}, as well as the validity of the virtual element interpolation bound  \eqref{VEMapproxI}.  It is also generally required to prove the bounds \eqref{defSK} for the most common choices of the stabilization term. We would however like to point out that it is not a necessary condition for our theoretical result to hold. Indeed, depending on the particular construction of the tessellation, the  bounds \eqref{inversebase} through  \eqref{eq:approxdn} as well as the interpolation bound \eqref{VEMapproxI} on which our theoretical analysis relies, might hold under weaker assumptions. In particular we recall that the bounds \eqref{inversebase} through  \eqref{eq:approxdn} on polynomials hold under the assumptions in \cite[Assumptions 13 and 30]{Cangiani_book}, which allow for elements with an a priori unbounded number of very small faces. 
\end{remark}

%% file: tessellation.tex
%\subsection*{Mesh assumptions} \todo{rilassare ipotesi - Cangiani }
%%Let us  then start by introducing the assumptions on the tessellation. 
%, which we assume 
%%to be shape regular according to the following definition (quite standard in the theoretical study of VEM).
%%We assume the tessellation $\Th$ 
%to be quasi uniform and uniformly shape regular.
%
%\begin{assumption}\label{ass:tessellation} We assume that there exist two constants $\Nstar$ and $\gammastar$ such that the tessellation $\Th$ verifies the following assumptions
%	\begin{enumerate}
%		
%		%		\item $\Th$ is geometrically conforming, that is for each $\el$, $\el'$ in $\Th $, $\bar\el \cap \bar \el'$ is either the empty set, 
%		%		 a vertex, an edge or a face of both $\el$ and $\el'$; \blu{answer {\bf m4} referee 1}
%		\item {$\Th$ is geometrically conforming, that is for each $\el$,  if a vertex,  edge  or  face of $\el$ is contained in $\el \cap \el'$,  it is also, respectively a vertex,  edge or face of $\el'$;} 	
%		\item all $\el \in \Th $ are shape regular of diameter $h_\el$ with constants $\gammaK \geq \gammastar$ and $\NK \leq \Nstar$; 
%		\item the tessellation is quasi uniform, that is there exists an $h$ such that for all $\el \in \Th $  $h_\el \simeq h$.
%	\end{enumerate}
%	
%\end{assumption}

%% file: weakBC.tex
\section{Weakly imposing boundary condition on polygonal domains} \label{sec:BH-Nitsche} 
In the framework of virtual element discretizations, we specifically address, in this paper, the problem of imposing Dirichlet boundary conditions in a weak form. For the sake of simplicity, we focus on a simple model problem, namely  the Poisson equation 
\begin{equation}\label{prob_mod_PolyDom}
-\Delta u = f, \text{ in }\PolyDom, \qquad u = g \text{ on }\Gamma = \partial \PolyDom,
\end{equation}
 on a convex bounded domain $\Omega \subset \mathbb{R}^d$, $d=2,3$, which, to start with, we assume to be polygonal/polyhedral.  Having in mind our goal,  we consider the following  variational formulation of the Poisson equation:
 {\em given $f\in L^2(\Omega)$ and
 	$g \in H^{1/2}(\Gamma )$, find 
 	$u \in H^1(\Omega)$ and $\lambda \in (H^{1/2}(\Gamma)) '$ such that}
 \begin{gather}\label{BabuskaCont1}
 \int_{\Omega }\nabla u \cdot\nabla v + \langle \lambda, v \rangle =  \int_{\Omega} f v, \quad \forall v \in H^1(\Omega), \\ 
 \langle \mu,u \rangle =  \langle \mu, g \rangle \qquad  \forall \mu \in  (H^{1/2}(\Gamma))'\label{BabuskaCont2}
 \end{gather}
 where 
 %the bilinear form is defined by
 %$\BB(u,\lambda;v,\mu) = (\nabla u, \nabla v) + (\lambda, v) + (\mu,u)$ and  
 $\langle \cdot, \cdot \rangle$ denotes the duality between $ (H^{1/2}(\Gamma))'$ and $H^{1/2}(\Gamma)$. The variational problem \multeqref{BabuskaCont1}{BabuskaCont2} admits a unique solution $(u, \lambda)$, that satisfies $\lambda = - \nabla u \cdot \n$ , see \cite{babuska}. Discretizing such a problem by a Galerkin approach yields 
  the Lagrange multiplier method, originally proposed by Babuska  \cite{babuska} in the context of the finite element method. It is well known that, %was shown that
  in order for a Galerkin discretization of \multeqref{BabuskaCont1}{BabuskaCont2}  to converge optimally,  
  %the finite element spaces i,.e. 
  the approximation spaces for the unknown $u$  in $\Omega$ and for the multiplier $\lambda$ on the boundary, must satisfy an ``inf-sup" condition. In order to relax such a requirement, it is possible to resort to a stabilized formulation, such as the one proposed by  Barbosa and Hughes in 
  \cite{HughesBarbosa}, which, given $V_h$ and $\Lambda_h$ finite element approximations of $H^1(\Omega)$ and $(H^{1/2}(\Gamma))'$, reads: \emph{
  	find $u_h \in V_h$, $\lambda_h \in \Lambda_h$ such that }
  	\begin{equation*}
  \int_{\Omega} \nabla u_h \cdot \nabla v_h + \int_\Gamma \lambda_h\, v_h + \int_\Gamma\mu_h\, u_h  - 
  	\alpha \sum_{\face \in \Fb} h_\face  \int_\face  \left( \lambda_h + \nabla u_h \cdot \nu \right)  \left( \mu_h +\dn v_h \right) 
  	= \int_{\Omega} f  v_h + \int_\Gamma  g \mu_h,
  	\end{equation*}
  	where we let $\dn w_h$ be a shorthand for $\nabla w_h \cdot \nu$.
  	 Such a formulation allows the approximation spaces $V_h$ and $\Lambda_h$  to be chosen arbitrarily, independently of each other, and  it can be  proven that, provided the stabilization parameter $\alpha > 0$ is small enough, the resulting discrete problem is well posed, and yields convergence with optimal order. We now  want to  extend such a formulation to the VEM.
%
% in problem (\ref{prob_mod_PolyDom}) 
%
%%with $f \in L^2(\Omega)$ and $g \in H^{1/2}(\partial \Omega)$,
%where 
%%Let 
%$\Omega \subset \mathbb{R}^d$, $d=2,3$ is either a smooth bounded domain or a polygon/polyhedron. 
%%, 
%%$f \in L^2(\Omega)$ and $g \in H^{1/2}(\partial\Omega)$.
%
%\
%
%\todo{add definitions of  $H^{1/2}$, norms  etc. }
%We will consider  the space $H^{-1/2}(\Gamma)$, i.e.  the dual space of $H^{-1/2}(\Omega)$, with the norm
%$$
%\norm{\mu}_{-1/2,\Gamma} = \sup_{z\in H^{1/2}(\Gamma)} \frac{\langle\mu,z\rangle}{\norm{z}_{1/2,\Gamma}}
%$$
%and where $\langle\cdot,\cdot\rangle$  denotes the duality pairing.
%
%\
%
%
%
%We then start by considering the following variational formulation of problem  (\ref{prob_mod_PolyDom}): 
%{\em given $f\in L^2(\Omega)$ and
%	$g \in H^{1/2}(\Gamma )$, find 
%	$u \in H^1(\Omega)$ and $\lambda \in H^{-1/2}(\Gamma)$ such that}
%
%\begin{gather}\label{BabuskaCont1}
%(\nabla u, \nabla v) + \langle \lambda, v \rangle =  (f,v), \quad \forall v \in H^1(\Omega), \\ 
%\langle \mu,u \rangle =  \langle \mu,q \rangle \qquad  \forall \mu \in  H^{-1/2}(\Gamma)\label{BabuskaCont2}
%\end{gather}
%where 
%%the bilinear form is defined by
%%$\BB(u,\lambda;v,\mu) = (\nabla u, \nabla v) + (\lambda, v) + (\mu,u)$ and  
%$(\cdot,\cdot)$ denotes the inner product in $L^2(\Omega)$ and $\langle \cdot, \cdot \rangle$ denotes the duality between $H^{-1/2}(\Gamma)$ and $H^{1/2}(\Gamma)$. 

%% file: Nitsche_BH.tex
%!TEX root = VEM_weaklyBC_rev.tex

\subsection{\bf Symmetric Barbosa-Hughes formulation for VEM }\label{sec:BHVEM} 
In order to discretize \multeqref{BabuskaCont1}{BabuskaCont2}, with $V_h$ being the virtual element space defined in \eqref{VEMm}, we  
choose $\Lambda_h$ as a space of discontinuous piecewise polynomials on the tessellation $\Fb$: \[\Lambda_h = \Poly{k'}(\Fb), \qquad k^\prime\in \{k,k-1\}.\]

We propose the following symmetric Barbosa--Hughes virtual element  discretization (BH-VEM):

\
 
 \noindent \emph{find $u_h \in V_h$, $\lambda_h\in \Lambda_h$ such that for all $v_h \in V_h$, $\mu_h\in \Lambda_h$  it holds that}
\begin{multline}\label{BH_VEM}
	a_h(u_h,v_h) + \int_\Gamma \lambda_h\, v_h + \int_\Gamma\mu_h\, u_h \\ - 
	\alpha \sum_{\face \in \Fb} \hKf  \int_\face  \left( \lambda_h + \dn \Pn(u_h) \right)  \left( \mu_h +\dn \Pn(v_h) \right) 
	= \int_{\PolyDom} f \, \Prhs(v_h) + \int_\Gamma  g \mu_h ,
\end{multline}
where
$\Pinabla : H^1(\Omega) \to \pwPoly$, $\Pi^0: L^2(\Omega) \to \pwPoly$ and $\Prhs : H^1(\Omega) \to \Poly{\max\{0,k-2\}}(\Tess)$ are defined as
\begin{gather*}
\Pi^\nabla (\phi)_{|K} = \PinablaK (\phi_{|K}), \qquad
\Pi^0 (\phi)_{|K} = \Pi^{0,k}_K (\phi_{|K}),\\[4mm] \Prhs (\phi)_{|K} = \begin{cases}
|\partial K |^{-1} \int_{\partial K}\phi & k=1,\\[2mm]
\Pi^{0,k-2}_K (\phi) & k \geq 2,
\end{cases}\qquad \forall K \in \Th.
\end{gather*}
The choice of the positive stabilization constant $\alpha$  will be discussed later  on. Observe that the formulation \eqref{BH_VEM} is obtained from the standard symmetric Barbosa--Hughes formulation by replacing the non computable terms (particularly the ones involving the $\partial_\nu$ operator) with suitable approximations involving computable projections, so that all the terms appearing in \eqref{BH_VEM} are computable.
It will be convenient, in the following, to rewrite \eqref{BH_VEM} in compact form as 
\[
\ahBH(u_h,\lambda_h;v_h,\mu_h) = \int_{\PolyDom} f_h\, v_h + \int_\Gamma  g\, \mu_h ,
\]
where the stabilized bilinear form $\ahBH: \Vhglob \times \Vhglob \to \mathbb{R}$ is defined as
	\begin{multline}\label{def:ahBH}
	\ahBH(u_h,\lambda_h;v_h,\mu_h) = 	a_h(u_h,v_h) + \int_\Gamma \lambda_h\, v_h + \int_\Gamma\mu_h\, u_h \\ - 
	\alpha \sum_{\face \in \Fb} \hKf  \int_\face  \left( \lambda_h +\dn \Pn(u_h) \right)  \left( \mu_h +\dn \Pn(v_h) \right),
	\end{multline}
where, for $f \in L^2(\Omega)$,  $f_h \in H^{-1}(\Omega)$ is defined by
\[
\int_{\Omega} f_h \phi = \int_{\Omega} f \Prhs \phi.
\]
Observe that, as for $k > 1$ we have that $\Prhs$ is selfadjont, in such a case we have that $f_h = \Prhs f$.

\

In order to analyze Problem \eqref{BH_VEM} we start by introducing the following mesh dependent norms on $\Gamma$:
\begin{gather}\label{defnormhG}
\| \lambda  \|^2_{-1/2,h} = \sum_{\face\in \Fb} \hKf \| \lambda \|_{0,\face}^2, \qquad \| \phi  \|^2_{1/2,h} = \sum_{\face\in \Fb} \hKf^{-1} \| \phi \|_{0,\face}^2.
\end{gather}
Observe that it holds that
\begin{equation}\label{conth}
\int_\Gamma \lambda\, \phi \lesssim \| \lambda \|_{-1/2,h} \| \phi \|_{1/2,h}, \qquad \forall \lambda, \phi \in L^2(\Gamma).
\end{equation}

We now let $\Pib: L^2(\Gamma) \to \Lambda_h$ denote the $L^2(\Gamma)$ orthogonal projection onto the space of discontinuous  piecewise polynomials of order less than or equal to $k'$, and we observe that  we have
\[
\Pib(u)_{|\face} = \Pi^{0,k'}_\face(u_{|\face})\qquad \forall \face \in \Fb.
 \] 
We introduce a mesh dependent semi norm on the space
\(H^1(\Tess) := \{u \in L^2(\Omega):\ u|_K \in  H^1(K)\}\) of discontinuous piecewise $H^1$ functions, defined as
\begin{equation}\label{normvsnormh}
\| u \|_{1,h}^2 = | u |^2_{1,h} + \| \Pib u \|^2_{1/2,h}\quad \text{ with } \quad | u |_{1,h}^2 = \sum_K | u |_{1,K}^2.
\end{equation}

Since the function identically assuming the value $1$ on $\Gamma$ belongs to $\Lambda_h$, we have that $\int_\Gamma  u =  \int_{\Gamma} \Pib u$. Letting $\bar u = | \Gamma |^{-1} \int_{\Gamma} u = | \Gamma |^{-1} \int_{\Gamma} \Pib u$ denote the average on $\Gamma$ of $u \in H^1(\Omega)$, it is then not difficult to see  that $| \bar u | \lesssim \| \Pib u \|_{0,\Gamma} \lesssim \| \Pib u \|_{1/2,h}$ (the implicit constant in the inequalities depending on $|\Gamma|$). We can then write
\begin{equation}\label{defnormhO}
\| u \|_{1,\Omega} \lesssim 
\| u -  \bar u\|_{1,\Omega} 
+ \| \bar u \|_{1,\Omega} \lesssim | u |_{1,\Omega} 
+ \|\bar u \|_{0,\Omega} = | u |_{1,\Omega} + | \Omega|^{1/2} | \bar u | \lesssim 
\| u \|_{1,h},
\end{equation}
where we used a Poincar\'e Friedrichs inequality on the space of $H^1(\Omega)$ functions with zero average on $\Gamma$, and the fact that the $H^1(\Omega)$ seminorm of the constant function $\bar u$ is zero. Then $\| \cdot \|_{1,h}$ is a norm on $H^1(\Omega)$.

\

We prove the following Lemma.
\begin{lemma}\label{prop:3.1} For all $u_h, v_h \in V_h$, $\lambda_h, \mu_h \in \Lambda_h$ it holds that
	\begin{equation}\label{contBH}
	\ahBH (u_h,\lambda_h;v_h, \mu_h) \lesssim (	\|u_h \|_{1,h} + \| \lambda_h \|_{-1/2,h})(
		\|v_h \|_{1,h} + \| \mu_h \|_{-1/2,h}
	).
	\end{equation}
Moreover,	there exists $\alpha_0$ such that, if $\alpha < \alpha_0$, then for all $u_h \in V_h$, $\lambda_h \in \Lambda_h$ it holds that
	\begin{equation}\label{stabBH}
	\sup_{
		(v_h,\mu_h) \in V_h\times \Lambda_h} \frac{\ahBH(u_h,\lambda_h;v_h,\mu_h)}{
		\|v_h \|_{1,h} + \| \mu_h \|_{-1/2,h}
		} \gtrsim 	\|u_h \|_{1,h} + \| \lambda_h \|_{-1/2,h}.
	\end{equation}
	\end{lemma}

\begin{proof} 
	By the definition of $\Pib$, and using \eqref{conth}, we can write
	\[
\int_\Gamma \lambda_h\, v_h  = \sum_\face \int_{\face} \lambda_h v_h = \sum_{\face }
	\int_\face \lambda_h \Pif (v_h)
	\leq \| \lambda_h \|_{-1/2,h} \| \Pib (v_h) \|_{1/2,h}.
	\]
	Moreover we can write
	\[
	\sum_{\face \in \Fb} \hKf  \int_\face  \left( \lambda_h +\dn \Pn(u_h) \right)  \left( \mu_h +\dn \Pn(v_h) \right) \lesssim \|  \lambda_h +\dn \Pn(u_h) \|_{-1/2,h} \| \mu_h +\dn \Pn(v_h) \|_{-1/2,h}.
	\]
	In view of the definition \eqref{defhtildef} of $\hKf$, thanks to \eqref{inversenormal}, we easily see that
\begin{equation}\label{bounddn}
\| \dn \Pn(u_h) \|^2_{-1/2,h}   \leq \sum_K \hel \| \nabla \Pn(u_h) \|_{0,\partial K}^2 \lesssim | \Pn(u_h) |_{1,h}^2 \lesssim | u_h |_{1,h}^2.
\end{equation}
The remaining terms are naturally bound in the considered norm and then \eqref{contBH} is easily proven.

\

In order to prove \eqref{stabBH}, we set at first $v_h = u_h$ and $\mu_h = - \lambda_h$. This gives the following bound, where we use \eqref{defSK} and \eqref{inversenormal}:
		\begin{multline}\label{coerc1}
		%	row 1
	\ahBH(u_h,\lambda_h;u_h,-\lambda_h) 	=\\	
%	a_h(u_h,u_h) + \int_\Gamma \lambda_h u_h - \int_\Gamma\lh u_h 
%%		%
%%		\sum_{\face} \int_{\face } \frac{\delta^j}{j!} \partial^j_{\sigma}
%%		(\Pinabla u_h)\lh
%		-	\alpha \sum_{\face \in \Fb} \hKf \int_\face \left|\nabla \Pn(u_h) \right|^2
%		+\alpha \sum_{\face \in \Fb} \hKf \int_\face |\lh|^2 \geq  
		% Row 2
	a_h(u_h,u_h) + \alpha \sum_{\face \in \Fb} \hKf \int_\face |\lh|^2  -
		\alpha \sum_{\face \in \Fb} \hKf \int_\face \left| \dn \Pn(u_h) \right|^2 
%		-	
%		\sum_{\face} 	\sum_j \int_{\face } \frac{\delta^j}{j!} \partial^j_{\sigma}	(\Pinabla u_h)\lh
 \geq \\
		\snorm{\Pinabla (u_h)}^2_{1,h} +c_* \snorm{(1-\Pinabla)u_h}^2_{1,h}
		+ \alpha \| \lh \|_{-1/2,h}^2  - \alpha \Cinv^2 \snorm{\Pinabla (u_h)}^2_{1,\Omega}  \\[2mm]
	=	(1-\alpha \Cinv^2) \snorm{\Pinabla (u_h)}^2_{1,h}  + c_* \snorm{(1-\Pinabla)u_h}^2_{1,h} + \alpha \| \lh \|^2_{-1/2,h}.
		\end{multline}

		Setting $v_h = 0$ and $\bar \mu_h|_\face = \hKf^{-1}  \Pif u_h$, and observing that $\| \bar \mu_h \|_{-1/2,h} = \| \Pib u_h \|_{1/2,h}$,  yields instead
		\begin{multline}\label{partialbound}
		\ahBH(u_h,\lambda_h;0,\bar \mu_h) =
	\int_{\Gamma} u_h \bar \mu_h - \alpha \sum_{\face} \hKf \int_{\face} (\lambda_h + \dn \Pn (u_h)  ) \bar \mu_h =\\
	\| \Pib u_h \|^2_{1/2,h} - \alpha \sum_{\face} \int_{\face} (\lambda_h + \dn \Pn (u_h)  )	\Pif(u_h)
		\geq \\
		 \| \Pib (u_h) \|_{1/2,h}^2 - \alpha \| \lambda_h \|_{-1/2,h} \| \Pib(u_h)
\|_{1/2,h} - \alpha \|  \dn \Pn (u_h) \|_{-1/2,h} \| \Pib(u_h) \|_{1/2,h} \geq \\[2mm]
	 \| \Pib (u_h) \|_{1/2,h}^2 - \alpha \| \lambda_h \|_{-1/2,h} \| \Pib(u_h)
	 \|_{1/2,h} - \alpha \Cinv |  \Pn (u_h) |_{1,h} \| \Pib(u_h) \|_{1/2,h},
		\end{multline}
		which, by Young's inequality, gives us
		\[
			\ahBH(u_h,\lambda_h;0,\bar \mu_h) \geq  \| \Pib (u_h) \|_{1/2,h}^2 (1-\alpha) - \frac \alpha 2 \| \lambda_h \|^2_{-1/2,h}  -  \frac{\alpha}{2} \Cinv^2 | \Pn(u_h) |_{1,h}^2.
		\]
		
Then, 
\begin{multline*}
\ahBH(u_h,\lambda_h;u_h,\bar\mu_h-\lambda_h) \geq \\
(1-\frac 3 2 \alpha \Cinv^2) | \Pn(u_h)  |_{1,h}^2 +  c_* \snorm{(1-\Pinabla)u_h}^2_{1,h} + \frac \alpha 2 \| \lambda_h \|_{-1/2,h}  + (1-\alpha ) \| \Pib(u_h) \|_{1/2,h}.
\end{multline*}		
		If $\alpha < \alpha_0 = \min \{ 1,  2/(3\Cinv^2) \}$, the coefficients of all the terms on the right hand side are strictly positive, finally yielding
	\begin{multline*}
	\ahBH(u_h,\lambda_h;u_h,\bar\mu_h-\lambda_h) \gtrsim \| u_h \|_{1,h}^2 + \| \lambda_h \|_{-1/2,h}^2\\[1.5mm] \gtrsim (
	\| u_h \|_{1,h} + \| \lambda_h \|_{-1/2,h}
	)(
	\| u_h \|_{1,h} + \| \bar\mu_h-\lambda_h \|_{-1/2,h}
	),
		\end{multline*}
		where we used that $\| \bar \mu_h \|_{-1/2,h} \lesssim \| u_h \|_{1,h}$.
The implicit constant in the inequality depends, of course, on $\alpha$. The desired bound is obtained by dividing both sides by  $(
\| u_h \|_{1,h} + \| \bar\mu_h-\lambda_h \|_{-1/2,h}
)$.	
\end{proof}

\begin{remark}
To ensure computability, in formulating the Barbosa--Hughes method for VEM,
	 the normal derivative of the discrete test and trial functions is replaced with the normal derivative applied to the polynomial part obtained via the projector $\Pn$. Rather than being an obstacle to the analysis, this approximation does actually help in attaining stability. Indeed, it allows to use the inverse inequality \eqref{inversenormal}, which, under our assumption on the mesh, we know to hold for polynomials, thus avoiding the need of proving it for general virtual element functions. More importantly, it yields a constant $\alpha_0$ which is independent of the constants $\cstar$ and $\Cstar$ (and therefore, independent of the choice of the stabilization for the VEM bilinear form).
	\end{remark}

Thanks to Proposition \ref{prop:3.1}, the following Theorem holds.

\begin{theorem}\label{thm:BH} There exists $\alpha_0$ such that if $\alpha < \alpha_0$,  Problem \eqref{BH_VEM} admits a unique solution $(u_h,\lambda_h)$ satisfying the following error estimate: if $u \in H^{k+1}(\Omega)$ then 
	\[
	\| u - u_h \|_{1,h} + \| \lambda-\lambda_h \|_{-1/2,h} \lesssim h^k | u |_{k+1,\Omega}.
	\]
\end{theorem}

\begin{proof} Existence and uniqueness of the solution $u_h$ follow, by standard arguments, from Proposition \ref{prop:3.1} (see, e.g., \cite{BoffiBrezziFortin}).
Let then $\uI  \in \Vh$ denote the VEM interpolant of $u$, and let 
$u_\pi = \Pi^0 (u)\in \Poly{k} (\Tess)$ and 
$\lI = \Pib (\lambda )\in \Lh$ denote, respectively the $L^2$ projections of $u$ onto $\Poly k(\Tess)$ and of $\lambda = \dn u$ onto $\Lambda_h$. Thanks to \eqref{stabBH}, there exist $\vh\in \Vh$, $\mh \in \Lh$, such that 
\[
\| \uh - \uI  \|_{1,h} + \| \lh - \lI \|_{-1/2,h}  \lesssim \frac{
	\ahBH(
 \uh - \uI ,\lh - \lI ; v_h, \mu_h
	)
	}{
	\| \vh \|_{1,h} + \| \mh \|_{-1/2,h}
	}.
\]
After possibly renormalizing $v_h$ and $\mu_h$, we can always assume that
\begin{equation}\label{normalization}\| \vh \|_{1,h} + \| \mh \|_{-1/2,h} = 1.\end{equation} 

Now we can write
\begin{multline*}\ahBH(\uh - \uI ,\lh - \lI ; v_h,\mu_h
)\\ =
a_h(\uh - \uI, \vh) + \int_{\Gamma} (\lh - \lI)\, \vh + \int_{\Gamma} (\uh - \uI)\, \mu_h
\\ - \alpha \sum_\face \hKf \int_\face (\lh - \lI + \dn \Pn(\uh - \uI))(\mh + \dn \Pn(\vh)) 
\\
\pm a_h(u_\pi,\vh) \pm a(u,\vh) \pm \int_{\Gamma} \lambda \vh \pm \int_{\Gamma } \mh u \mp 
\alpha \sum_\face \hKf \int_\face (\lambda + \dn \Pn(u)(\mh + \dn \Pn(\vh)) .
\end{multline*}

 Using the $k$ consistency property \eqref{mconsistency} on one of the terms $a_h(u_\pi,v_h)$, the identity $\lambda = \dn u$
 	and using \eqref{BabuskaCont1}, \eqref{BabuskaCont2} and \eqref{BDT_BH_VEM}, 
we can then write
\begin{multline}
\| \uh - \uI  \|_{1,h} + \| \lh - \lI \|_{-1/2,h}  \lesssim \ahBH(\uh - \uI ,\lh - \lI ; v_h,\mu_h
)\\[2mm] 
%\rosso{\pm a_h(u_\pi,\vh) \pm a(u,\vh) \pm \int_{\Gamma} \lambda \vh \pm \int_{\Gamma } \mh u \mp 
%\alpha \sum_\face \hKf \int_\face (\lambda + \dn \Pn(u)(\mh + \dn \Pn(\vh)) } \\
= \int_{\Omega } (f_h - f)\, \vh + a_h(u_\pi - u_I, v_h) + \sum_K\int_{K}\nabla(u - u_\pi)\cdot\nabla v_h
+ \int_{\Gamma} (\lambda - \lI)\, \vh + \int_{\Gamma} (u - \uI)\, \mu_h 
\\
- \alpha \sum_\face \hKf \int_\face (\lambda - \lI + \dn \Pn(u - \uI))(\mh + \dn \Pn(\vh)) 
\\ + 
\alpha \sum_\face \hKf \int_\face (\dn \Pn(u) - \dn u)(\mh + \dn \Pn(\vh)) = A + B + C + D + E + F + G.
\end{multline}

 Observe that the normalization \eqref{normalization} of $\vh$ and $\mh$ implies that both $\| \vh \|_{1,h}$ and $\| \mh \|_{-1/2,h}$ are less than or equal to $1$.
 The terms $A$, $B$ and $C$ are the ones that are usually encountered in the analysis of the VEM (see \cite{basicVEM}), and the following bounds hold:
 \begin{equation}
 \label{boundABC}
 A \lesssim h^k | f |_{k-1}, \qquad B + C \lesssim h^k | u |_{k+1,\Omega}.
 \end{equation}
%  and then we can bound the five terms on the right hand side as follows. As $f_h - f$ is orthogonal to $\Poly{0}(
% \Tess)$ ($\subset \Poly{k-2}(\Tess)$), we have, with $\bar v^K = | K |^{-1} \int_K \vh$,
%\begin{equation}\label{boundA}
%A = \sum_K \int_{K} (f_h - f) (\vh - \bar v^K) \leq \sum_K h_K \| f_h - f \|_{0,K} | \vh |_{1,K} \leq h \| f_h - f \|_{0,\Omega} | v_h |_{1,\Omega} 
%\end{equation}

We individually bound the remaining terms. Since, for $k \geq 1$, $u \in H^{k+1}(\Omega)$ implies that $\lambda \in H^{k-1/2}(\face)$ for all $\face \in  \Fb$, with 
\[
\sum_{\face \in \Fb \cap \FK} | \lambda |^2_{k-1/2,\face} \lesssim | u |^2_{k+1,K},
\]	
using \eqref{approxjface} we easily see that
\begin{equation}\label{boundD}
D \leq \| \lambda - \lI \|_{-1/2,h} \lesssim h^{1/2} \| \lambda - \lI \|_{0,\Gamma} \lesssim h^k | u |_{k+1}.
\end{equation}
Moreover we have that
\begin{equation}\label{boundE}
E \leq \| u - \uI \|_{1/2,h} \lesssim h^k | u |_{k+1,\Omega},
\end{equation}
and, using \eqref{inversenormal}
\begin{multline}\label{boundF}
	F \lesssim \| \lambda - \lI \|_{-1/2,h} + \| \dn \Pn (u - \uI) \|_{-1/2,h} \lesssim \| \lambda - \lI \|_{-1/2,h} + \|  \Pn (u - \uI) \|_{1,h} \\[2mm]
	\lesssim \| \lambda - \lI \|_{-1/2,h} + \|  u - \uI \|_{1,\Omega} \lesssim h^k | u |_{k+1,\Omega}.
	\end{multline}
	Finally, using \eqref{eq:approxdn}  we bound $G$ as
	\begin{multline}\label{boundG}
	G^2 \leq \| \dn(\Pn(u) -  u) \|^2_{-1/2,h} = \sum_K \sum_{\face\in \Fb\cap \FK} \hKf \| \dn(\Pn(u) - u) \|_{0,\face}^2 \\ \leq h \sum_K  \| \dn(\Pn(u) - u) \|^2_{0,\partial K}
	 \lesssim 
	 h^{2k} | u |^2_{k+1,\Omega}.
	\end{multline}
	Combining \crefrange{boundABC}{boundG} and using a triangular inequality  we obtain the desired bound.
\end{proof}

In view of \eqref{normvsnormh} we immediately obtain the following corollary.
\begin{corollary}\label{cor:errorH1}
	Provided $\alpha < \alpha_0$, $\alpha_0$ given by  Theorem \ref{theo:error_BH}, we have that
	\[
	\| u - u_h \|_{1,\Omega} \lesssim h^k | u |_{k+1}.
	\]
	\end{corollary}

\begin{remark}
For the sake of simplicity, we chose the definition \eqref{defspace2} for the local VE discretization space $\VKp$, which yields a space $V_h$ for which the projectors $\Pi^{0,k}_K$ are not computable, whence the need of introducing the somewhat cumbersome definition of the projector $\Prhs$ for computing the action of the source term on VE functions. With a similar approach similar to the one used for defining $\Vfp$, it is however possible to define $\VKp$, in such a way that $\Pi^{0,k}_K$ is computable for all $k \geq 1$. For the resulting enhanced space $V_h$, we can then set $\Prhs$ to be the $L^2$ projection  onto $\Poly{k}(\Tess)$. In view of the results in \cite{3DVEM}, our analysis holds unchanged, and  also for such a definition of the VE discretization space and right hand side, Theorem \ref{thm:BH} and Corollary \ref{cor:errorH1} hold.
	\end{remark}

\subsection{\bf Obtaining Nitsche's method from BH--VEM}

In the finite element framework
it can be shown (see \cite{Stenberg.1995}) that, by suitably choosing $\Lambda_h$, and eliminating the multiplier $\lambda_h$ by static condensation, the popular Nitsche's method for imposing Dirichlet boundary conditions can be retrieved. A similar thing also happens in the virtual elements framework. Indeed, setting $v_h=0$ in (\ref{BH_VEM}), we get the equation 
	\begin{gather*}
 \int_\Gamma\mu_h\, u_h  - 
\alpha \sum_{\face \in \Fb} \hKf  \int_\face  \left( \lambda_h +\dn \Pn(u_h) \right)  \mu_h = 
\int_{\Gamma} g\, \mu_h.
\end{gather*}
This equation can be solved for $\lambda_h$, face by face, and, setting $\gamma = \alpha^{-1}$, we obtain
\begin{equation}\label{eq_lambda}
\lambda_h = \gamma  \hKf^{-1} 
\Pib(u_h - g) - \dn \Pn(u_h).
\end{equation}

If we substitute \eqref{eq_lambda} in \eqref{BH_VEM}, and set $\mu_h=0$ we then get
\begin{multline*}
	a_h(u_h,v_h) + \gamma\sum_{\face \in \Fb}\hKf^{-1}  \int_\face \Pib(u_h)\, v_h  - \sum_{\face \in \Fb} \int_\face \dn \Pn(u_h)\, v_h  - \sum_{\face \in \Fb} \int_\face \Pib(u_h) \, \dn \Pn(v_h) \\
	= 
	\int_{\PolyDom} f_h v_h + \gamma\sum_{\face \in \Fb}\hKf^{-1} \int_\face g_h v_h - \sum_{\face  \in \Fb} \int_\face  g_h \dn \Pn(v_h),
\end{multline*}
where $g_h = \Pib (g)$.
As $\dn \Pn (v_h)$ is a polynomial of degree $k-1\leq k'$ on each face $\face\in \Fb$, using  the definition of $\Pib$ allows us to rewrite the above equation in a form that underlines its symmetry,  namely
\begin{multline}\label{nitsche}
a_h(u_h,v_h) - \sum_{\face \in \Fb} \int_\face \dn \Pn(u_h)\, v_h  - \sum_{\face \in \Fb} \int_\face u_h \, \dn \Pn(v_h)  + \gamma\sum_{\face \in \Fb}\hKf^{-1}  \int_\face \Pib(u_h) \Pib(v_h)   \\
= 
\int_{\PolyDom}  f_h v_h + \gamma\sum_{\face \in \Fb}\hKf^{-1} \int_\face g_h v_h - \sum_{\face  \in \Fb} \int_\face  g_h \dn \Pn(v_h). 
\end{multline}
Conversely, if $u_h$ satisfies \eqref{nitsche} and $\lambda_h$ is defined by \eqref{eq_lambda}, $(u_h,\lambda_h)$ satisfies \eqref{BH_VEM}. Therefore the analysis  obtained in the previous section for the Barbosa-Hughes problem (\ref{BH_VEM}) can also be extended to the Nitsche VEM problem (\ref{nitsche}), as stated by the following corollary of Theorem \ref{thm:BH}, where, once again, we use \eqref{defnormhO} to bound the $H^1(\Omega)$ norm by the mesh dependent norm $\| \cdot \|_{1,h}$.

\begin{corollary}\label{cor:nitsche} There exists $\gamma_0$ such that if $\gamma > \gamma_0$,  Problem \eqref{nitsche} admits a unique solution $u_h$ satisfying the following error estimate: if $u \in H^{k+1}(\Omega)$ then 
	\[
	\| u - u_h \|_{1,\Omega}  \lesssim h^k | u |_{k+1,\Omega}.
	\]
	\end{corollary}

	\begin{remark} Observe that,  in two dimensions, for $k'=k$, we have that, for $u_h \in \Vh$, $\Pib(u_h) = u_h$. Then, \eqref{nitsche} can be rewritten as
	\begin{multline}\label{nitsche2d}
	a_h(u_h,v_h) - \sum_{\face \in \Fb} \int_\face \dn \Pn(u_h)\, v_h  - \sum_{\face \in \Fb} \int_\face u_h \, \dn \Pn(v_h)  + \gamma\sum_{\face \in \Fb}\hKf^{-1}  \int_\face u_h v_h   \\
	= 
	\int_{\PolyDom} f_h v_h + \gamma\sum_{\face \in \Fb}\hKf^{-1} \int_\face g_h v_h - \sum_{\face  \in \Fb} \int_\face  g_h \dn \Pn(v_h),
	\end{multline}	 
which is the formulation originally proposed and analyzed in \cite{VEM_curvo}.
		\end{remark}

\begin{remark} In three dimensions, 
	the formulation \eqref{nitsche} of Nitsche's method for the VEM,  obtained from  \eqref{BH_VEM} by static condensation of the multiplier $\lambda_h$, coincides with the formulation that we would obtain by adapting the standard Nitsche's method to the VEM, by replacing the non computable terms with computable terms involving the projection $\Pinabla$, in the spirit of the virtual element method. 
	\end{remark}

%% file: BH_BDT_VEM.tex
%!TEX root = VEM_weaklyBC_rev.tex

\renewcommand{\epsilon}{\varepsilon}
\newcommand{\dnh}{\partial_{\nu_h}}
\newcommand{\tg}{g^\star}

\newcommand{\nuhx}{\nu_h}
\newcommand{\sx}{\sigma}

\newcommand{\supd}{\delta_h}

\section{Weakly imposing boundary conditions on domains with curved boundary}\label{sec:curvo}

%\section{The Virtual Element Method on domains with curved boundary}\label{sec:curvo}

Let us now consider the solution of the same model problem
\begin{equation}\label{prob_mod}
	-\Delta u = f, \text{ in }\Omega, \qquad u = g, \text{ on }\partial \Omega,
\end{equation}
with, once again, $f \in L^2(\Omega)$, $g \in H^{1/2}(\partial \Omega)$, where now $\Omega \subseteq \mathbb{R}^d$, $d=2,3$, is a bounded domain with  $C^1$ boundary $\Gamma = \partial\Omega$. 
In order to solve such a problem by the Virtual Element method, 
we assume that $\Omega$ is approximated by a family of polygonal/polyhedral  domains $\Omh$, $0 < h \leq 1$,  each endowed with a quasi uniform tessellation $\Th$  into 
polyhedrals $K$ with diameter $h_K \simeq h$, and we follow the strategy proposed in \cite{BDT}, and already applied to the virtual element method in the framework of Nitsche's method in two dimensions (\cite{VEM_curvo}). The idea is to solve a discrete problem on $\Oh$, satisfying, on $\partial\Oh$ a modified boundary condition that takes  into account that the boundary data is given on $\partial \Omega$ rather than on $\partial \Oh$. Once again we will focus on the more complex three dimensional case, but the results obtained also hold in two dimensions with minor modifications.

\

Consider then a family $\{\Oh \}_h$, with $\Oh\subset \Omega$, of polyhedral domains approximating $\Omega$ and let $\Th$ denote a quasi uniform tessellation of $\Oh$ in polyhedral elements. We assume that the family $\{\Tess\}_h$ thus obtained satisfies Assumption \ref{shape_regular}. To avoid pathological situations we also assume that the measure of $\partial\Oh$ does not explode as $h$ goes to $0$, that is, that $| \partial \Oh | \lesssim | \Gamma|$. We let $\Fb$ denote the set of faces lying on the boundary $\Gamma_h = \partial\Oh$ of the approximating domain. We let $\n_h$ denote the outer unit normal to $\Gamma_h$. On $\Gamma_h$ we also choose an outward direction $\sigma$, not necessarily normal to $\Gamma_h$, which we assume to be constant on each face $\face$, and for $u \in C^m(\bar K)$, we let $\partial_\sigma^m$ denote the $m$-th derivative in the direction $\sigma$.
For 
$x \in \Gamma_h$ we let $\delta(x) > 0$ denote the smallest non negative scalar such that 
\[
x + \delta(x) \sx(x) \in \partial \Omega.\]

We recall that, for smooth convex domains in two dimensions, provided that the vertices of  $\Omega_h$ are placed on $\Gamma$, and choosing $\sigma = \nu_h$, for $x \in \partial K \cap \Gh$ we have that $\delta(x) \lesssim h_K^2$. We can not rely on a similar bound in three dimensions, as, in general, even for convex domains, building a polyhedral mesh with all the vertices on $\Gamma$ is not possible,
unless we either restrict ourselves to (possibly agglomerated) tetrahedral meshes, or we allow for curved  faces and, possibly, for curved edges. To overcome this issue, we need to make an assumption ensuring that the discrete boundary $\Gh$ is sufficiently close to the true boundary $\Gamma$. More precisely, we assume that for some  $\tau \in (0,1)$ sufficiently small, for all tessellations in the family $\{ \Th \}_h$ we have that
	\begin{equation}\label{deftau}
\max_{\face\in \Fb}\,	 \max_{x\in \face} \frac{\delta(x)}{\hKf} \leq \tau.
	\end{equation}
Of course, if the tessellations satisfy $\delta|_{\partial K \cap \Gamma_h} \lesssim h_K^2$, then for all $\tau$, there exists an $h_0$ such that for all $h < h_0$, \eqref{deftau} is satisfied. However, in three dimensions, it is desirable to allow  for situations where $\delta$ goes to $0$ only linearly in $h$, as this happens in many situations that are encountered in practice. 

\

\begin{remark}
	We want to point out from the start that  the assumption that the boundary $\Gamma$ of $\Omega$ is of class $C^1$ is, at the same time, too restrictive and too weak. Of course, as we will see in the following, optimal order $k$ convergence will require sufficient smoothness of the continuous solution, which can be deduced from the smoothness of the data only if the domain is sufficiently smooth, which is the typical situation that we have in mind in designing the method. However, the definition of the method itself and its theoretical analysis only require  that $\Gamma$ has enough regularity to ensure that the function $\delta$ is well defined on $\Gamma_h$ and that it satisfies $\delta \in L^{\infty}(\Gamma_h)$. We can easily see that $\Gamma$ being of class $C^1$ is not a necessary condition for this to happen.  On the other hand, a higher smoothness of the function $\delta$ (and therefore of the boundary $\Gamma$) will play a role in the practical implementation of the method, as we will see later on (see Remark \ref{rem:deltasmooth} in Section \ref{sec:expes}). 
	\end{remark}

%
%now defined on $\Oh$. More precisely we set
%\begin{gather}
%\| \phi \|^2_{1/2,h} = \sum_{\face \in \Fb} \hKf^{-1} \| \phi \|_{0,\face}^2, \qquad \| \phi \|^2_{-1/2,h} = \sum_{\face \in \Fb} \hKf^{1} \| \phi \|_{0,\face}^2, \\ | \phi |^2_{1,*} = \sum_{K \in \Tess} | \phi |_{1,K}^2, \qquad \| \phi \|^2_{1,h} = | \phi |^2_{1,*} + \| \Pib \phi \|^2_{1/2,h}
%\end{gather}

 \begin{figure}
	\centering
	\includegraphics[height=4cm]{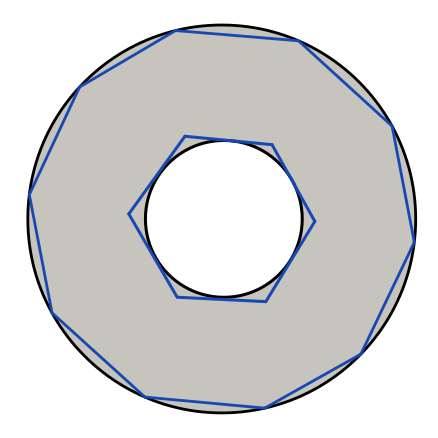}\qquad
		\includegraphics[height=4cm]{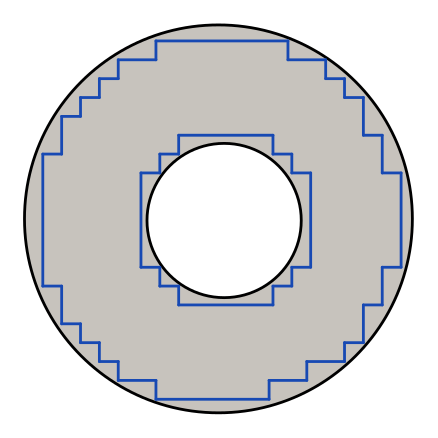}
			\includegraphics[height = 4cm]{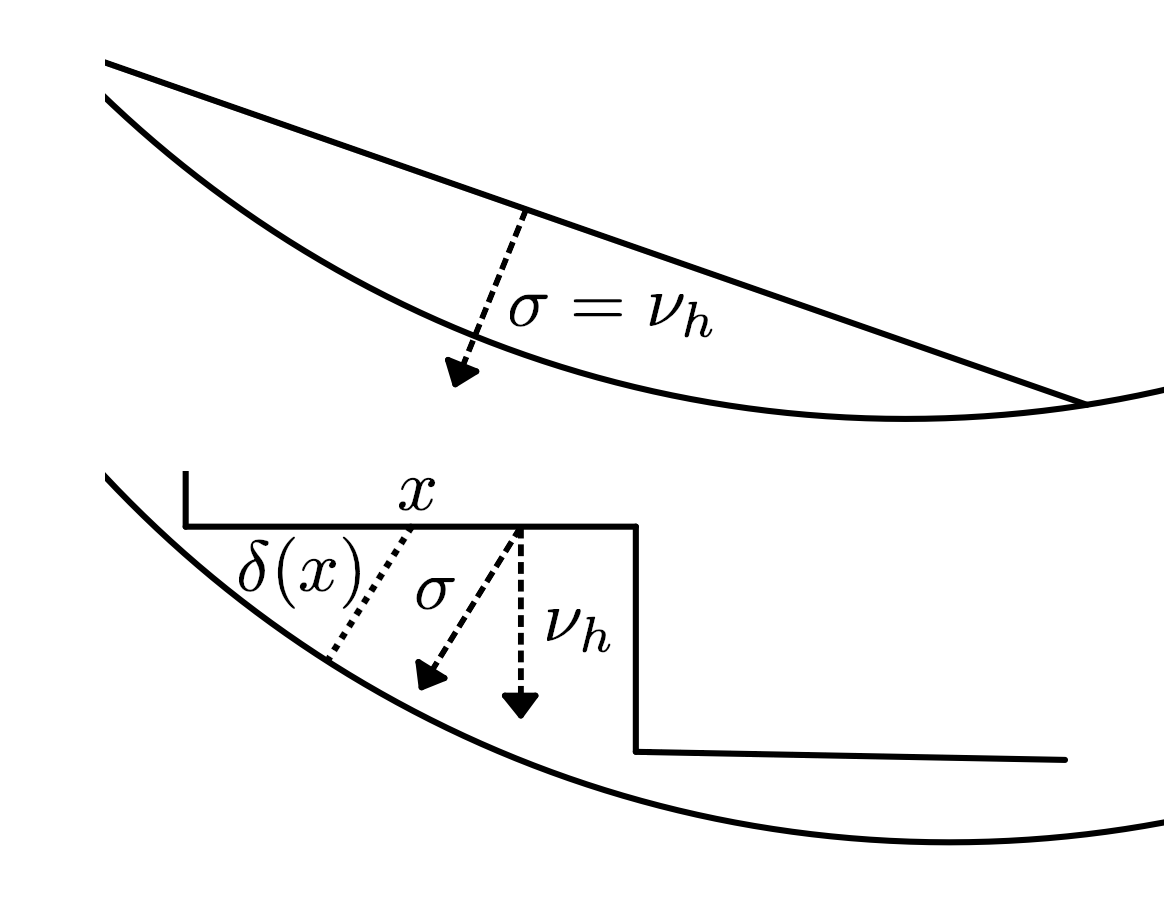}
				\caption{Two approximate domains $\Omega_h$ belonging to families falling in  our framework. For the domain on the left, the choice $\sigma = \nu_h$ yields $\delta(x)  \lesssim h^2$. For the choice in the center choosing 
$\sigma \neq \nu_h$, 
yields smaller values for $\delta(x)$. }\label{fig:sigma}
\end{figure}

\subsection{\bf VEM Barbosa-Hughes  for curved domains (BDT-BH-VEM)}
%We introdiuce the bilinear form
%\underline{BDT-BH-VEM}: 
%\begin{gather}\label{bilinear_BDT_BH}
%	\Ch(\varphi,\phi,\lambda,\mu) =
%		a_h(\varphi,\phi) + \int_\Gamma \lambda \phi + \int_\Gamma\mu \varphi
%	+ \sum_{\face} 	\sum_j \int_{\face } \frac{\delta^j}{j!} \partial^j_{\sigma}
%(\Pinabla \varphi)\mu \\ \nonumber
%	-\alpha \sum_{\face \in \Fb} \int_\face  \left( \lambda +\dn \Pn(\varphi) \right)  \left( \mu+\dn \Pn(\phi) \right) 
%%
%%	- \sum_{\face\in \Fb} \int_\face \dn \Pn(\varphi)\, \phi  - \sum_{\face\in \Fb} \int_\face \varphi \, \dn \Pn(\phi) .
%\end{gather}
%where $\alpha$ is a positive constant and $\Fb$ denotes the set of faces of $\Th$ lying on $\partial\Omega$.

We write down a discretization of \eqref{prob_mod} on the approximate domain $\Oh$. 
The global norms $\| \cdot \|_{1/2,h}$, $\| \cdot \|_{-1/2,h}$ and $\| \cdot \|_{1,h}$, still given by \eqref{defnormhG} and \eqref{normvsnormh}, 
%\eqref{defnormhO}, 
are now defined on $\Gamma_h$ and $\Oh$, as are the discrete spaces $V_h$, $\Poly{k}(\Tess)$ and $\Lambda_h = \Poly{k'}(\Fb)$, as well as the projectors $\Pinabla$, $\Prhs$ and $\Pib$. The bilinear form $a_h$ is also defined on $\Oh$, where
we consider the following discrete problem: 

\

\noindent {\em find $u_h \in V_h$, $\lambda_h\in \Lambda_h$ such that for all $v_h \in V_h$, $\mu_h\in \Lambda_h$  it holds that}
\begin{multline}\label{BDT_BH_VEM}
	a_h(u_h,v_h) + \int_{\Gamma_h} \lambda_h\, v_h + \int_{\Gamma_h}\mu_h\, u_h
	+ \sum_{\face\in \Fb} \sum_{j=1}^{k^*} \int_{\face } \frac{\delta^j}{j!} \partial^j_{\sigma}
	(\Pinabla u_h)\mu_h \\
	-\alpha \sum_{\face \in \Fb} \hKf \int_\face  \left( \lambda_h +\dn \Pn(u_h) \right)  \left( \mu_h +\dn \Pn(v_h) \right) 
%\begin{equation}\label{BDT_BH_VEM}
%	\Ch(u_h,v_h) 
	= \int_{\Oh} f_h\, v_h + \sum_{\face  \in \Fb} \int_\face  \widetilde g \, \mu_h ,
\end{multline}
where $\widetilde g$ is defined as
\[\widetilde g(x) = g(x+\delta(x) \sigma), \qquad x\in \Gamma_h,\] and
  where $k^* \leq k$ is a parameter whose choice we will discuss later on. Also in this case we can write the problem in a more compact form as: find $(u_h,\lambda_h) \in V_h \times \Lambda_h$ such that for all $(v_h,\mu_h) \in V_h \times \Lambda_h$ it holds that
  \[
  \ahBH(u_h,\lambda_h;v_h,\mu_h) + \BDT(u_h,\mu_h) = \int_{\Oh} f_h v_h + \sum_{\face  \in \Fb} \int_\face  \widetilde g\, \mu_h,
  \]
  	where, as before, $\ahBH$ is defined by \eqref{def:ahBH}
  		and where
  		\[  		
  		\BDT(u_h,\mu_h) = \sum_{\face\in \Fb} \sum_{j=1}^{k^*} \int_{\face } \frac{\delta^j}{j!} \partial^j_{\sigma}
  		(\Pinabla u_h)\mu_h.\]
  
 \begin{remark} We observe that, if we neglect the stabilization term, the equation on $\Gamma_h$ obtained by setting $v_h =0$ in \eqref{BDT_BH_VEM}, is an approximation, obtained by a Taylor expansion of order $k^*$, of
 \[
 u(x + \delta \sigma) = g(x + \delta \sigma), \qquad x \in \Gamma_h,
 \]
 which, as $x + \delta \sigma \in \Gamma$, is the correct boundary condition.
 	\end{remark}

\subsection{Analysis of the BDT-BH-VEM method}  We start by proving the following lemma.
\begin{lemma}\label{lem:coercivity_BDT_BH} There exists a constant $\CBDT$, depending on $k$, such that for all $u_h \in V_h$, $\mu_h \in \Lambda_h$ we have 
	\begin{equation}\label{contBDT}
\BDT(u_h,\mu_h) \leq \CBDT\, \tau | u_h |_{1,h}  \| \mu_h \|_{-1/2,h}.	\end{equation}
Moreover, there exists $\alpha_0 > 0$ such that for all $0 < \alpha < \alpha_0$, the following holds: there exists a constant $\tau_\alpha > 0$, depending on $\alpha$, such that if \eqref{deftau} holds for $\tau < \tau_\alpha$,  then for all $(u_h,\lambda_h) \in V_h \times \Lambda_h$
	\begin{equation}\label{coercivityBDT}
	\sup_{(v_h,\mu_h) \in   V_h \times \Lambda_h} \frac{ \ahBH(u_h,\lambda_h;v_h,\mu_h) + \BDT(u_h,\mu_h) }{
		\| v_h \|_{1,*} + \| \lambda_h \|_{-1/2,h}
		} \gtrsim 1,
	\end{equation}
	the implicit constant in the inequality depending on $\alpha$ and $\tau$.
\end{lemma}

\begin{proof}
			We start by observing  that, for $\sigma \in \mathbb{R}^d$ fixed unit vector, the higher order directional derivative $\partial^j_\sigma p$ is a linear combination of  mixed derivatives of order exactly $j$, that is, there exist real coefficients $c_\beta$ independent of $\sigma$ and $p$ such that
			\[
			\partial^j_\sigma p = \sum_{\beta \in \mathbb{N}^d \atop | \beta | = j} c_\beta \sigma^\beta D^\beta p, \qquad \text{ with } \qquad\sigma^\beta = \prod_{i=1}d \sigma_i^{\beta_i}, \ D^\beta p =  \frac{\partial^{|\beta|} p }{\partial x_1^{\beta_1} \cdots  \partial x_d^{\beta_d}}, \ \text{ and }\ | \beta| = \sum_{i=1}^d 
			\beta_i
			\]  (this can be easily proven by induction on $j$). This implies  that
			\begin{equation*}%\label{bounddjsigma}
			\|  \partial^j_\sigma p  \|^2_{0,\face} \lesssim \sum_{| \beta| = j} \| D^\beta p \|^2_{0,\face}, \end{equation*}
			the implicit constant in the inequality only depending on $j$. Then, as for $p \in \Poly{k}$, $D^\beta p \in \Poly{k}$, using  \eqref{tracePoly} and \eqref{inversebase}, as well as  \eqref{deftau}, we can write
			\begin{multline*}
			h_K^{-1} \sum_{\face\in \Fb\cap\FK}  \norm{ \delta^j \partial_\sigma^j p}^2_{0,\face} \leq 
			h_K^{-1} \sum_{\face\in \Fb\cap\FK}\| \delta \|^{2j}_{L^\infty(\face)}  \norm{\partial_\sigma^j p}^2_{0,\face}  \lesssim \tau^{2j} h_\el^{2j - 1} \sum_{|\beta|=j}\| D^\beta p \|^2_{0,\partial\el} \\
			\lesssim \tau^{2j} h_K^{2j-1} \sum_{|\beta|=j} \left(
			h_K^{-1} \| D^\beta p \|^2_{0,K} + h_K | D^\beta |^2_{1,K}
			\right)
			\lesssim \tau^{2j} \left(
			h_\el^{2j-2} |  p |^2_{j,K} + h_\el^{2j} |  p |^2_{j+1,K}
			\right)
			\lesssim \tau^{2j} | p |_{1,K}^2.
			\end{multline*}
			Summing up for $j = 1,\cdots,k^*$ we then have, for some constant $\wCinv$ depending on $k$ 
				\begin{equation}
				\hel^{-1} \sum_{\face\in \Fb \cap \FK} \sum_{j=1}^{k^*} \norm{\delta^j \partial_\sigma^j p}^2_{0,\face}
				\leq 
					\hel^{-1} \sum_{\face\in \Fb \cap \FK} \sum_{j=1}^{k} \norm{\delta^j \partial_\sigma^j p}^2_{0,\face}
				 \leq \wCinv \tau^{2j} \snorm{p}^2_{1,\el}.
				\end{equation}

Then, as  for $j \geq 1$ we have that $\tau^j \leq \tau$ (recall that, by assumption, $\tau \in (0,1)$), we can write
\begin{multline}\label{contC}
\BDT(u_h,\mu_h) 
= \sum_{\face\in \Fb} \sum_{j=1}^{k^*} \int_{\face } \frac{\delta^j}{j!} \partial^j_{\sigma}
(\Pinabla u_h)\mu_h \leq 
\sum_{\face\in \Fb}   
%\left( 
\| \mu_h \|_{0,\face} 
\sum_{j=1}^{k^*} 
\| \delta^j \partial^j_{\sigma}
(\Pinabla u_h)   \|_{0,\face}  
%\right) 
\\[2mm]
\leq  \sqrt{k^*}
\left(\sum_{\face\in \Fb}  \hKf  \| \mu_h \|^2_{0,\face} \right)^{1/2}
\left(\sum_{\face\in \Fb}   \hKf^{-1} \sum_{j=1}^{k^*}
  \|\delta^j \partial^j_{\sigma}
(\Pinabla u_h)   \|^2_{0,\face} \right)^{1/2}
\leq  \CBDT\tau | \Pn u_h |_{1,\Th} \| \mu_h \|_{-1/2,h},
\end{multline}
with $\CBDT\! = (k^* \wCinv)^{1/2}$, which gives us \eqref{contBDT}.

\

To prove \eqref{coercivityBDT},	let once again $(v_h,\mu_h)=(u_h,\bar \mu_h -\lambda_h)$. Combining \crefrange{coerc1}{partialbound} and \eqref{contC} we obtain
		\begin{multline}
\ahBH(u_h,\lambda_h;u_h,\bar\mu_h -\lambda_h) + \BDT(u_h,\bar \mu_h - \lambda_h) \\[2mm]
\geq 	
	(1-\alpha \Cinv^2) \snorm{\Pinabla u_h}^2_{1,*}  + c_* \snorm{(1-\Pinabla)u_h}^2_{1,*} + \alpha \| \lh \|^2_{-1/2,h} +
 \| \Pib (u_h) \|_{1/2,h}^2 \\[2mm] - (\alpha + \CBDT  \tau) \| \lambda_h \|_{-1/2,h} \| \Pib(u_h)
 \|_{1/2,h} - (\Cinv\alpha + \CBDT  \tau ) |  \Pn (u_h) |_{1,*} \| \Pib(u_h) \|_{1/2,h} \\[2mm]
 \geq 	(1-\frac 3 2 \alpha \Cinv^2) \snorm{\Pinabla u_h}^2_{1,*}  + c_* \snorm{(1-\Pinabla)u_h}^2_{1,*}  + \frac \alpha 2 \| \lh \|^2_{-1/2,h} +
 \| \Pib (u_h) \|_{1/2,h}^2  \\[2mm]
 - \frac 1{2 \alpha} (\alpha + \CBDT \tau)^2 \| \Pib (u_h) \|_{1/2,h}^2  + \frac 1 {2\alpha}  (\alpha + \frac \CBDT \Cinv \tau)^2 \| \Pib (u_h) \|_{1/2,h}^2\\[2mm]  \geq
 (1-\frac 3 2 \alpha \Cinv^2) \snorm{\Pinabla u_h}^2_{1,*}  + c_* \snorm{(1-\Pinabla)u_h}^2_{1,*} + \frac \alpha 2 \| \lh \|^2_{-1/2,h} \\[2mm] +
\Big(1- \frac 1{\alpha} (\alpha + \CBDT \tau)^2 \Big) \| \Pib (u_h) \|_{1/2,h}^2,
%( 1-\alpha \Cinv^2 - \frac \alpha 2 (\Cinv + \CBDT \frac \tau \alpha)^2) \snorm{\Pinabla u_h}^2_{1,*} + \| \Pib(u_h) \|_{1/2,h}(1 - )
 \end{multline}
%Choosing $\varepsilon = \alpha 2/C'$ we finally get
%\begin{multline*}
%\ahBH(u_h,\lambda_h;u_h,-\lambda_h) - \BDT(u_h,\lambda_h) \geq\\ (1-\alpha C_\perp - (C')^2 \frac{\tau^2}{4 \alpha}) \snorm{\Pinabla u_h}^2_{1,*}  + c_* \snorm{(1-\Pinabla)u_h}^2_{1,*} + \frac \alpha 2 \| \lh \|_{-1/2,h}.
%\end{multline*}
where we used that, as $\Cinv \geq 1$, it holds that $\alpha +  (\CBDT/\Cinv)  \tau \leq \alpha + \CBDT \tau$. Provided $\alpha < \alpha_0 = 2/(3 \Cinv^2)$, the coefficient in front of $| \Pn u_h |_{1,*}^2$ is positive. Given $\alpha$ and letting $C(\alpha,\tau)$ denote the coefficient in front of $\| \Pib (u_h) \|_{1/2,h}^2$, we have that $C(\alpha,\tau)$ is positive, provided $\tau$ is such that 
$(\alpha + \CBDT \tau)^2 < \alpha$. It is not difficult to check that such an inequality is satisfied provided  $\CBDT\tau \in ( - \alpha - \sqrt{\alpha}, -\alpha + \sqrt{\alpha})$. As $\alpha \leq \alpha_0 < 1$, $-\alpha + \sqrt{\alpha}$ is strictly positive. Then, setting $\tau_\alpha = (\sqrt{\alpha}-\alpha)/\CBDT$\!, for $\tau < \tau_\alpha$ we have that 
$C(\alpha,\tau) > 0$,  and we have proven our thesis.
	\end{proof}

	We are then able to prove the following result:
	\smallskip
	
	\begin{theorem} \label{theo:error_BH}
		 There exists $\alpha_0 > 0$ such that for all $\alpha$ with $0 < \alpha < \alpha_0$, the following holds: there exists a constant $\tau_\alpha$, depending on $\alpha$, such that if \eqref{deftau} holds for $\tau  < \tau_\alpha$, Problem \eqref{BDT_BH_VEM} is well posed, and,
		if  $u\in H^{k+1}(\Omega)\cap W^{k^*+1,\infty}(\Omega)$, the following error estimate holds
	\[
	\| u - u_h \|_{1,h} + \| \partial_{\n_h} u - \lambda_h \|_{-1/2,h} \lesssim h^{k} | u |_{k+1,\Omega} + h^{-1/2} \delta^{k^* + 1} | u |_{k^*+1,\infty,\Omega}.
	\]
	\end{theorem}
	
	\begin{proof} Let $\uI \in V_h$ and $u_\pi  \in \Poly{k}(\Tess)$ denote once again the  VE interpolant and the $L^2$ orthogonal projection of $u$,  and let this time $\lI\in \Lh$ denote $L^2(\Gamma_h)$ projection of $\partial_{\nu_h}u$. We observe that the solution $u$ of \eqref{prob_mod} satisfies, for all $v \in H^1(\Oh)$,
	\[
	\int_{\Oh} \nabla u \cdot \nabla v - \int_{\Gamma_h} \partial_{\n_h} u v = \int_{\Oh} f v.
	\]
	
Thanks to \eqref{coercivityBDT}, there exists $v_h \in V_h$ and $\mu_h \in \Lambda_h$ with $\| v_h \|_{1,h} + \| \mu_h \|_{-1/2,h} = 1$, such that we have:
		\begin{multline}
		\| \uh - \uI  \|_{1,h} + \| \lh - \lI \|_{-1/2,h}  \lesssim
		\ahBH(\uh - \uI,\lh - \lI;v_h,\mu_h) + \BDT(\uh - \uI,\mu_h) \\[3mm] =		
			A + B + C + D + E + F + G \\[1mm]
	+ \BDT(u - \uI,\mu_h) + \sum_{\face\in \Fb} \sum_{j=1}^{k^*} \int_{\face } \frac{\delta^j}{j!} \partial^j_{\sigma}
		( u - \Pinabla(u))\mu_h
		 + 
		\sum_{\face\in \Fb} \int_{\face } \Big( \widetilde g  - u - \sum_{j=1}^{k^*} \frac{\delta^j}{j!} \partial^j_{\sigma}
	u \Big)\mu_h
\\[1mm]	\lesssim h^k \| u \|_{k+1,\Oh} + H + I + J,
		\end{multline}
where $A$--$G$ are as in 	Section \ref{sec:BHVEM}, with  $\Omega$ and $\Gamma$ replaced by $\Oh$ and $\Gamma_h$, and are bounded as in \crefrange{boundABC}{boundG}. We then only need to bound $H$, $I$ and $J$. Using \eqref{contBDT} and \eqref{VEMapproxI}, we can write
\begin{equation}\label{boundH}
H \lesssim | u - u_I |_{1,\Oh} \lesssim h^k | u |_{k+1,\Oh} \leq h^k | u |_{k+1,\Omega}.
\end{equation}
As far as $I$ is concerned, adding and subtracting $u_\pi = \Pi^0(u)$, we can write
\begin{equation*}
 \sum_{\face\in \Fb} \sum_{j=1}^{k^*} \int_{\face }\int_{\face} \frac {\delta^j}{j!} \partial^j_\sigma(u - \Pn u) \mu_h \leq
 \sum_{\face\in \Fb} \sum_{j=1}^{k^*} \int_{\face }\int_{\face} \frac {\delta^j}{j!} \partial^j_\sigma(u - u_\pi) \mu_h  +  \BDT(u_\pi - u,\mu_h).
%\\[2mm]
% \lesssim \Big(h_K^{j}\|  \partial^j_\sigma(u - u_\pi) \|_{0,\face} +  h_K^{j} \| \partial^j_\sigma \Pn(u - u_\pi) \|_{0,\face} \Big) \| \lambda_h \|_{0,f} \\
%  \lesssim  h^j \sum_{|\beta| = j} \Big(\|  D^\beta(u - u_\pi) \|_{0,\face} +   \| D^\beta \Pn(u - u_\pi) \|_{0,\face} \Big)\| \mu_h \|_{0,f} 
 	\end{equation*}
 On the one hand, by \eqref{contBDT}, we can bound (recall that $\| \mu_h \|_{-1/2,h} \leq 1$)
 	\[
\BDT(u_\pi - u,\mu_h)
 \lesssim | u_\pi - u |_{1,*} \lesssim h^k | u |_{k+1,\Oh}.
 	\]
On the other hand, we observe that, as $\delta|_\face \lesssim \hKf$, for $1 \leq j \leq k^*$ integer, using \eqref{approxj} we have that
\begin{multline}
\sum_{\face} \int_{\face} \frac {\delta^j}{j!} \partial^j_\sigma(u - u_\pi) \mu_h \lesssim \sum_{|\beta|=j}
 \sum_{\face} \hKf^j \|  D^\beta(u - u_\pi) \|_{0,\face} \| \mu_h \|_{0,f}\\ \lesssim 
\sum_{|\beta|=j} 	\Big( \sum_K h_K^{2j-1}\sum_{\face \in \Fb \cap \FK} \|  D^\beta(u - u_\pi) \|^2_{0,\face} \Big)^{1/2}  \Big(
 	\sum_K h_K  \sum_{\face\in \Fb \cap \FK} \| \mu_h \|_{0,\face}^2
 	\Big)^{1/2}\\
 	\lesssim \sum_{|\beta|=j}
	\Big(
	\sum_K h_K^{2j-1} \| D^\beta (u - u_\pi) \|_{0,\partial K}^2 
	\Big)^{1/2} 
	\| \mu_h \|_{-1/2,h} \lesssim \Big(\sum_K (h_K^{2j-2} | u - u_\pi |^2_{j,K} + h_K^{2j} | u-u_\pi |^2_{j+1,K} )  \Big)^{1/2} 
\\	\lesssim \left(\sum_K h_K^{2k} | u |_{k+1,K}^2 \right)^{1/2}\lesssim h^k | u |_{k+1,\Omega},
	\end{multline}
that, summing for $j$ integer, $1\leq j \leq k^*$	yields
	\begin{equation}
	\label{boundI} I \lesssim h^k | u |_{k+1,\Omega}.
	\end{equation}
 	
% 	
% 	
% 	
% 	\begin{multline*}
%  \lesssim 
% h_K^{j+1/2}(h_K^{-1/2} | u-u_\pi |_{j,K} + h_K^{1/2} | u- u_\pi |_{j+1,K}) + h_K^{1/2} | \Pn ( u - u_\pi) |_{1,K}\\
% \lesssim 
%\end{multline*}

Moreover, using standard bounds on the Taylor expansion, for $x \in \Gh$ we have 
\[
| \widetilde g(x) - u - \sum_{j=1}^{k^*} \frac{\delta(x)^j}{j!} \partial^j_{\sigma}
u(x) | = | u(x+\delta(x) \sigma) - u(x) - \sum_{j=1}^{k^*} \frac{\delta(x)^j}{j!} \partial^j_{\sigma} u(x) | \lesssim \delta^{k^*+1} | u |_{k^*+1,\infty,\Omega},
\]
which, thanks to condition \eqref{deftau}, yields
\begin{multline*}
J \lesssim \Big( \sum_K \sum_{\face\in \Fb \cap \FK} \hKf^{-1} \|  \widetilde g - u - \sum_{j=1}^{k^*} \frac{\delta^j}{j!} \partial^j_{\sigma}
u  \|_{0,\face}^2 \Big)^{1/2}
\| 
\mu_h \|_{-1/2,h}\\ \lesssim 
\Big( \sum_{\face\in \Fb} \hKf^{-1} | \face |  \delta^{2 k^* + 2} |  u |_{k^*+1,\infty,\Omega}^2 \Big)^{1/2}
\lesssim
\tau^{1/2}| \Gamma_h |^{1/2} \delta^{k^* + 1/2} | u |_{k^*+1,\infty,\Omega}.
\end{multline*}

Then, as by assumption, $\tau < 1$, we have
\[
\| \uI - \uh \|_{1,h} + \| \lI - \lambda_h \|_{-1/2,h} \lesssim h^k | u  |_{k+1,\Omega} + \delta^{k^* + 1/2}  | u |_{k^*+1,\infty,\Omega}.
\]
		\end{proof}

				\begin{remark}
					Remark that if we replace $\hKf$ with $\hf$ in the definition of the discrete problem \eqref{BDT_BH_VEM}, the theoretical analysis of the method would require replacing the condition \eqref{deftau} with the much stronger condition
					\[
					\max_{\face\in \Fb}\,	 \max_{x\in \face} \frac{\delta(x)}{\hf} \leq \tau,		
					\] 
					which would drastically reduce the set of eligible tessellations.
				\end{remark}

Let us now consider the issue of the choice of the parameter $k^*$ appearing in the definition of the correction operator $\BDT$, which, of course, should be chosen as small as possible, in order to avoid unnecessary computational cost. As the order $j$ derivatives  in the definition of $\BDT$ always act on an order $k$ polynomial, we can safely choose $k^* \leq k$. The sensible criterion is, as usual, to balance the order of convergence of the two terms on the right hand side of the error bound in Theorem \ref{theo:error_BH}.  As already observed, there are situations where the approximating domains can be constructed in such a way that $\delta \lesssim h^2$. In such situation the error term resulting from the approximation of the domain is of order $h^{2 k^* + 3/2}$. We can then choose $k^*$ as the smallest integers such that $2k^* + 3/2 > k$, which gives us $k^* = \lceil k/2-3/4 \rceil$.  Remark that, by the Sobolev embedding theorem, for such choice of $k^*$ we have that $H^{k+1} (\Omega) \subset W^{k^*+1,\infty} (\bar \Omega)$, provided, if $d=3$, that $k > 1$. 
We then have the following corollary
\begin{corollary}
	Assume that the approximating subdomains $\Oh$ are such that $\delta \simeq h^2$, and, for $d=3$, that $k > 1$. Then, setting $k^* = \lceil k/2-3/4 \rceil$, the following error  bound holds:
	\[
\| u - u_h \|_{1,\Oh} \lesssim h^k | u |_{k+1,\Omega}.
	\] 
	\end{corollary}

\begin{remark}
		For the approximation of non convex domains it would be desirable to  give up the requirement that $\Omega_h \subset \Omega$, which would allow for smaller values for $\delta$. In order to define the discrete method with $\Oh \not \subset \Omega$, we would however need to construct an extension $\widetilde f$ of the load term $f$ to $\Omega_h\setminus \Omega$.
		In the two dimensional finite element case, this last strategy for adapting the method to non convex domains has been analyzed in \cite{Dupont1974}, that states a bound on the error of order $h^k +  h^{-1/2} \delta^2$. To obtain an optimal convergence rate we would then need to require that $\delta \lesssim h^{k/2+1/4}$, which, as $k$ increases, is a much less favorable condition than \eqref{deftau} with $\tau < \tau_\alpha$.
		This lack of optimality in the high order case is to be expected, as the resulting method shares some characteristics with fictitious domain methods, which generally suffer from an upper bound on the order of convergence, unless the ``right'' extension of $f$ or some additional information on the solution are retrieved, which adds a non negligible computational overhead.  
	\end{remark}

	\subsection{Nitsche's method with curved boundaries}

	Also for the discrete formulation \eqref{BDT_BH_VEM} we can eliminate the multiplier to obtain a Nitsche's type formulation in the sole unknown $u_h$. More precisely, with $k'=k$, setting $v_h = 0$ in \eqref{BDT_BH_VEM}, we see that $\lambda_h$ satisfies, for all $\mu_h \in \Lambda_h$,
\begin{gather*}%\label{BDT_BH_VEM}
\int_\face \lambda_h   \mu_h
%\begin{equation}\label{BDT_BH_VEM}
%	\Ch(u_h,v_h) 
= \gamma \hKf^{-1} \int_{\face}\mu_h u_h + \gamma \hKf^{-1} \sum_{j=1}^{k^*} \int_{\face } \frac{\delta^j}{j!} \partial^j_{\sigma}
\Pinabla (u_h)\mu_h -   \int_\face  \dn \Pn(u_h)   \mu_h 
- \gamma \hKf^{-1} \int_\face  \widetilde g_h \mu_h ,
\end{gather*}
whence, since by definition, $\Pinabla (u_h) \in \Poly{k}(\Tess)$, we easily see that
\begin{multline}\label{eq:u2lambda}
\lambda_h |_\face = \gamma \hKf^{-1}\Pib  \left(u_h + \sum_{j=1}^{k^*}  \frac{\delta^j}{j!} \partial^j_{\sigma}
\Pinabla (u_h) - \widetilde g_h -  \dn \Pn(u_h)   \right) \\ =
 \gamma \hKf^{-1}\Pib  (u_h - \widetilde g_h ) + \gamma \hKf^{-1} \sum_{j=1}^{k^*} \frac{\delta^j}{j!} \partial^j_{\sigma}
 \Pinabla (u_h)  -  \dn \Pn(u_h).
\end{multline}
We can then once again substitute the expression on the right hand side of \eqref{eq:u2lambda} for $\lambda_h$ in \eqref{BDT_BH_VEM}, and, setting $\mu_h=0$, we obtain the following equation for $u_h$:
\begin{multline}\label{BDT_Nitsche_VEM}
 a_h(u_h,v_h) - \sum_{\face \in \Fb}  \int_\face \dn \Pn(u_h) v_h - \sum_{\face \in \Fb}  \int_\face   u_h  \dn \Pn(v_h)+ \gamma\sum_{\face \in \Fb} \hKf^{-1} \int_\face   \Pib  (u_h) \Pib(v_h )\\ - \sum_{\face \in \Fb}  \int_\face   \sum_{j=1}^{k^*} \frac{\delta^j}{j!} \partial^j_{\sigma}\Pinabla (u_h)  \left( \dn \Pn(v_h) - \gamma\hKf^{-1}  \Pib(v_h )\right)  
%\begin{equation}\label{BDT_BH_VEM}
%	\Ch(u_h,v_h) 
\\ = \int_{\Oh} f_h v_h   - \sum_{\face \in \Fb}  \int_\face \widetilde g_h  \left( \dn \Pn(v_h) - \gamma\hKf^{-1}  \Pib(v_h )\right),
\end{multline}
where, as $\Pinabla( u_h)$ and $\Pinabla (v_h)$ are piecewise polynomials, we once again used the definition of $\Pib$ to switch from $u_h$ and $v_h$ to $\Pib(u_h
)$ and $\Pib(v_h
)$ and vice versa, thus underlining the symmetry of some of the components appearing in the equation. The well posedness and error estimates for Problem \eqref{BDT_Nitsche_VEM}  are then consequences of the analogous results for Problem \eqref{BDT_BH_VEM}, as stated by the following corollary of Theorem \ref{theo:error_BH}.
\begin{corollary}\label{cor:nitsche:curved}
	 There exists $\gamma_0 > 0$ such that for all $\gamma > \gamma_0$, the following holds: there exists a constant $\tau_\gamma$, depending on $\gamma$, such that if \eqref{deftau} holds for $\tau  < \tau_\gamma$, Problem \eqref{BDT_Nitsche_VEM} is well posed, and,
	 if  $u\in H^{k+1}(\Omega)\cap W^{k^*+1,\infty}(\Omega)$, the following error estimate holds
	 \[
	 \| u - u_h \|_{1,\Oh}  \lesssim h^{k} | u |_{k+1,\Omega} + h^{-1/2} \delta^{k^* + 1} | u |_{k^*+1,\infty,\Omega}.
	 \]
	\end{corollary}

%% file: numerical_convergence_nitsche.tex
\newcommand{\testzero}{1}

\subsection{Test \testzero}

Our first goal is to validate the method proposed in Section \ref{sec:BHVEM}. To this aim we set $\Omega = [0,1]^3$ and  we choose the right hand side $f$ and the boundary data $g$ for Problem \eqref{prob_mod_PolyDom} in such a way that
\[
u = \frac{1}{3\pi^2}\cos(\pi x)\cos(\pi y)\cos(\pi z)
\]
is the exact solution. We take meshes made of random Voronoi cells like the one shown in Figure~\ref{fig:unitcube-voro}. Geometrical data are listed in Table~\ref{tab:unitcube-voro}. Figure~\ref{fig:convergence-voro} displays the results obtained by Nitsche's method with $\gamma = 100.0$ and $d$-recipe stabilization. The results confirm the theoretical estimate.

\begin{figure}
  \centering
  \includegraphics[width=0.5\textwidth]{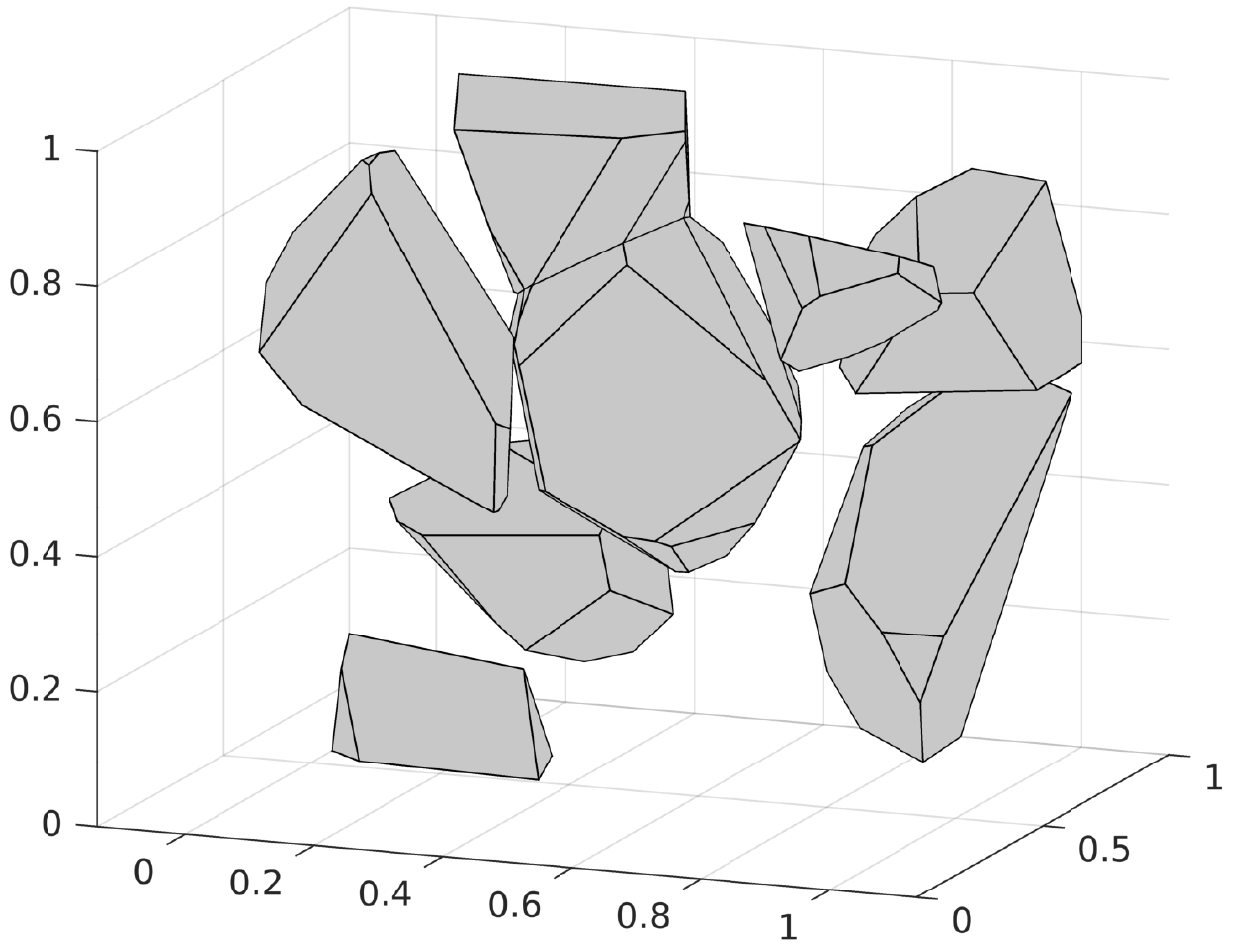}
  \caption{Examples of random Voronoi cells.}
  \label{fig:unitcube-voro}
\end{figure}

\begin{table}
  \centering
  \caption{Data for the meshes of random Voronoi cells for Test \testzero.}
  \label{tab:unitcube-voro}
  \begin{tabular}{
      c
      S[table-format=5.0]
      S[table-format=6.0]
      S[table-format=6.0]
      S[table-format=6.0]
      S[table-format=1.{\roundPrecision}e-1]
      S[table-format=1.{\roundPrecision}e-1]
      S[table-format=1.{\roundPrecision}e-1]
      S[table-format=1.{\roundPrecision}e-1]
    }
    \toprule
        {Mesh} & {$N_P$} & {$N_F$} & {$N_E$} & {$N_V$} & {$h$} & {$\overline h$} & {$h^{\text{min}}$} & {$\gamma_0$}\\
        \midrule
        voro$_{1}$ & 64 & 415 & 704 & 354 & 7.159673e-01 & 5.187269e-01 & 1.849483e-05 & 8.641639e-02\\
        voro$_{2}$ & 512 & 3625 & 6228 & 3116 & 3.616750e-01 & 2.403042e-01 & 2.602670e-05 & 4.240701e-02\\
        voro$_{3}$ & 4096 & 30364 & 52538 & 26271 & 1.896539e-01 & 1.176534e-01 & 1.313472e-06 & 6.524460e-02\\
        voro$_{4}$ & 32768 & 248586 & 431638 & 215821 & 9.520392e-02 & 5.790240e-02 & 3.684832e-08 & 5.456278e-02\\
        \bottomrule
  \end{tabular}
\end{table}

\begin{figure}
  \begin{tabular}{rl}
    \begin{tikzpicture}[trim axis left]
      \begin{loglogaxis}
	[ mark size=4pt, grid=major, small,
	  xlabel={Average mesh size $\overline h$},
	  ylabel={$e_1^u$},
	  legend columns=4 ]
        \addplot[color=red,mark=x] coordinates {
          (0.518727,0.45378)
          (0.240304,0.247293)
          (0.117653,0.120913)
          (0.0579024,0.0593836)
        };
%        \addlegendentry{$k = 1$}
        \addplot[color=green,mark=+] coordinates {
          (0.518727,0.173633)
          (0.240304,0.0371315)
          (0.117653,0.00916831)
          (0.0579024,0.00220962)
        };
%        \addlegendentry{$k = 2$}
        \addplot[color=blue,mark=o] coordinates {
          (0.518727,0.0317977)
          (0.240304,0.00391588)
          (0.117653,0.000445901)
          (0.0579024,5.29398e-05)
        };
%        \addlegendentry{$k = 3$}
        \addplot[color=orange,mark=square] coordinates {
          (0.518727,0.00770122)
          (0.240304,0.000342951)
          (0.117653,1.95113e-05)
          (0.0579024,1.11862e-06)
        };
%        \addlegendentry{$k = 4$}
        \addplot[thick,color=black,no markers] coordinates {
          (0.240304, 1.11862e-06)
          (0.518727, 1.11862e-06)
          (0.518727, 2.41468e-06)
          (0.240304, 1.11862e-06)
        };
        \node [anchor=west,font=\footnotesize] at (0.518727,2.41468e-06) {$1$};
        \addplot[thick,color=black,no markers] coordinates {
          (0.240304, 1.11862e-06)
          (0.518727, 1.11862e-06)
          (0.518727, 5.21241e-06)
          (0.240304, 1.11862e-06)
        };
%        \node [anchor=west,font=\footnotesize] at (0.518727,5.21241e-06) {$2$};
        \addplot[thick,color=black,no markers] coordinates {
          (0.240304, 1.11862e-06)
          (0.518727, 1.11862e-06)
          (0.518727, 1.12516e-05)
          (0.240304, 1.11862e-06)
        };
%        \node [anchor=west,font=\footnotesize] at (0.518727,1.12516e-05) {$3$};
        \addplot[thick,color=black,no markers] coordinates {
          (0.240304, 1.11862e-06)
          (0.518727, 1.11862e-06)
          (0.518727, 2.42881e-05)
          (0.240304, 1.11862e-06)
        };
        \node [anchor=west,font=\footnotesize] at (0.518727,2.42881e-05) {$4$};
      \end{loglogaxis}
    \end{tikzpicture}
    
    &
    
    \begin{tikzpicture}[trim axis right]
      \begin{loglogaxis}
	[ mark size=4pt, grid=major, small,
	  xlabel={Average mesh size $\overline h$},
	  ylabel={$e_0^u$},
	  legend columns=2 ]
        \addplot[color=red,mark=x] coordinates {
          (0.518727,0.165537)
          (0.240304,0.0502974)
          (0.117653,0.0130444)
          (0.0579024,0.00333671)
        };
%        \addlegendentry{$k = 2$}
        \addplot[color=green,mark=+] coordinates {
          (0.518727,0.0587216)
          (0.240304,0.00585629)
          (0.117653,0.000709113)
          (0.0579024,8.40116e-05)
        };
%        \addlegendentry{$k = 3$}
        \addplot[color=blue,mark=o] coordinates {
          (0.518727,0.00814477)
          (0.240304,0.000469936)
          (0.117653,2.61778e-05)
          (0.0579024,1.52652e-06)
        };
%        \addlegendentry{$k = 4$}
        \addplot[color=orange,mark=square] coordinates {
          (0.518727,0.00154632)
          (0.240304,2.93908e-05)
          (0.117653,8.11017e-07)
          (0.0579024,2.21885e-08)
        };
%        \addlegendentry{$k = 5$}
        \addplot[thick,color=black,no markers] coordinates {
          (0.240304, 2.21885e-08)
          (0.518727, 2.21885e-08)
          (0.518727, 1.03391e-07)
          (0.240304, 2.21885e-08)
        };
        \node [anchor=west,font=\footnotesize] at (0.518727,1.03391e-07) {$2$};
        \addplot[thick,color=black,no markers] coordinates {
          (0.240304, 2.21885e-08)
          (0.518727, 2.21885e-08)
          (0.518727, 2.23183e-07)
          (0.240304, 2.21885e-08)
        };
%        \node [anchor=west,font=\footnotesize] at (0.518727,2.23183e-07) {$3$};
        \addplot[thick,color=black,no markers] coordinates {
          (0.240304, 2.21885e-08)
          (0.518727, 2.21885e-08)
          (0.518727, 4.8177e-07)
          (0.240304, 2.21885e-08)
        };
%        \node [anchor=west,font=\footnotesize] at (0.518727,4.8177e-07) {$4$};
        \addplot[thick,color=black,no markers] coordinates {
          (0.240304, 2.21885e-08)
          (0.518727, 2.21885e-08)
          (0.518727, 1.03996e-06)
          (0.240304, 2.21885e-08)
        };
        \node [anchor=west,font=\footnotesize] at (0.518727,1.03996e-06) {$5$};
      \end{loglogaxis}
    \end{tikzpicture}
  \end{tabular}
  
    \begin{tabular}{c}
    \begin{tikzpicture} 
      \begin{axis}[%
          hide axis,
          xmin=10,
          xmax=50,
          ymin=0,
          ymax=0.4,
          legend style={draw=white!15!black,legend cell align=left},
          legend columns=4
        ]
        
        \addlegendimage{color=red,mark=x}
        \addlegendentry{$k = 1$}
        
        \addlegendimage{color=green,mark=+}
        \addlegendentry{$k = 2$}
        
        \addlegendimage{color=blue,mark=o}
        \addlegendentry{$k = 3$}
        
        \addlegendimage{color=orange,mark=square}
        \addlegendentry{$k = 4$}
      \end{axis}
    \end{tikzpicture}
  \end{tabular}

  \caption{Convergence plots for Nitsche's method on Test \testzero, with $\gamma = 100.0$ and $d$-recipe stabilization on the random Voronoi meshes of the unit cube.}
  \label{fig:convergence-voro}
\end{figure}
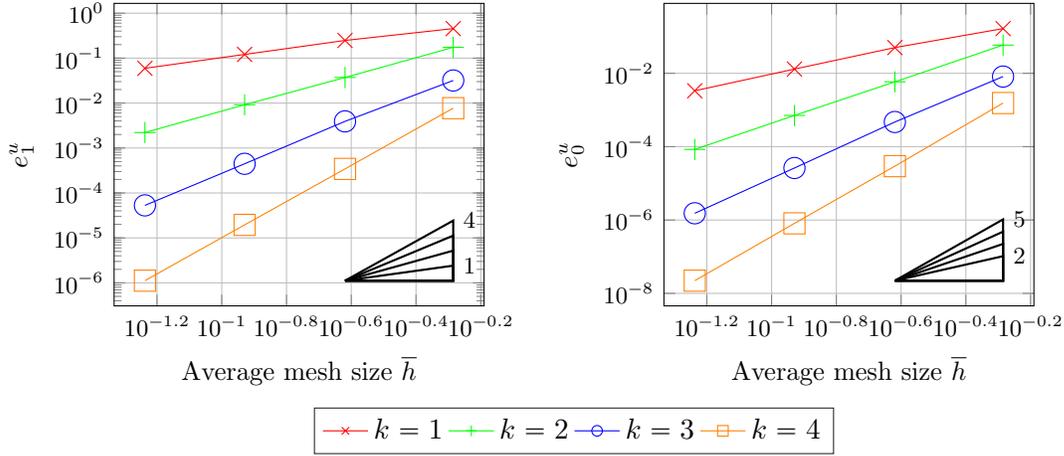

%% file: numerical.tex
\newcommand{\testone}{2}
\newcommand{\testwo}{3}

\subsection{Test \testone}
We now consider the method proposed in Section \ref{sec:curvo}.
As a domain, we take the volume enclosed by the torus shown in Figure~\ref{fig:torus}. The right hand side $f$ and the boundary data $g$ are chosen in such a way that the solution to our model problem is given by: 
\[
u = \cos(\pi z/8)\sqrt{x^2+y^2}.
\]
We consider a family of highly regular meshes like the one shown in Figure~\ref{fig:torus-mesh}. Geometrical data are listed in Table~\ref{tab:torus-structured}. 
To generate these meshes, we exploit the fact that, for this test case, the domain is a solid of revolution: first, we generate a two dimensional tesselation of a cross section of a plane passing through the z-axis; second, we extrude it using stepwise rotations around the z-axis. The resulting mesh $\Omega_h$ is interpolatory on the whole boundary of the domain, including its concave portion. %, and so $\Omega_h\not\subset\Omega$. 
	The final step consists in slightly perturbing the boundary faces %of $\Omega_h$ 
	that lie outside $\Omega$ %slightly 
	enough to  %keep them flat and 
	ensure that $\Omega_h\subset\Omega$, while also taking care of keeping them flat.
%\textcolor{green}{Note that these meshes are strictly contained inside the domain. }
%Even if $\delta = \mathcal O(h_K^2)$ for these meshes, we need to set $k^* = k$ to properly characterize the curvature of the toroidal domain and get optimal order of convergence. 
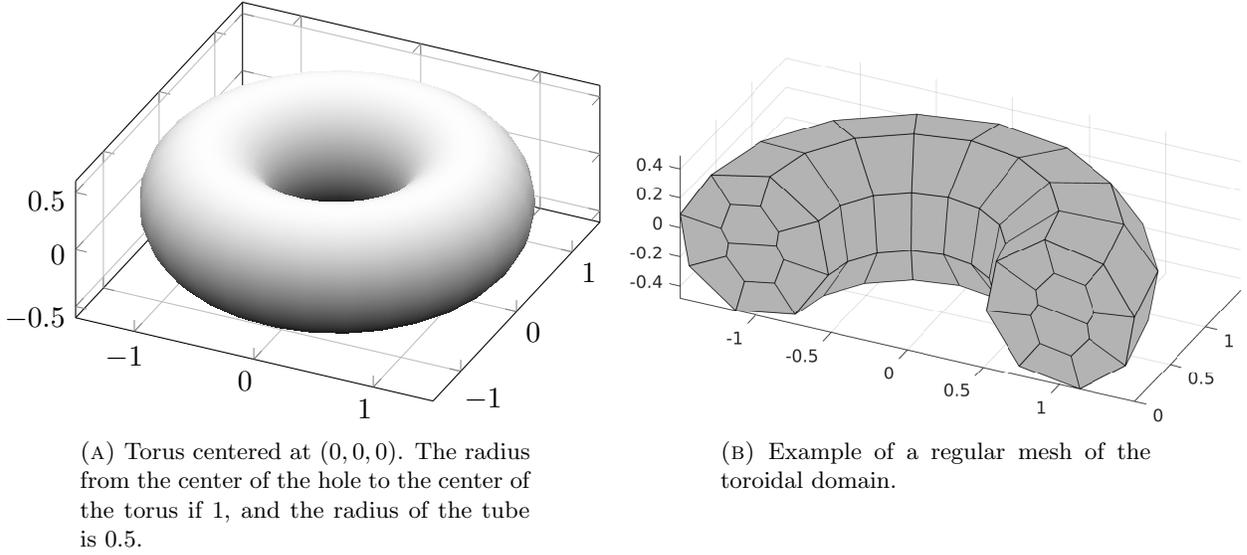
\begin{figure}
  \centering
  \subfloat[Torus centered at $(0,0,0)$. The radius from the center of the hole to the center of the torus if $1$, and the radius of the tube is $0.5$.]{\resizebox{0.5\columnwidth}{!}{
  \begin{tikzpicture}
    \begin{axis}[axis equal image,z buffer=sort,colormap/blackwhite,grid=major]
      \addplot3[surf,blue,samples=30,shader=interp,domain=0:2*pi]
      ({(1+0.5*cos(deg(x)))*cos(deg(y))},{(1+0.5*cos(deg(x)))*sin(deg(y))},{0.5*sin(deg(x))});
    \end{axis}
  \end{tikzpicture}}
  \label{fig:torus}
  }
  \subfloat[Example of a regular mesh of the toroidal domain.]{\includegraphics[width=.5\columnwidth]{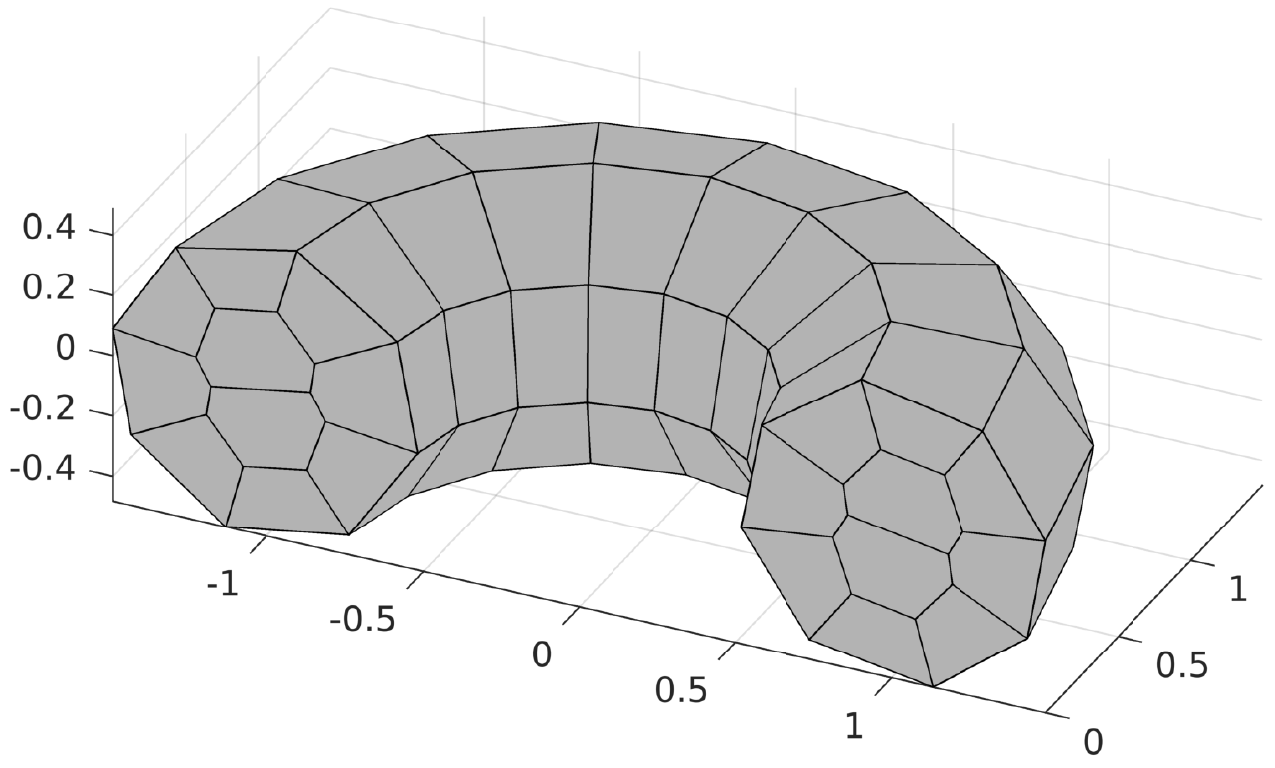}\label{fig:torus-mesh}}
\caption{Geometry for Test~\testone}
%\label{fig:torus-mesh}
\end{figure} 

\begin{table}
  \centering
  \caption{Data for the regular meshes of the toroidal domain.}
  \label{tab:torus-structured}
  \begin{tabular}{
      c
%      S[table-format=5.0]
%      S[table-format=6.0]
%      S[table-format=6.0]
%      S[table-format=6.0]
      S[table-format=6.0]
      S[table-format=7.0]
      S[table-format=7.0]
      S[table-format=7.0]
      S[table-format=1.{\roundPrecision}e-1]
      S[table-format=1.{\roundPrecision}e-1]
      S[table-format=1.{\roundPrecision}e-1]
      S[table-format=1.{\roundPrecision}e-1]
    }
    \toprule
        {Mesh} & {$N_P$} & {$N_F$} & {$N_E$} & {$N_V$} & {$h$} & {$\overline h$} & {$h^{\text{min}}$} & {$\gamma_0$}\\
        \midrule
        torus$_{1}$ & 160 & 592 & 720 & 288 & 6.659261e-01 & 5.614744e-01 & 9.831548e-02 & 1.667362e-01\\
        torus$_{2}$ & 1280 & 5024 & 6240 & 2496 & 3.342116e-01 & 2.720203e-01 & 3.371526e-03 & 1.737496e-01\\
        torus$_{3}$ & 10240 & 40768 & 50880 & 20352 & 1.726452e-01 & 1.348243e-01 & 1.923248e-03 & 1.885854e-01\\
        torus$_{4}$ & 81920 & 327296 & 408960 & 163584 & 8.736604e-02 & 6.718889e-02 & 9.134261e-04 & 1.888235e-01\\
        \bottomrule
  \end{tabular}
\end{table}

\input{torus_results.tex}

Figure~\ref{fig:torus_H1_L2} displays the results obtained by the BH method with $k' = k-1$, $\alpha = 0.001$, $k^* = k$, $\delta$ computed along the outer normal ($\sigma=\nu_h$), and $d$-recipe stabilization.  As expected, we get very similar results by using  Nitsche's method (equivalent to the Barbosa--Hughes method with $k'= k$) with $\gamma = 1000$ (not shown). It is worth observing that the choice of the VEM stabilization becomes critical to get proper convergence rates as the order $k$ of the VEM method increases.  Table~\ref{tab:torus-stabilizations} shows what we get by replacing the $d$-recipe with the euclidean stabilization for $k = 4$, which still provides the correct order but with much worse values for the error.

For this test case, we also compared the 1-norm condition number of the matrices relative to Nitsche's method with and without the BDT correction. The results, computed according to \cite{hager1984,higham2000} and displayed in 
Table~\ref{tab:torus-condest-d-recipe}, show that, asymptotically as $h$ decreases, the addition of the correction term does not significantly degrade the condition number, despite the presence of the higher order derivatives $\partial^j_\sigma$, whose negative effect is indeed dampened by the term $\delta^j$.

\subsection{Test~\testwo}
Creating a polyhedral mesh approximating a general three-dimensional domain with curved boundaries is a challenging task, around which an active area of research revolves~\cite{vorocrust}. In many fields like computer graphics, complex domains are simply and easily approximated by using a collection of cubes, providing an approximation of the curved boundary surfaces characterized by $\delta = \mathcal O(h_K)$ and yielding, when solving a PDE, an $O(h)$ error, which, for higher order methods, dominates the overall error. In this framework, the virtual element method allows to easily build polyhedral decomposition, where the elements are obtained as union of cubes, in such a way that condition \eqref{deftau}  holds, so that we can resort to either one of the methods proposed in Section \ref{sec:curvo}, with optimal error bounds also for higher order discretizations. To demonstrate this potential, we perform the following test. We consider meshes whose elements of diameter $\sim h$ are union of cubes with edge length $h/2^n$ with $n$ independent of $h$ large enough, so that condition \eqref{deftau} holds. We point out that the underlying hexahedral mesh does not, in general, satisfy such a condition, an then it is not well suited to be used directly. For the purpose of the present test, we obtain such elements by successive refinement of a starting cubic mesh with 
% edge length $=h$, as shown in Figures \ref{fig:ball3d-sequence} and \ref{fig:torus3d-sequence}. 
 
With meshes of this kind, for which it may well happen that a facet is almost orthogonal to the boundary of the smooth domain, choosing the right direction $\sigma$, along which to compute  the Taylor expansion involved in the boundary correction, is of paramount importance. Here we define $\sigma$ on a facet $f$ as the direction of the gradient of the distance from the boundary, computed at the center of the facet. 

\

\begin{figure}
  \centering
  \subfloat[Starting mesh.]{\includegraphics[width=0.3\textwidth]{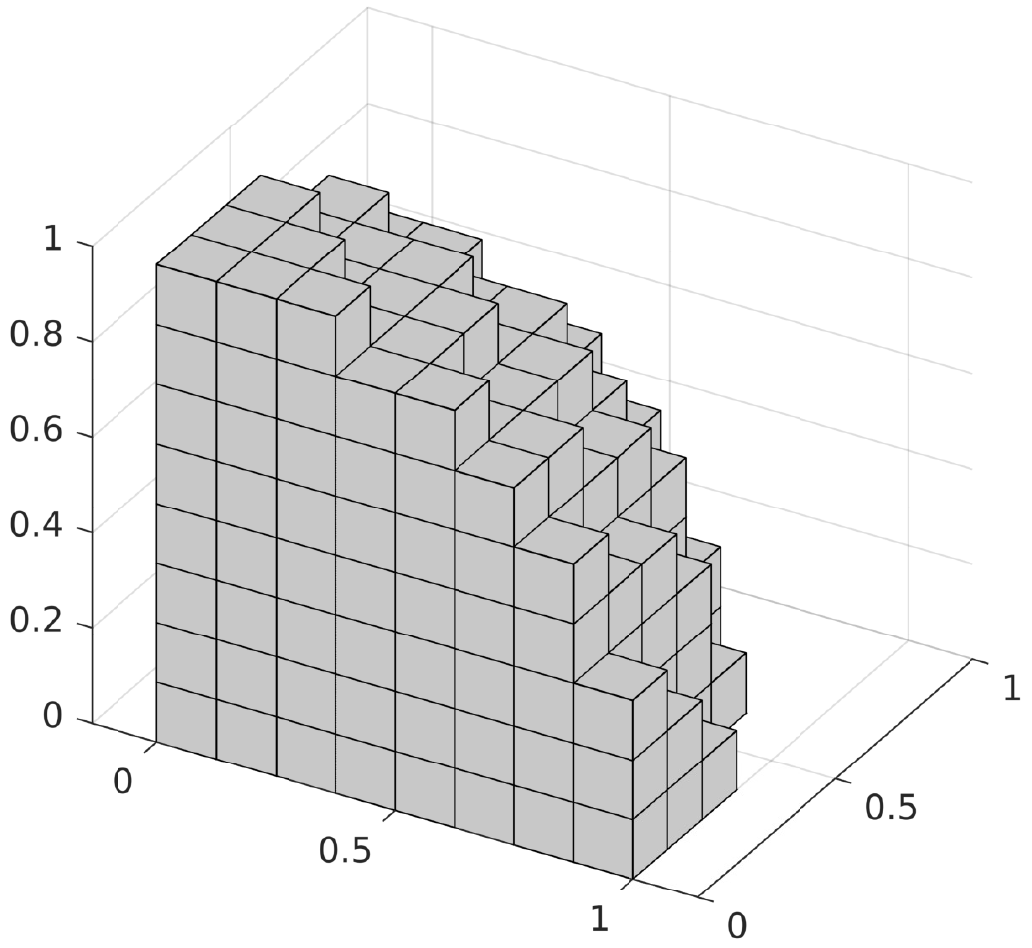}}\quad
  \subfloat[Mesh after one refinement step.]{\includegraphics[width=0.3\textwidth]{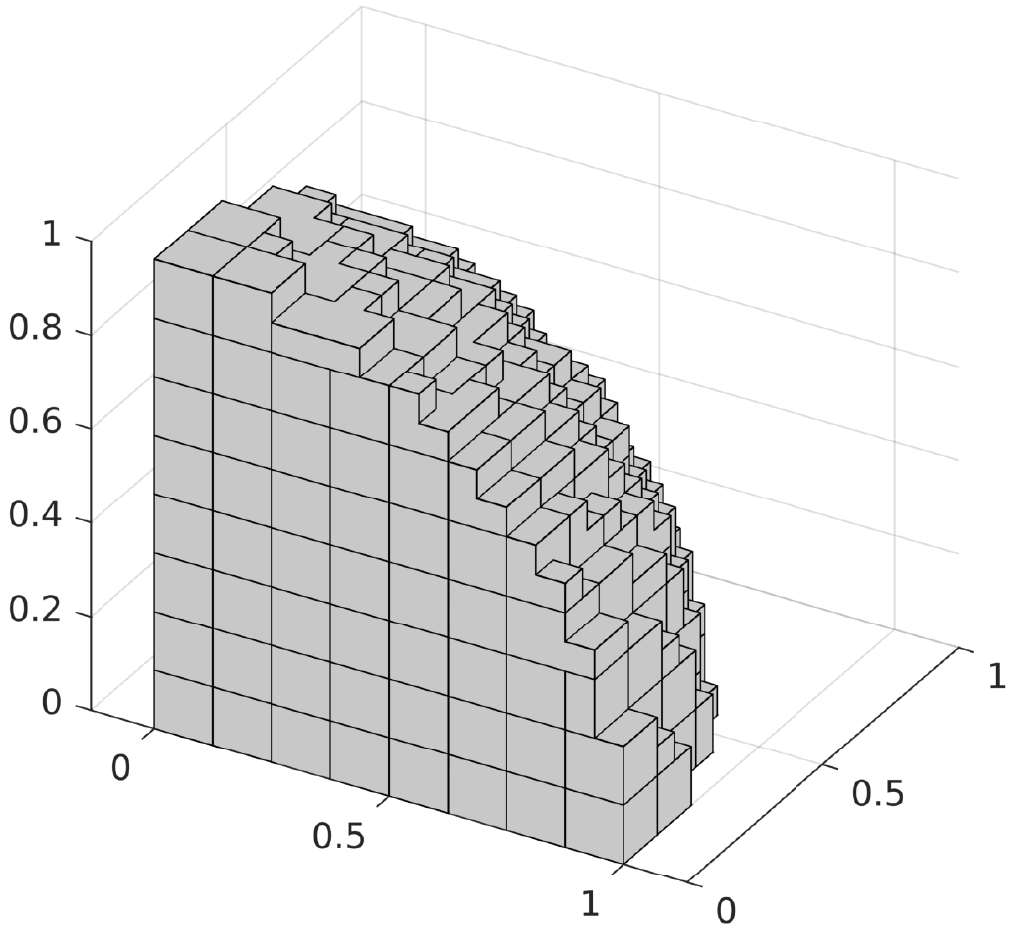}}\quad
  \subfloat[Mesh after two refinement steps.]{\includegraphics[width=0.3\textwidth]{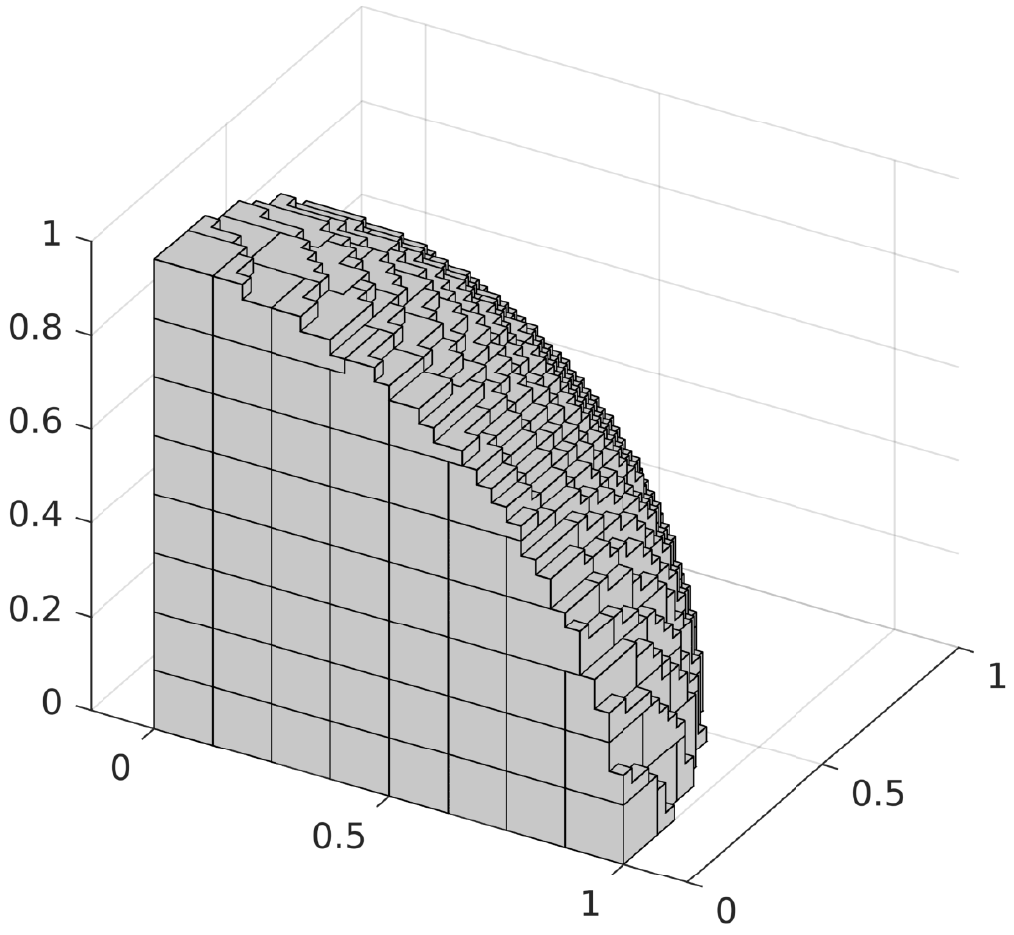}}
  \caption{Starting cubic mesh of a portion of the unit ball and two refined meshes.}
  \label{fig:ball3d-sequence}
\end{figure}

\begin{figure}
  \centering
  \subfloat[Starting mesh.]{\includegraphics[width=0.3\textwidth]{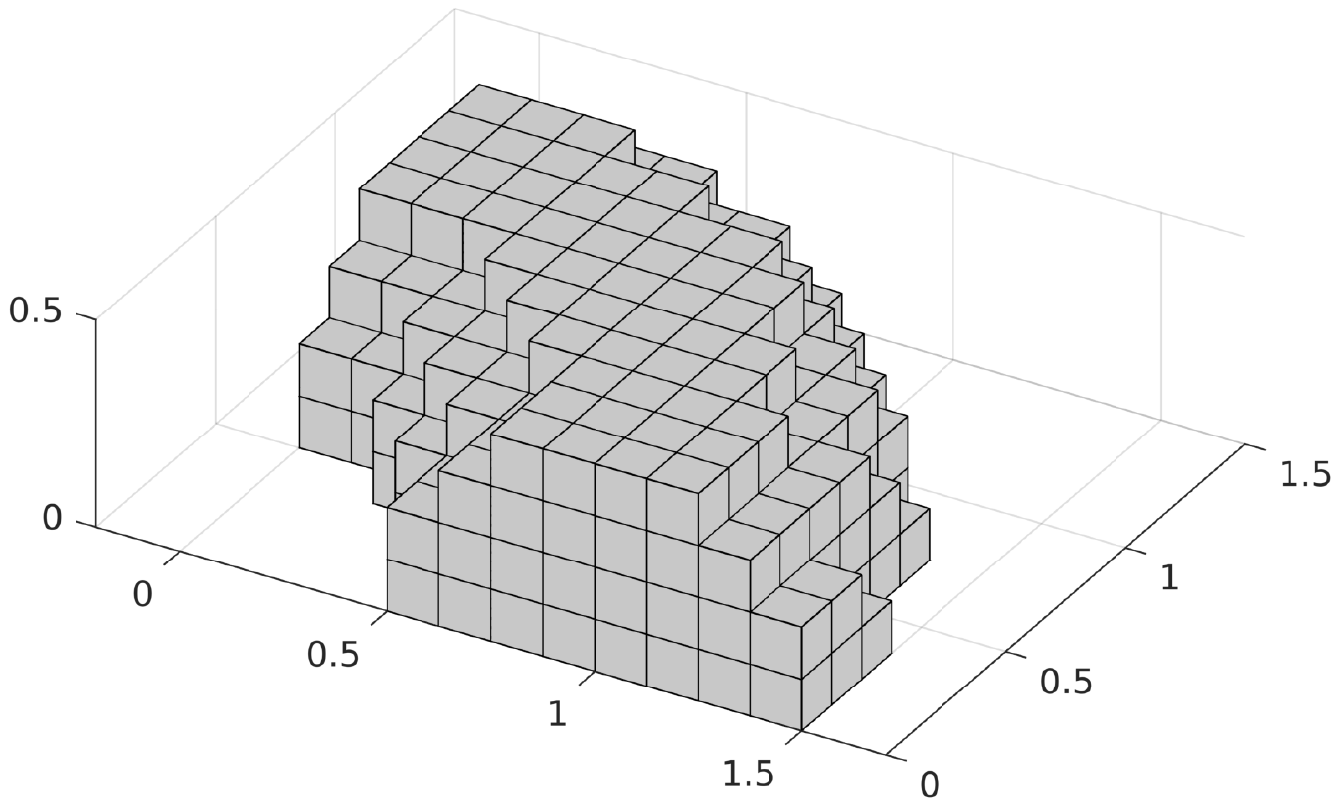}}\quad
  \subfloat[Mesh after one refinement step.]{\includegraphics[width=0.3\textwidth]{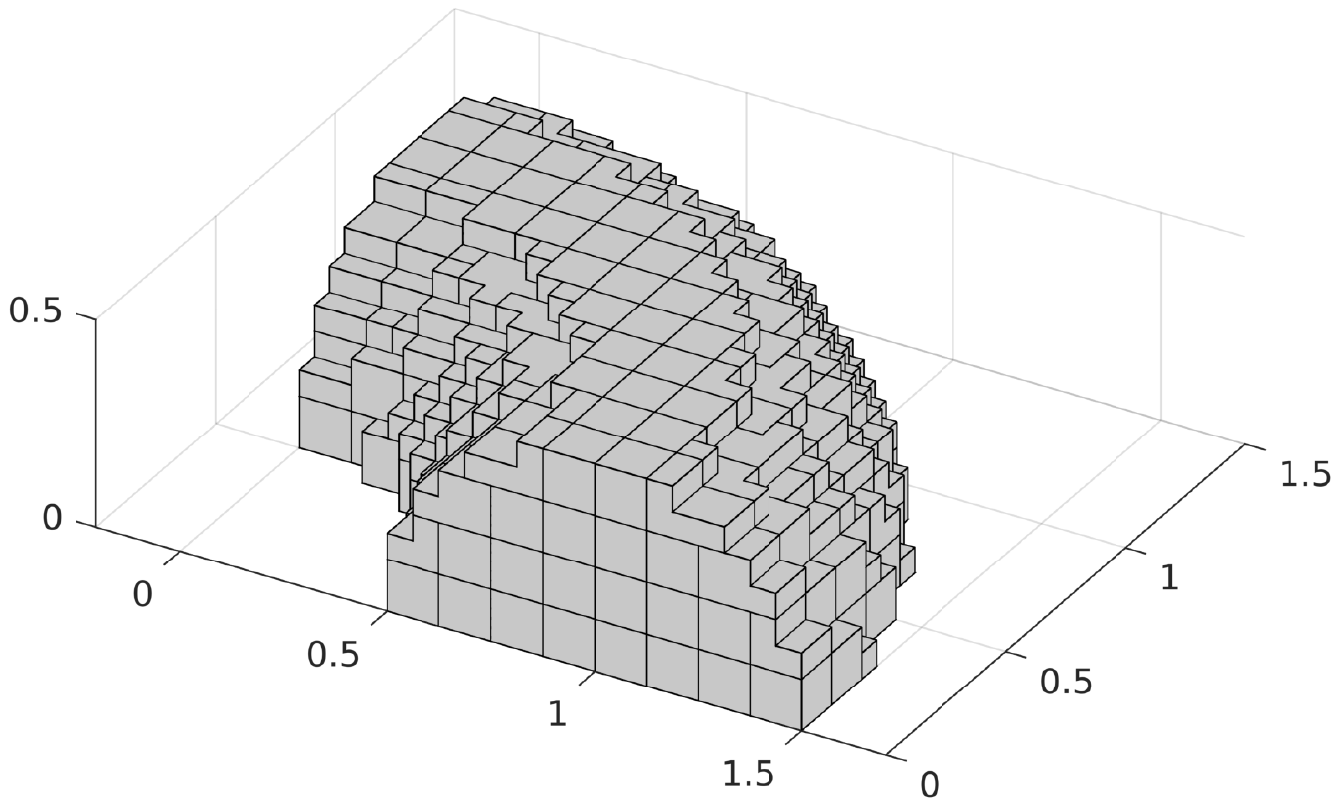}}\quad
  \subfloat[Mesh after two refinement steps.]{\includegraphics[width=0.3\textwidth]{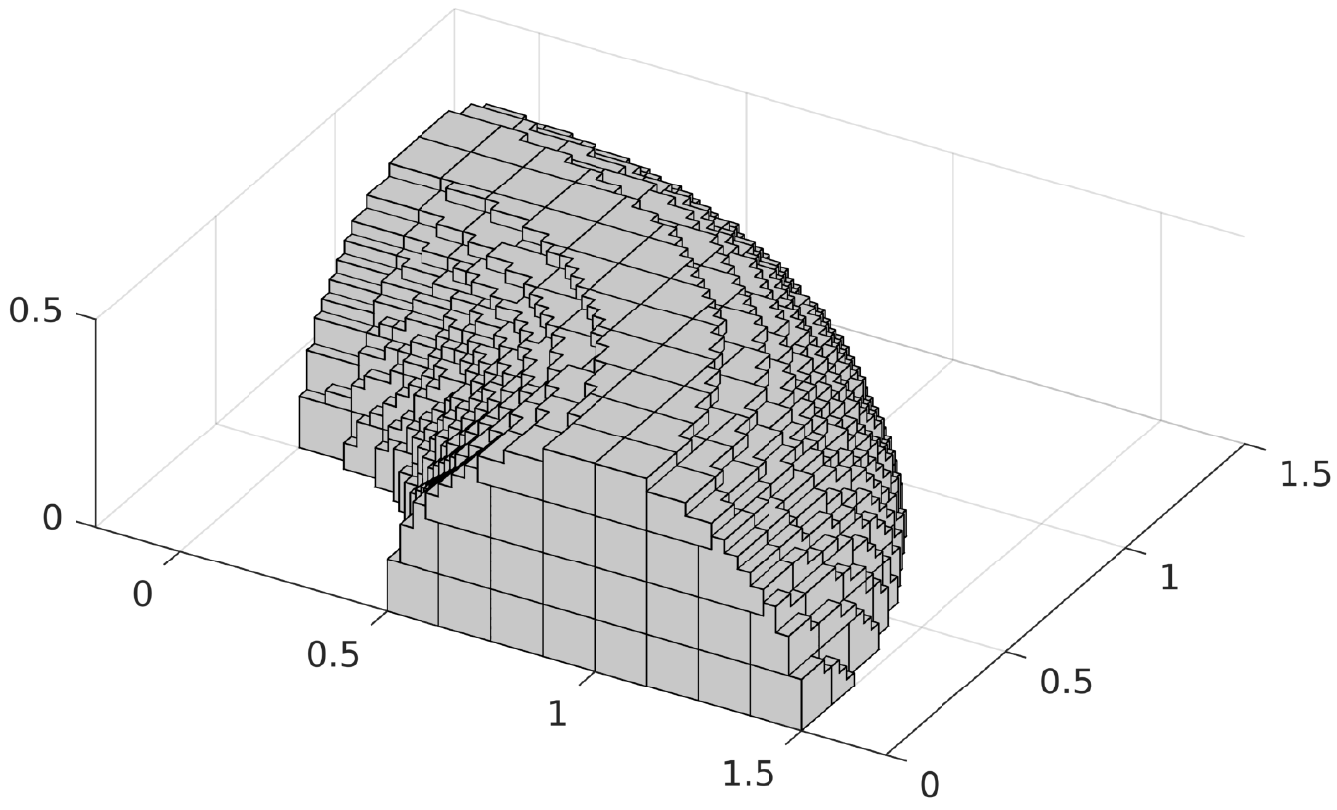}}
  \caption{Starting cubic mesh of a portion of the toroidal domain and two refined meshes.}
  \label{fig:torus3d-sequence}
\end{figure}

\begin{table}
  \centering
  \caption{Data for the meshes of the portion of the unit ball obtained after four steps of refinement.}
  \label{tab:ball-ref4}
  \begin{tabular}{
      c
      S[table-format=4.0]
      S[table-format=5.0]
      S[table-format=6.0]
      S[table-format=5.0]
      S[table-format=1.{\roundPrecision}e-1]
      S[table-format=1.{\roundPrecision}e-1]
      S[table-format=1.{\roundPrecision}e-1]
    }
    \toprule
        {Mesh} & {$N_P$} & {$N_F$} & {$N_E$} & {$N_V$} & {$h$} & {$\overline h$} & {$h^{\text{min}}$}\\
        \midrule
ball$^b_{1}$ & 35 & 3795 & 9340 & 5581 & 4.684895e-01 & 4.387450e-01 & 1.562500e-02\\
ball$^b_{2}$ & 272 & 15551 & 37332 & 22054 & 2.520667e-01 & 2.189163e-01 & 7.812500e-03\\
ball$^b_{3}$ & 2157 & 65524 & 151775 & 88409 & 1.260333e-01 & 1.089149e-01 & 3.906250e-03\\
        \bottomrule
  \end{tabular}
\end{table}

We run experiments on two different domains characterized by different curvatures. As bases for $\mathbb P_k(K)$ and $\mathbb P_k(f)$, we use the monomials defined in equations \eqref{eq:monomials-3d} and \eqref{eq:monomials-2d}, respectively.
\subsubsection*{Domain 1} The first domain is the portion of the unit ball contained in the octant $x \geq 0, y \geq 0, z \geq 0$, which has constant curvature. The right hand side $f$ and the boundary data $g$ are chosen such that
    \[
    u = \cos\left(\frac{\pi}{4}(x^2+y^2+z^2)\right)
    \]
    is the exact solution. We start with a family of nested cubical meshes, where a cube is retained if its center belongs to the actual domain. As expected, if the cubic elements are used ``as they are'' and no local refinement is performed, Nitsche's method with $\gamma = 1000$, $d$-recipe stabilization, $k^* = k$, and $\delta$ computed along the gradient of the distance to the domain boundary, yields linear convergence for any $k \geq 1$ (see Figure~\ref{fig:ball-voxels} for $k = 1,2$). Then, we consider families of nested meshes where boundary elements are replaced by polyhedral elements obtained as unions of cubes of refined meshes, such that $\delta$ is reduced by a factor of $1/2$ after each refinement step (see Figure~\ref{fig:ball3d-sequence}). Data for the family of meshes obtained by reducing $\delta$ by a factor of $1/16$ ($n=4$ refinement steps) are shown in Table~\ref{tab:ball-ref4}. We remark that, while the number of faces, edges and vertices may seem quite high, these are much lower of the number of faces, edges and vertices of the uniform cubic tessellation that we would need to get the same error by the finite element method with no boundary correction.
    	Moreover, the requirement $\tau \leq \tau_\alpha$ ensuring optimal convergence allows for sequences of meshes where the number of refinements (and therefore the number of faces) is uniformly bounded as $h$ decreases. Nevertheless, in three dimension, the relatively large number of faces results in large number of boundary degrees of freedom and in large dense blocks in the stiffness matrix, and further work is needed to address these issues. Some quite promising results are already available in two dimensions \cite{squadrettamento}, where a tailored static condensation technique can be applied that allows to drastically reduce the size of the dense blocks in the stiffness matrix, as well as the overall size of the linear system, that can be lowered down to a size comparable to the one it would have if no refinements was performed. Such a paper also contains a detailed analysis of the method, for meshes obtained as the union of elements of an underlying structured squared mesh, showing that bounds \eqref{inversebase} through \eqref{eq:approxdn}, as well as an optimal best approximation result for the VEM space hold under a much weaker assumption than Assumption \ref{shape_regular}.

    On these meshes, we attain optimal convergence also for $k = 2$ (see Figure~\ref{fig:ball-voxels}). By looking at the boundary element shown in Figure~\ref{fig:ball-ref4-bnd}, which belongs to mesh ball$^b_1$, we clearly see the advantage of combining the great flexibility of the VEM framework with our strategy for imposing boundary conditions.

\begin{figure}
    \centering
    \includegraphics[width=0.5\columnwidth]{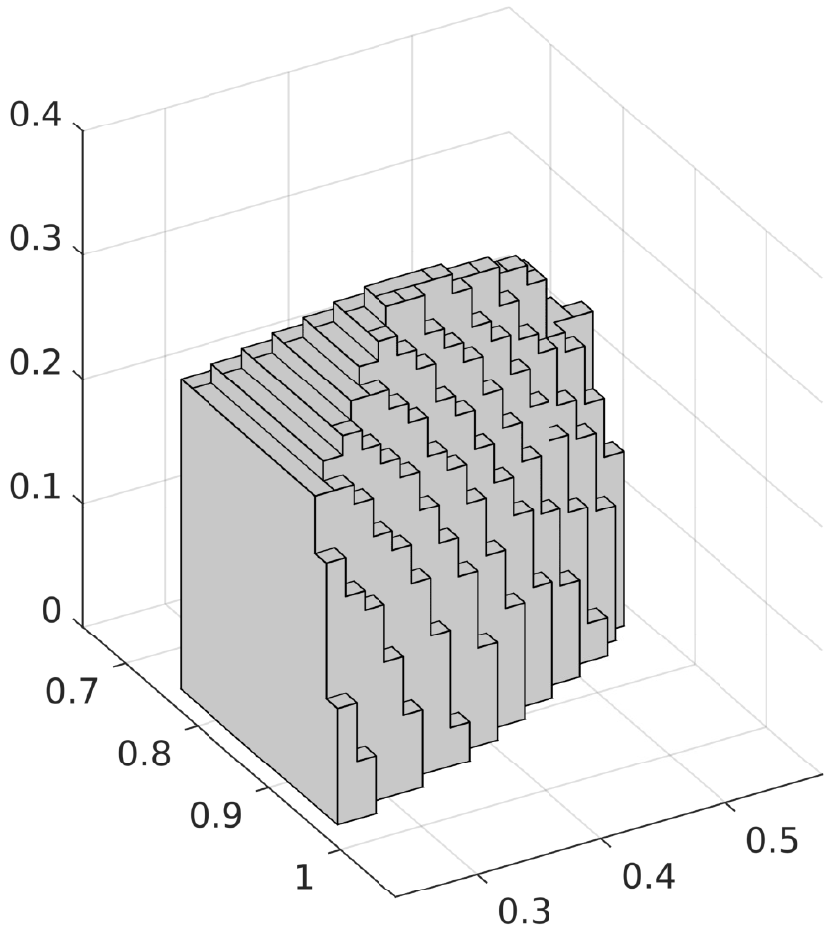}
    \caption{A boundary element in mesh ball$^b_1$.}
    \label{fig:ball-ref4-bnd}
\end{figure}

\subsubsection*{Domain 2} The second domain is the portion of the toroidal domain from \emph{Test~\testone} contained in the octant $x \geq 0, y \geq 0, z \geq 0$, which has non constant curvature. The right hand side $f$ and the boundary data $g$ are chosen as in \emph{Test~\testone}. As done before, we start with a family of cubical meshes (family $a$), and then generate meshes characterized by increasing levels of refinement of the boundary elements. Figure~\ref{fig:torus3d-sequence} shows an example of an initial mesh and the meshes obtained after different levels of refinement. Using Nitsche's method with $\gamma = 1000$, $d$-recipe stabilization, $k^* = k$, $\delta$ computed along the gradient of the distance to the domain boundary, we get optimal convergence rates for $k \leq 2$ on the family of meshes obtained after $5$ steps of refinement (family $b$) (see Figure~\ref{fig:torus-voxels}). Note the additional refinement step required in this case to restore convergence for $k = 2$. This suggests that different domains may require different levels of refinement.

\input{ball_voxels_results}
\input{torus_voxels_results}

%% file: torus_results.tex
\begin{figure}
  \begin{tabular}{rl}
    \begin{tikzpicture}[trim axis left]
      \begin{loglogaxis}
	[ mark size=4pt, grid=major, small,
	  xlabel={Average mesh size $\overline h$},
	  ylabel={$e_1^u$} ]
        \addplot[color=red,mark=x] coordinates {
          (0.561474,0.115251)
          (0.27202,0.0572143)
          (0.134824,0.0285515)
          (0.0671889,0.0142706)
        };
%        \addlegendentry{$k = 1$}
        \addplot[color=green,mark=+] coordinates {
          (0.561474,0.0168358)
          (0.27202,0.00389292)
          (0.134824,0.000878697)
          (0.0671889,0.000197784)
        };
%        \addlegendentry{$k = 2$}
        \addplot[color=blue,mark=o] coordinates {
          (0.561474,0.00172072)
          (0.27202,0.00020075)
          (0.134824,2.48085e-05)
          (0.0671889,3.07132e-06)
        };
%        \addlegendentry{$k = 3$}
        \addplot[color=orange,mark=square] coordinates {
          (0.561474,0.000418501)
          (0.27202,2.07201e-05)
          (0.134824,1.1275e-06)
          (0.0671889,6.56087e-08)
        };
%        \addlegendentry{$k = 4$}

        \addplot[thick,color=black,no markers] coordinates {
          (0.27202, 6.56087e-08)
          (0.561474, 6.56087e-08)
          (0.561474, 1.35422e-07)
          (0.27202, 6.56087e-08)
        };
        \node [anchor=west,font=\footnotesize] at (0.561474,1.35422e-07) {$1$};
        \addplot[thick,color=black,no markers] coordinates {
          (0.27202, 6.56087e-08)
          (0.561474, 6.56087e-08)
          (0.561474, 2.79524e-07)
          (0.27202, 6.56087e-08)
        };
        % \node [anchor=west,font=\footnotesize] at (0.561474,2.79524e-07) {$2$};
        \addplot[thick,color=black,no markers] coordinates {
          (0.27202, 6.56087e-08)
          (0.561474, 6.56087e-08)
          (0.561474, 5.76963e-07)
          (0.27202, 6.56087e-08)
        };
        % \node [anchor=west,font=\footnotesize] at (0.561474,5.76963e-07) {$3$};
        \addplot[thick,color=black,no markers] coordinates {
          (0.27202, 6.56087e-08)
          (0.561474, 6.56087e-08)
          (0.561474, 1.1909e-06)
          (0.27202, 6.56087e-08)
        };
        \node [anchor=west,font=\footnotesize] at (0.561474,1.1909e-06) {$4$};
      \end{loglogaxis}
    \end{tikzpicture}
    
    &
    
    \begin{tikzpicture}[trim axis right]
      \begin{loglogaxis}
	[ mark size=4pt, grid=major, small,
	  xlabel={Average mesh size $\overline h$},
	  ylabel={$e_0^u$} ]
        \addplot[color=red,mark=x] coordinates {
          (0.561474,0.00590172)
          (0.27202,0.00145595)
          (0.134824,0.000363026)
          (0.0671889,9.05957e-05)
        };
        % \addlegendentry{$k = 2$}
        \addplot[color=green,mark=+] coordinates {
          (0.561474,0.000759653)
          (0.27202,8.75969e-05)
          (0.134824,9.86367e-06)
          (0.0671889,1.07741e-06)
        };
        % \addlegendentry{$k = 3$}
        \addplot[color=blue,mark=o] coordinates {
          (0.561474,5.3669e-05)
          (0.27202,3.15751e-06)
          (0.134824,1.91987e-07)
          (0.0671889,1.21039e-08)
        };
        % \addlegendentry{$k = 4$}
        \addplot[color=orange,mark=square] coordinates {
          (0.561474,6.26592e-06)
          (0.27202,1.8413e-07)
          (0.134824,5.65031e-09)
          (0.0671889,1.75264e-10)
        };
        % \addlegendentry{$k = 5$}

        \addplot[thick,color=black,no markers] coordinates {
          (0.27202, 1.75264e-10)
          (0.561474, 1.75264e-10)
          (0.561474, 7.46707e-10)
          (0.27202, 1.75264e-10)
        };
        \node [anchor=west,font=\footnotesize] at (0.561474,7.46707e-10) {$2$};
        \addplot[thick,color=black,no markers] coordinates {
          (0.27202, 1.75264e-10)
          (0.561474, 1.75264e-10)
          (0.561474, 1.54127e-09)
          (0.27202, 1.75264e-10)
        };
        % \node [anchor=west,font=\footnotesize] at (0.561474,1.54127e-09) {$3$};
        \addplot[thick,color=black,no markers] coordinates {
          (0.27202, 1.75264e-10)
          (0.561474, 1.75264e-10)
          (0.561474, 3.18132e-09)
          (0.27202, 1.75264e-10)
        };
        % \node [anchor=west,font=\footnotesize] at (0.561474,3.18132e-09) {$4$};
        \addplot[thick,color=black,no markers] coordinates {
          (0.27202, 1.75264e-10)
          (0.561474, 1.75264e-10)
          (0.561474, 6.56654e-09)
          (0.27202, 1.75264e-10)
        };
        \node [anchor=west,font=\footnotesize] at (0.561474,6.56654e-09) {$5$};
      \end{loglogaxis}
    \end{tikzpicture}
  \end{tabular}
  
    \begin{tabular}{c}
    \begin{tikzpicture} 
      \begin{axis}[%
          hide axis,
          xmin=10,
          xmax=50,
          ymin=0,
          ymax=0.4,
          legend style={draw=white!15!black,legend cell align=left},
          legend columns=4
        ]
        
        \addlegendimage{color=red,mark=x}
        \addlegendentry{$k = 1$}
        
        \addlegendimage{color=green,mark=+}
        \addlegendentry{$k = 2$}
        
        \addlegendimage{color=blue,mark=o}
        \addlegendentry{$k = 3$}
        
        \addlegendimage{color=orange,mark=square}
        \addlegendentry{$k = 4$}
      \end{axis}
    \end{tikzpicture}
  \end{tabular}

  \caption{Convergence plots for the BH method with $k' = k-1$, $\alpha = 0.001$, $k^* = k$, $\delta$ computed along the outer normal, and $d$-recipe stabilization on the regular meshes of the toroidal domain.}
  \label{fig:torus_H1_L2}
\end{figure}

{\color{red}
\begin{table}
  \centering
  \caption{Comparison between the errors and the estimated convergence rates (ecr) obtained either with the $d$-recipe or the euclidean stabilization for $k = 4$.}
  \label{tab:torus-stabilizations}
  \begin{tabular}{
      c
      S[table-format=1.{\roundPrecision}e-1]
      S[table-format=1.2]
      S[table-format=1.{\roundPrecision}e-1]
      S[table-format=1.2]
      S[table-format=1.{\roundPrecision}e-1]
      S[table-format=1.2]
      S[table-format=1.{\roundPrecision}e-1]
      S[table-format=1.2]
    }
    \toprule
    \multirow{2}*{Mesh} & \multicolumn{4}{c}{$d$-recipe} & \multicolumn{4}{c}{euclidean}\\
    \cmidrule(lr){2-5} \cmidrule(lr){6-9}
    & {$e_1^u$} & {ecr} & {$e_0^u$} & {ecr} & {$e_1^u$} & {ecr} & {$e_0^u$} & {ecr}\\
    \midrule
    torus$_1$ & 4.185013e-04   &   {-}   &   6.265923e-06   &   {-} & 4.992212e-02   &   {-}   &   2.047646e-04   &   {-}\\
    torus$_2$ & 2.072005e-05   &   4.147400   &   1.841304e-07   &   4.867238 & 2.799140e-03   &   3.975704   &   2.472961e-06   &   6.094256\\
    torus$_3$ & 1.127496e-06   &   4.147436   &   5.650313e-09   &   4.963544 & 1.051680e-04   &   4.675151   &   4.035922e-08   &   5.863124\\
    torus$_4$ & 6.560872e-08   &   4.083548   &   1.752637e-10   &   4.986866 & 3.862017e-06   &   4.744492   &   8.923135e-10   &   5.473012\\
    \bottomrule
  \end{tabular}
\end{table}
}

	\begin{table}
  \centering
  \caption{Numerical estimates of the $1$-norm condition number of the global stiffness matrices obtained by solving Test \testone~ with Nitsche's method with $\gamma = 1000$ and $d$-recipe stabilization.}
  \label{tab:torus-condest-d-recipe}
  \begin{tabular}
    {
      c
      S[table-format=1.{\roundPrecision}e+1]
      S[table-format=1.{\roundPrecision}e+1]
      S[table-format=1.{\roundPrecision}e+1]
      S[table-format=1.{\roundPrecision}e+2]
    }
    \toprule
    \multicolumn{5}{c}{With BDT correction}\\
    \midrule
    {Mesh} & {$k = 1$} & {$k = 2$} & {$k = 3$} & {$k = 4$}\\
    \midrule
    torus$_1$ & 1.586304451148514e+03 & 2.789711832237715e+04 & 2.493787456605698e+06 & 8.630045027962251e+08\\
    torus$_2$ & 4.857920872559084e+03 & 6.970237581239808e+04 & 4.611372093143048e+06 & 4.587007326711392e+09\\
    torus$_3$ & 1.908920726954592e+04 & 2.334391219008937e+05 & 1.624468844395351e+07 & 3.530341785387154e+09\\
    torus$_4$ & 7.719503856552491e+04 & 9.143943094689633e+05 & 6.156586341248339e+07 & 3.583748998563371e+10\\
    \midrule
    \multicolumn{5}{c}{Without BDT correction}\\
    \midrule
    {Mesh} & {$k = 1$} & {$k = 2$} & {$k = 3$} & {$k = 4$}\\
    \midrule
    torus$_1$ & 1.137638582135073e+03 & 1.465560140736096e+04 & 9.696495829775213e+05 & 1.946078815578626e+08\\
    torus$_2$ & 4.402377401054355e+03 & 5.325964803120038e+04 & 3.715623047379809e+06 & 4.445026208477753e+09\\
    torus$_3$ & 1.832941971380572e+04 & 2.095957731764170e+05 & 1.528297074592439e+07 & 3.503322903716391e+09\\
    torus$_4$ & 7.564292980415387e+04 & 8.647849803302061e+05 & 5.982933255123231e+07 & 3.576979207717406e+10\\
    \bottomrule
  \end{tabular}
\end{table}

%% file: ball_voxels_results.tex
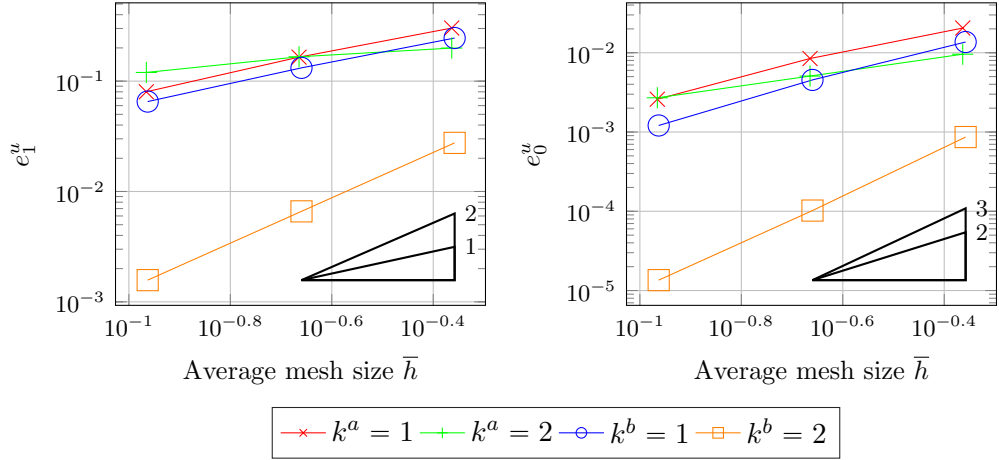
\begin{figure}
  \begin{tabular}{rl}
    
    \begin{tikzpicture}[trim axis left]
      \begin{loglogaxis}
	[ mark size=4pt, grid=major, small,
	  xlabel={Average mesh size $\overline h$},
	  ylabel={$e_1^u$} ]
        \addplot[color=red,mark=x] coordinates {
        (0.433013,0.30269)
        (0.216506,0.164983)
        (0.108253,0.0801115)
        };
        \addplot[color=green,mark=+] coordinates {
        (0.433013,0.199412)
        (0.216506,0.166182)
        (0.108253,0.119847)
        };
        \addplot[color=blue,mark=o] coordinates {
        (0.438745,0.245794)
        (0.218916,0.131876)
        (0.108915,0.0652921)
        };
        \addplot[color=orange,mark=square] coordinates {
        (0.438745,0.0275734)
        (0.218916,0.00660703)
        (0.108915,0.001573)
        };
        \addplot[thick,color=black,no markers] coordinates {
        (0.218916, 0.001573)
        (0.438745, 0.001573)
        (0.438745, 0.00315256)
        (0.218916, 0.001573)
        };
        \node [anchor=west,font=\footnotesize] at (0.438745,0.00315256) {$1$};
        \addplot[thick,color=black,no markers] coordinates {
        (0.218916, 0.001573)
        (0.438745, 0.001573)
        (0.438745, 0.00631827)
        (0.218916, 0.001573)
        };
        \node [anchor=west,font=\footnotesize] at (0.438745,0.00631827) {$2$};
      \end{loglogaxis}
    \end{tikzpicture}
    
    &
    
    \begin{tikzpicture}[trim axis right]
      \begin{loglogaxis}
	[ mark size=4pt, grid=major, small,
	  xlabel={Average mesh size $\overline h$},
	  ylabel={$e_0^u$} ]
        \addplot[color=red,mark=x] coordinates {
        (0.433013,0.0205021)
        (0.216506,0.00845468)
        (0.108253,0.00259289)
        };
        \addplot[color=green,mark=+] coordinates {
        (0.433013,0.00958811)
        (0.216506,0.0050781)
        (0.108253,0.00270206)
        };
        \addplot[color=blue,mark=o] coordinates {
        (0.438745,0.0136654)
        (0.218916,0.0045416)
        (0.108915,0.00120893)
        };
        \addplot[color=orange,mark=square] coordinates {
        (0.438745,0.000863814)
        (0.218916,0.000102061)
        (0.108915,1.35277e-05)
        };
        \addplot[thick,color=black,no markers] coordinates {
        (0.218916, 1.35277e-05)
        (0.438745, 1.35277e-05)
        (0.438745, 5.43367e-05)
        (0.218916, 1.35277e-05)
        };
        \node [anchor=west,font=\footnotesize] at (0.438745,5.43367e-05) {$2$};
        \addplot[thick,color=black,no markers] coordinates {
        (0.218916, 1.35277e-05)
        (0.438745, 1.35277e-05)
        (0.438745, 0.0001089)
        (0.218916, 1.35277e-05)
        };
        \node [anchor=west,font=\footnotesize] at (0.438745,0.0001089) {$3$};
      \end{loglogaxis}
    \end{tikzpicture}
  \end{tabular}
  
  \begin{tabular}{c}
    \begin{tikzpicture} 
      \begin{axis}[%
          hide axis,
          xmin=10,
          xmax=50,
          ymin=0,
          ymax=0.4,
          legend style={draw=white!15!black,legend cell align=left},
          legend columns=4
        ]
        
        \addlegendimage{color=red,mark=x}
        \addlegendentry{$k^a = 1$}
        
        \addlegendimage{color=green,mark=+}
        \addlegendentry{$k^a = 2$}
        
        \addlegendimage{color=blue,mark=o}
        \addlegendentry{$k^b = 1$}

        \addlegendimage{color=orange,mark=square}
        \addlegendentry{$k^b = 2$}
      \end{axis}
    \end{tikzpicture}
  \end{tabular}

  \caption{\emph{Test 2.1}: convergence plots for the Nitsche's method with $\gamma = 1000$, $k^* = k$, $\delta$ computed along the gradient of the distance to the domain boundary, on a family of cubical meshes (a) and a family of meshes with refined boundary elements (b).}
  \label{fig:ball-voxels}
\end{figure}

%% file: torus_voxels_results.tex
\begin{figure}
  \begin{tabular}{rl}
    
    \begin{tikzpicture}[trim axis left]
      \begin{loglogaxis}
	[ mark size=4pt, grid=major, small,
	  xlabel={Average mesh size $\overline h$},
	  ylabel={$e_1^u$} ]
        \addplot[color=red,mark=x] coordinates {
        (0.433013,0.0775678)
        (0.216506,0.0386963)
        (0.108253,0.0197905)
        };
        \addplot[color=green,mark=+] coordinates {
        (0.433013,0.0864002)
        (0.216506,0.072088)
        (0.108253,0.05876)
        };
        \addplot[color=blue,mark=o] coordinates {
        (0.445932,0.0847906)
        (0.218966,0.0410026)
        (0.109054,0.0199117)
        };
        \addplot[color=orange,mark=square] coordinates {
        (0.445932,0.00936089)
        (0.218966,0.00213967)
        (0.109054,0.000508412)
        };
        \addplot[thick,color=black,no markers] coordinates {
        (0.216506, 0.000508412)
        (0.433013, 0.000508412)
        (0.433013, 0.00101683)
        (0.216506, 0.000508412)
        };
        \node [anchor=west,font=\footnotesize] at (0.433013,0.00101683) {$1$};
        \addplot[thick,color=black,no markers] coordinates {
        (0.216506, 0.000508412)
        (0.433013, 0.000508412)
        (0.433013, 0.00203366)
        (0.216506, 0.000508412)
        };
        \node [anchor=west,font=\footnotesize] at (0.433013,0.00203366) {$2$};
      \end{loglogaxis}
    \end{tikzpicture}
    
    &
    
    \begin{tikzpicture}[trim axis right]
      \begin{loglogaxis}
	[ mark size=4pt, grid=major, small,
	  xlabel={Average mesh size $\overline h$},
	  ylabel={$e_0^u$} ]
        \addplot[color=red,mark=x] coordinates {
        (0.433013,0.00302434)
        (0.216506,0.000819949)
        (0.108253,0.000293527)
        (0.0541266,8.92746e-05)
        };
        \addplot[color=green,mark=+] coordinates {
        (0.433013,0.00404427)
        (0.216506,0.00239967)
        (0.108253,0.0015856)
        (0.0541266,0.00073825)
        };
        \addplot[color=blue,mark=o] coordinates {
        (0.445932,0.00401665)
        (0.218966,0.00133159)
        (0.109054,0.000368237)
        };
        \addplot[color=orange,mark=square] coordinates {
        (0.445932,0.000427538)
        (0.218966,4.5132e-05)
        (0.109054,4.98571e-06)
        };
        \addplot[thick,color=black,no markers] coordinates {
        (0.218966, 4.98571e-06)
        (0.445932, 4.98571e-06)
        (0.445932, 2.06781e-05)
        (0.218966, 4.98571e-06)
        };
        \node [anchor=west,font=\footnotesize] at (0.445932,2.06781e-05) {$2$};
        \addplot[thick,color=black,no markers] coordinates {
        (0.218966, 4.98571e-06)
        (0.445932, 4.98571e-06)
        (0.445932, 4.21117e-05)
        (0.218966, 4.98571e-06)
        };
        \node [anchor=west,font=\footnotesize] at (0.445932,4.21117e-05) {$3$};
      \end{loglogaxis}
    \end{tikzpicture}
  \end{tabular}
  
  \begin{tabular}{c}
    \begin{tikzpicture} 
      \begin{axis}[%
          hide axis,
          xmin=10,
          xmax=50,
          ymin=0,
          ymax=0.4,
          legend style={draw=white!15!black,legend cell align=left},
          legend columns=4
        ]
        
        \addlegendimage{color=red,mark=x}
        \addlegendentry{$k^a = 1$}
        
        \addlegendimage{color=green,mark=+}
        \addlegendentry{$k^a = 2$}
        
        \addlegendimage{color=blue,mark=o}
        \addlegendentry{$k^b = 1$}

        \addlegendimage{color=orange,mark=square}
        \addlegendentry{$k^b = 2$}
      \end{axis}
    \end{tikzpicture}
  \end{tabular}

  \caption{\emph{Test 2.2}: convergence plots for the Nitsche's method with $\gamma = 1000$, $k^* = k$, $\delta$ computed along the gradient of the distance to the domain boundary, on mesh families $a$ and $b$.}
  \label{fig:torus-voxels}
\end{figure}
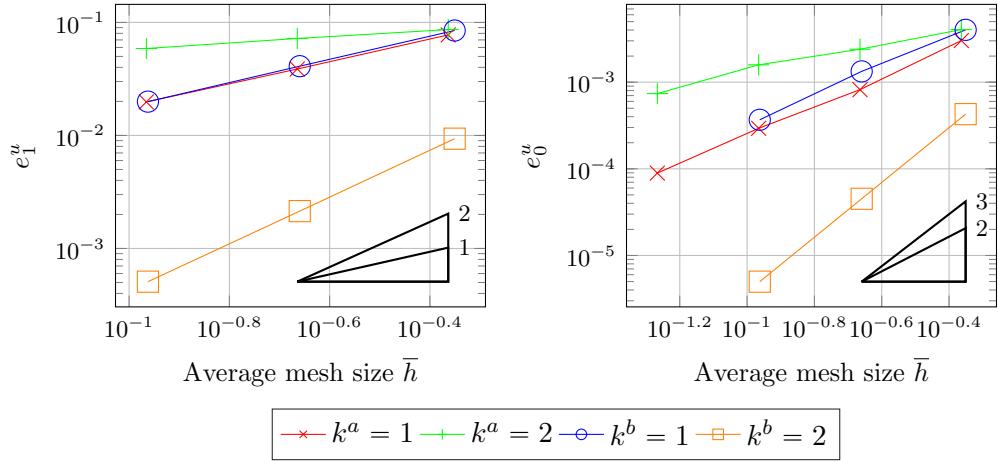

%% file: VEM_weaklyBC_rev.bbl
\providecommand{\bysame}{\leavevmode\hbox to3em{\hrulefill}\thinspace}
\providecommand{\MR}{\relax\ifhmode\unskip\space\fi MR }
\providecommand{\MRhref}[2]{%
  \href{http://www.ams.org/mathscinet-getitem?mr=#1}{#2}
}
\providecommand{\href}[2]{#2}
\begin{thebibliography}{10}

\bibitem{vorocrust}
A.~Abdelkader, C.L. Bajaj, M.S. Ebeida, A.H. Mahmoud, S.A. Mitchell, J.D.
  Owens, and A.A. Rushdi, \emph{Vorocrust: Voronoi meshing without clipping},
  ACM Trans. Graph. \textbf{39} (2020), no.~3.

\bibitem{3DVEM}
B.~Ahmad, A.~Alsaedi, F.~Brezzi, L.~D. Marini, and A.~Russo, \emph{Equivalent
  projectors for virtual element methods}, {Comput. Math. Appl.} \textbf{66}
  (2013), no.~3, 376--391.

\bibitem{curved_Trefftz}
A.~Anand, J.~S. Ovall, S.~E. Reynolds, and S.~Wei\ss{}er, \emph{Trefftz finite
  elements on curvilinear polygons}, SIAM Journal on Scientific Computing
  \textbf{42} (2020), no.~2, A1289--A1316.

\bibitem{Antonietti_VEM_Stokes}
P.~F. Antonietti, L.~{Beir\~ao} da~Veiga, D.~Mora, and M.~Verani, \emph{A
  stream virtual element formulation of the {Stokes} problem on polygonal
  meshes}, SIAM J. Numer. Anal. \textbf{52} (2014), no.~1, 386--404.

\bibitem{Antonietti_VEM_Cahn}
P.~F. Antonietti, L.~{Beir\~ao} da~Veiga, S.~Scacchi, and M.~Verani, \emph{A
  {$C^1$} virtual element method for the {Cahn-Hilliard} equation with
  polygonal meshes}, SIAM J. Numer. Anal. \textbf{54} (2016), no.~1, 34--56.

\bibitem{antonietti_p_VEM}
P.~F. Antonietti, L.~Mascotto, and M.~Verani, \emph{A multigrid algorithm for
  the $p$-version of the virtual element method}, ESAIM Math. Model. Numer.
  Anal. \textbf{52} (2018), no.~1, 337--364.

\bibitem{ABPV:minsurf}
P.F. Antonietti, S.~Bertoluzza, D.~Prada, and M.~Verani, \emph{The virtual
  element method for a minimal surface problem}, Calcolo \textbf{57} (2020),
  no.~4, 39.

\bibitem{de2016nonconforming}
B.~Ayuso~de Dios, K.~Lipnikov, and G.~Manzini, \emph{The nonconforming virtual
  element method}, ESAIM: Mathematical Modelling and Numerical Analysis
  \textbf{50} (2016), no.~3, 879--904.

\bibitem{babuska}
I.~Babu\v{s}ka, \emph{The finite element method with {L}agrangian multipliers},
  Numer.Math. \textbf{20} (1973), 179--192.

\bibitem{HughesBarbosa}
H.J.C. Barbosa and T.J.R. Hughes, \emph{The finite element method with
  {L}agrange multipliers on the boundary: circumventing the
  {B}abu{s}ka-{B}rezzi condition}, Comput. Methods Appl. Mech. En. \textbf{85}
  (1991), 109--128.

\bibitem{HughesBarbosa2}
\bysame, \emph{Boundary {L}agrange multipliers in finite element methods: Error
  analysis in natural norms}, Numer. Math. (1992), 1 -- 15.

\bibitem{BazilevHughes2007}
Y.~Bazilevs and T.~J.~R. Hughes, \emph{{Weak imposition of Dirichlet boundary
  conditions in fluid mechanics}}, {Computers \& Fluids} \textbf{{36}}
  ({2007}), no.~{1}, {12--26}.

\bibitem{BAZILEVS20074853}
Y.~Bazilevs, C.~Michler, V.M. Calo, and T.J.R. Hughes, \emph{Weak dirichlet
  boundary conditions for wall-bounded turbulent flows}, Computer Methods in
  Applied Mechanics and Engineering \textbf{196} (2007), no.~49, 4853--4862.

\bibitem{basicVEM}
L.~{Beir\~ao}~da Veiga, F.~Brezzi, A.~Cangiani, G.~Manzini, L.~D. Marini, and
  A.~Russo, \emph{Basic principles of virtual element methods}, Math. Models
  Methods Appl. Sci. \textbf{23} (2013), no.~1, 199--214.

\bibitem{beirao_linear_elasticity}
L.~{Beir\~ao}~da Veiga, F.~Brezzi, and L.~D. Marini, \emph{Virtual elements for
  linear elasticity problems}, SIAM J. Numer. Anal. \textbf{51} (2013), no.~2,
  794--812.

\bibitem{hitchVEM}
L.~{Beir\~ao}~da Veiga, F.~Brezzi, L.~D. Marini, and A.~Russo, \emph{The
  {hitchhiker}'s guide to the virtual element method}, Math. Models Methods
  Appl. Sci. \textbf{24} (2014), no.~8, 1541--1573.

\bibitem{VEM_mixed}
\bysame, \emph{Mixed virtual element methods for general second order elliptic
  problems on polygonal meshes}, ESAIM Math. Model. Numer. Anal. \textbf{50}
  (2016), no.~3, 727--747.

\bibitem{VEM-curved}
\bysame, \emph{{Polynomial preserving virtual elements with curved edges}},
  Math. Models Methods Appl. Sci. \textbf{{30}} ({2020}), no.~{8},
  {1555--1590}.

\bibitem{beirao_hp}
L.~{Beir\~ao}~da Veiga, A.~Chernov, L.~Mascotto, and A.~Russo, \emph{Basic
  principles of $hp$ virtual elements on quasiuniform meshes}, Math. Models
  Methods Appl. Sci. \textbf{26} (2016), no.~8, 1567--1598.

\bibitem{3DVEM2}
L.~{Beir\~ao}~da Veiga, F.~Dassi, and A.~Russo, \emph{High-order virtual
  element method on polyhedral meshes}, {Comput. Math. Appl.} \textbf{74}
  (2017), 1110--1122.

\bibitem{beirao_elastic}
L.~{Beir\~ao}~da Veiga, C.~Lovadina, and D.~Mora, \emph{A virtual element
  method for elastic and inelastic problems on polytope meshes}, Comput.
  Methods Appl. Mech. Engrg. \textbf{295} (2015), 327 -- 346.

\bibitem{beirao_stab}
L.~{Beir\~ao}~da Veiga, C.~Lovadina, and A.~Russo, \emph{Stability analysis for
  the virtual element method}, Math. Models Methods Appl. Sci. \textbf{27}
  (2017), no.~13, 2557--2594.

\bibitem{beirao_stokes}
L.~{Beir{\~a}o da Veiga}, C.~{Lovadina}, and G.~{Vacca}, \emph{Divergence free
  virtual elements for the {Stokes} problem on polygonal meshes}, ESAIM Math.
  Model. Numer. Anal. \textbf{51} (2017), no.~2, 509--535.

\bibitem{beirao_Navier_Stokes}
\bysame, \emph{Virtual elements for the {Navier-Stokes} problem on polygonal
  meshes}, SIAM J. Numer. Anal. \textbf{56} (2018), no.~3, 1210--1242.

\bibitem{VEM_curved_beirao}
L.~{Beir{\~a}o da Veiga}, A.~{Russo}, and G.~{Vacca}, \emph{The virtual element
  method with curved edges}, ESAIM Math. Model. Numer. Anal. \textbf{53}
  (2019), no.~2, 375--404.

\bibitem{VEM_discrete_fracture}
M.~F. Benedetto, S.~Berrone, and S.~Scial\'o, \emph{A globally conforming
  method for solving flow in discrete fracture networks using the virtual
  element method}, Finite Elem. Anal. Des. \textbf{109} (2016), 23 -- 36.

\bibitem{berrone_borio_17}
S.~Berrone and A.~Borio, \emph{Orthogonal polynomials in badly shaped polygonal
  elements for the virtual element method}, Finite Elem. Anal. Des.
  \textbf{129} (2017), 14--31.

\bibitem{BMPP_VEM_nonconforming_stab}
S.~Bertoluzza, G.~Manzini, M.~Pennacchio, and D.~Prada, \emph{Stabilization of
  the nonconforming virtual element method}, Comput. Math. with Appl.
  \textbf{116} (2022), 25--47.

\bibitem{squadrettamento}
S.~Bertoluzza, M.~Montardini, M.~Pennacchio, and D.~Prada, \emph{The virtual
  element method on image-based domain approximations}, Tech. report,
  arXiv:2206.03449, 2022.

\bibitem{FETI_VEM_2D}
S.~Bertoluzza, M.~Pennacchio, and D.~Prada, \emph{{BDDC} and {FETI-DP} for the
  virtual element method}, Calcolo \textbf{54} (2017), 1565--1593.

\bibitem{VEM_curvo}
\bysame, \emph{High order {VEM} on curved domains.}, Atti Accad. Naz. Lincei
  Rend. Lincei Mat. Appl. \textbf{30} (2019), 391--412.

\bibitem{FETI_VEM_3D}
\bysame, \emph{{FETI-DP} for the three dimensional virtual element method},
  SIAM Journal on Numerical Analysis \textbf{58} (2020), no.~3, 1556--1591.

\bibitem{BoffiBrezziFortin}
D.~Boffi, F.~Brezzi, and M.~Fortin, \emph{Mixed finite element methods and
  applications}, Springer Series in Computational Mathematics, Springer Berlin
  Heidelberg, 2013.

\bibitem{BDT}
J.~H. Bramble, T.~Dupont, and V.~Thom{\'e}e, \emph{Projection methods for
  {D}irichlet's problem in approximating polygonal domains with boundary-value
  corrections}, Math. Comp. \textbf{26} (1972), no.~120, 869--879.

\bibitem{brenner2018virtual}
S.C. Brenner and L.~Sung, \emph{Virtual element methods on meshes with small
  edges or faces}, In press. Mathematical Models and Methods in Applied
  Sciences \textbf{28} (2018), no.~07, 1291--1336.

\bibitem{brenner2021ac}
S.C. Brenner, L.~Sung, and Z.~Tan, \emph{A ${C}^1$ virtual element method for
  an elliptic distributed optimal control problem with pointwise state
  constraints}, Mathematical Models and Methods in Applied Sciences (2021),
  2887--2906.

\bibitem{burman2020dirichlet}
E.~Burman, P.~Hansbo, and M.G. Larson, \emph{Dirichlet boundary value
  correction using lagrange multipliers}, BIT Numerical Mathematics \textbf{60}
  (2020), no.~1, 235--260.

\bibitem{CALVO20191163}
J.G. Calvo, \emph{An overlapping schwarz method for virtual element
  discretizations in two dimensions}, Computers \& Mathematics with
  Applications \textbf{77} (2019), no.~4, 1163--1177.

\bibitem{Cangiani_book}
A.~Cangiani, Z.~Dong, E.H. Georgoulis, and P.~Houston, \emph{hp-version
  {Discontinuous Galerkin Methods on Polygonal and Polyhedral Meshes}},
  SpringerBriefs in Mathematics, Springer International Publishing, 2017.

\bibitem{dassi_mascotto_3DVEM}
F.~Dassi and L.~Mascotto, \emph{Exploring high-order three dimensional virtual
  elements: Bases and stabilizations}, Comput. Math. Appl. \textbf{75} (2018),
  no.~9, 3379 -- 3401.

\bibitem{Dupont1974}
T.~Dupont, \emph{L2 error estimates for projection methods for parabolic
  equations in approximating domains}, Mathematical Aspects of Finite Elements
  in Partial Differential Equations (Carl {de Boor}, ed.), Academic Press,
  1974, pp.~313--352.

\bibitem{VEM_Laplace_Beltrami}
M.~{Frittelli} and I.~{Sgura}, \emph{Virtual element method for the
  {Laplace-Beltrami} equation on surfaces}, ESAIM Math. Model. Numer. Anal.
  \textbf{52} (2018), no.~3, 965 -- 993.

\bibitem{VEM_3D_elasticity}
A.~L. Gain, C.~Talischi, and G.~H. Paulino, \emph{On the virtual element method
  for three-dimensional linear elasticity problems on arbitrary polyhedral
  meshes}, Comput. Methods Appl. Mech. Engrg. \textbf{282} (2014), 132--160.

\bibitem{hager1984}
William~W. Hager, \emph{Condition estimates}, SIAM Journal on Scientific and
  Statistical Computing \textbf{5} (1984), no.~2, 311--316.

\bibitem{Hansbo2005nitsche}
Peter Hansbo, \emph{Nitsche's method for interface problems in computational
  mechanics}, GAMM-Mitteilungen \textbf{28} (2005), no.~2, 183--206.

\bibitem{higham2000}
Nicholas~J. Higham and Fran\c{c}oise Tisseur, \emph{A block algorithm for
  matrix 1-norm estimation, with an application to 1-norm pseudospectra}, SIAM
  Journal on Matrix Analysis and Applications \textbf{21} (2000), no.~4,
  1185--1201.

\bibitem{Stenberg_Nitsche_2009}
M.~Juntunen and R.~Stenberg, \emph{Nitsche's method for general boundary
  conditions}, Math. Comp. \textbf{78} (2009), no.~267, 1353--1374.

\bibitem{beirao_hp_exponential}
L.~{Mascotto}, L.~{Beir{\~a}o da Veiga}, A.~{Chernov}, and A.~{Russo},
  \emph{Exponential convergence of the hp virtual element method with corner
  singularities}, Numer. Math. (2018), 138--581.

\bibitem{Nitsche}
J.A. Nitsche, \emph{{\"U}ber ein variationsprinzip zur {L}\"osung von
  {D}irichlet-{P}roblemen bei {V}erwendung von {T}eilr \"aumen, die keinen
  {R}andbedingungen unterworfen sind}, Abhandlungen aus dem Mathematischen
  Seminar der Universit\"at Hamburg \textbf{36} (1970), 9--15.

\bibitem{perugia_Helmholtz}
I.~Perugia, P.~Pietra, and A.~Russo, \emph{A plane wave virtual element method
  for the {Helmholtz} problem}, ESAIM Math. Model. Numer. Anal. \textbf{50}
  (2016), no.~3, 783--808.

\bibitem{FEMsurvey_curved}
R.~Sevilla, S.~Fern\'andez-M\'endez, and A.~Huerta, \emph{Comparison of
  high-order curved finite elements}, Internat. J. Numer. Methods Engrg.
  \textbf{87} (2011), no.~8, 719--734.

\bibitem{Stenberg.1995}
R.~Stenberg, \emph{On some techniques for approximating boundary conditions in
  the finite element method}, Journal of Computational and Applied Mathematics
  \textbf{63} (1995), 139--148.

\bibitem{strang1973change}
Gilbert Strang and Alan~E Berger, \emph{The change in solution due to change in
  domain}, Partial differential equations, 1973, pp.~199--205.

\bibitem{beirao_parab}
G.~Vacca and L.~{Beir\~ao}~da Veiga, \emph{Virtual element methods for
  parabolic problems on polygonal meshes}, Numer. Methods Partial Differential
  Equations \textbf{31} (2015), no.~6, 2110--2134.

\bibitem{van1995new}
B.~van Rietbergen, H.~Weinans, Rik. Huiskes, and A.~Odgaard, \emph{A new method
  to determine trabecular bone elastic properties and loading using
  micromechanical finite-element models}, Journal of biomechanics \textbf{28}
  (1995), no.~1, 69--81.

\bibitem{wriggers2017efficient}
P.~Wriggers, B.D. Reddy, W.T. Rust, and B.~Hudobivnik, \emph{Efficient virtual
  element formulations for compressible and incompressible finite
  deformations}, Computational Mechanics \textbf{60} (2017), no.~2, 253--268.

\bibitem{wriggers2016virtual}
P.~Wriggers, W.T. Rust, and B.D. Reddy, \emph{A virtual element method for
  contact}, Computational Mechanics \textbf{58} (2016), no.~6, 1039--1050.

\end{thebibliography}
